\documentclass[12pt,letterpaper,reqno]{amsart}
\usepackage{amsrefs}
\usepackage{amsthm}
\usepackage{amssymb}
\usepackage{mathrsfs}
\usepackage{amsmath}
\usepackage{latexsym}
\usepackage[CJKbookmarks=true]{hyperref}
\usepackage{graphicx,tikz}
\usepackage[all]{xy}
\usepackage{ytableau}
\usepackage{enumerate}

\usepackage{fancyhdr}

\numberwithin{equation}{section}
\newtheorem{Thm}{Theorem}

\newtheorem{Th}{Theorem}[section]
\newtheorem{Lemma}[Th]{Lemma}
\newtheorem{Coro}[Th]{Corollary}
\newtheorem{Prop}[Th]{Proposition}
\newtheorem{Eg}[Th]{Example}
\newtheorem{Rmk}[Th]{Remark}
\newtheorem{Conj}[Th]{Conjecture}
\newtheorem*{Claim}{Claim}
\allowdisplaybreaks

\begin{document}

\def\q{{q}}
\def\z{{z}}
\def\Q{{\operatorname{Q}}}
\def\Z{{\operatorname{Z}}}
\let\oldnabla\nabla
\def\nabla{\oldnabla\!}
\def\Fl{\operatorname{\mathcal{F}\ell}}
\def\Gr{\operatorname{\mathcal{G}\it r}}
\def\pt{\operatorname{\mathsf{pt}}}
\def\Fun{\operatorname{\mathsf{Fun}}}
\def\S{\mathfrak{S}}

\def\kar{\expandafter\ar@[*0.75]@{->}}
\def\kkar{\expandafter\ar@<+1pt>@{-->}}
\def\kkto{\mathop{{\dashrightarrow}}\limits}
\def\kto{\mathop{{\longrightarrow}}\limits}

\newdir{|>}{*:(1,-0.5)@^{>}*:(1,+0.5)@_{>}}
\def\kar{\expandafter\ar@<+1pt>@{{-}{-}{|>}}}

\title{Pieri and Murnaghan--Nakayama type Rules\\ for Chern classes of Schubert Cells}

\author{Neil J.Y. Fan}
\address[Neil J.Y. Fan]{Department of Mathematics, 
Sichuan University, Chengdu, Sichuan, P.R. China}
\email{fan@scu.edu.cn}

\author{Peter L. Guo}
\address[Peter L. Guo]{Center for Combinatorics, LPMC, 
Nankai University, Tianjin 300071, P.R. China}
\email{lguo@nankai.edu.cn}

\author[R.~Xiong]{Rui Xiong}
\address[Rui Xiong]{Department of Mathematics and Statistics, University of Ottawa, 150 Louis-Pasteur, Ottawa, ON, K1N 6N5, Canada}
\email{rxion043@uottawa.ca}

\date{}

\maketitle


\begin{abstract}
We develop Pieri type as well as  Murnaghan--Nakayama  type formulas 
for equivariant Chern--Schwartz--MacPherson classes
 of Schubert cells in the classical flag variety.  
These formulas include as special cases  many previously known
multiplication  formulas for Chern--Schwartz--MacPherson  classes or Schubert classes.  We apply the equivariant  Murnaghan--Nakayama   formula  to the enumeration  of   rim hook tableaux. 
\end{abstract}

\setcounter{tocdepth}{1}
\tableofcontents

\section{Introduction}

The Chern--Schwartz--MacPherson (CSM) classes,
constructed explicitly by MacPherson \cite{MacPherson} in order to resolve  a conjecture of Deligne and Grothendieck  \cite{Sullivan},
are one way to extend the Chern classes of complex manifolds to complex varieties (possibly singular or noncompact).
Their equivariant setting was developed by Ohmoto \cite{Ohmoto}. 
The structure of (equivariant) CSM classes for Schubert cells in flag varieties has received much attention in recent years, and  evolved into a  rich area of research, see for example \cite{AM0,AM,AMSS17, HpGr,Jones,Kumar, Lee, QGQC,Su2,Su3}. Such CSM classes also draw special interests from geometric representation theory due to their  intimate connections with characteristic cycles of Verma $\mathcal{D}$-modules over flag varieties, as well as to stable envelopes for the cotangent bundles of flag manifolds introduced  by Maulik and Okounkov \cite{QGQC}, see   \cite{AMSS17,RV18,Su}.

In this work, we shall focus on the 
torus equivariant CSM classes $c^T_{\mathrm{SM}}(Y(w)^{\circ})$ for Schubert 
cells $Y(w)^{\circ}$ in the classical flag variety $\Fl(n)$, where $w$ varies over permutations in the symmetric group $\S_n$.
A notable  feature is that the lowest degree component of  $c^T_{\mathrm{SM}}(Y(w)^{\circ})$ 
recovers the equivariant Schubert class $[Y(w)]_T$. 
Our goal is to establish Pieri type formulas as well as Murnaghan--Nakayama (MN) type formulas for $c^T_{\mathrm{SM}}(Y(w)^{\circ})$.

The Chevalley formula for $c^T_{\mathrm{SM}}(Y(w)^{\circ})$ multiplied by divisors was derived by combining the work of Aluffi, Mihalcea, Sch\"urmann and Su \cite{AMSS17}, and Su \cite{Su}.
When the divisor corresponds to the first Chern class of the tautological bundle over $\Fl(n)$, the formula may be described in terms of certain one step walks in the  Bruhat graph of $\S_n$, see Mihalcea, Naruse and Su \cite[Theorem 4.2]{MNS}. 
A natural and desirable extension of the Chevalley formula is then to develop Pieri or MN type formulas for $c^T_{\mathrm{SM}}(Y(w)^{\circ})$. Towards this direction,  we prove
\begin{itemize}
\item a Pieri formula for multiplying $c^T_{\mathrm{SM}}(Y(w)^{\circ})$ by the $r$-th Chern class  or Segre class  of the tautological bundle over  $\Fl(n)$ (more generally by a Schur polynomial of hook shape), see Theorem \ref{LLLL-1};   

\vspace{5pt}

\item a Pieri formula for multiplying $c^T_{\mathrm{SM}}(Y(w)^{\circ})$ by an equivariant Schubert class of a Grassmannian permutation associated to a one row/column partition (more generally to a hook shape), see Theorem \ref{LUUU-1};  

\vspace{5pt}

\item a MN formula for multiplying $c^T_{\mathrm{SM}}(Y(w)^{\circ})$ by the
Chern character of the tautological bundle { over  $\Fl(n)$ }(equivalently, by a power sum symmetric polynomial), see Theorem \ref{UUU-1}.
\end{itemize}

Although one could  compute the above multiplications according to the splitting principle by iterating  the Chevalley formula, there would be no explicit  formulas given for the structure constants because of the occurring of cancellations.
{Our Pieri formulas are formulated in terms of increasing and decreasing paths in the Bruhat graph of $\S_n$ with respect to a specific labeling on the edges, which manifestly imply  the positivity of the structure constants. 
We propose general positivity  conjectures in  Section \ref{CCC_I}, and discuss their relations to Kumar's recent   conjectures \cite{Kumar}. 
{In particular, we may apply our Pieri formula to
show that the CSM class  of the Richardson cell
is monomial-positive, proving  a weaker form of Kumar's conjecture \cite[Conjecture B]{Kumar} in type $A$, see Theorem \ref{MonomialPositive}.  }
}



Each of the above formulas vastly generalizes the Chevalley formula  for $c^T_{\mathrm{SM}}(Y(w)^{\circ})$. 
Both   Pieri formulas for $c^T_{\mathrm{SM}}(Y(w)^{\circ})$ 
specialize  to the Pieri formula for nonequivariant Schubert classes $[Y(w)]$   by Sottile \cite{Sottile1}, while
the Pieri formula in Theoerm \ref{LUUU-1} may reduce  to the Pieri formula for equivariant
Schubert classes $[Y(w)]_T$   by Robinson \cite{Robinson} (see also  Li, Ravikumar, Sottile and Yang \cite{LSY}).
The MN formula in Theorem \ref{UUU-1}
is a remarkable generalization from the  MN rule for nonequivariant Schubert classes  $[Y(w)]$ due to Morrison and Sottile \cite{MS2} to equivariant CSM classes    $c^T_{\mathrm{SM}}(Y(w)^{\circ})$. 
This in particular leads to  a MN rule for   equivariant Schubert classes, see Corollary \ref{LRrulehookshapeSchubert}.

The proofs of Theorems \ref{LLLL-1}, \ref{LUUU-1}, and \ref{UUU-1} are mainly combinatorial. In the proof of Theorem \ref{LLLL-1} or Theorem \ref{UUU-1}, 
one of the key ingredients we developed  is the
Rigidity Theorem  (see Theorem \ref{noneqimplieEqhook}, Theorem \ref{noneqimplieEqhookMN}), which means   that the structure constants in Theorem \ref{LLLL-1} or Theorem \ref{UUU-1}  can be
controlled by the structure constants in their  nonequivariant situations. So, to accomplish  Theorem \ref{LLLL-1} or Theorem \ref{UUU-1}, our strategy is to first establish their nonequivariant versions and then employ the Rigidity Theorem. {This philosophy has appeared for example in the study of  equivariant quantum cohomology  by Braverman,  Maulik and  Okounkov \cite{BMO}, and $K$-theoretic stable envelopes by Okounkov \cite[\textsection 2.4]{LectureK}.} 

{

Starting from the work of Ikeda and Naruse \cite{INaruse}, it was gradually realized that equivariant Chevalley type formulas could  be utilized  to deduce  hook formulas, see   Naruse  \cite{Narusee} and the recent work of  Morales, Pak and Panova \cite{Pak-4} and Mihalcea, Naruse and Su \cite{MNS}. 
We generalize this idea to higher degrees by viewing the equivariant MN formula as a higher degree analogue of the equivariant Chevalley formula. 
More concretely, we employ our equivariant MN formula to realize the number of standard  rim hook tableaux as a coefficient of the Laurant expansion relating to the localization of equivariant Schubert classes, see Theorem \ref{RimHooktableaux}. 
In this framework, we reveal the enumeration  formulas
for standard  rim hook tableaux due to Alexandersson, Pfannerer, Rubey  and  Uhlin \cite{APRU} and  Fomin and  Lulov \cite{FL} in   a relatively uniform manner. }

This paper is arranged as follows. Section \ref{UP} contains descriptions of the main theorems.  In Section \ref{Sec-2}, we give an overview of the background of CSM classes. In Section \ref{Secc-3}, we prove the Rigidity Theorem for Theorem \ref{LLLL-1}. The parallel idea is used
in Section \ref{Sect5} to establish the   Rigidity Theorem for Theorem  \ref{UUU-1}.
 In Section \ref{Sect4}, we finish the proofs of 
  Theorems \ref{LLLL-1} and  \ref{LUUU-1}. Section \ref{Sect5} is devoted to a proof of  Theorem \ref{UUU-1}. We investigate the Pieri as well as MN formulas over Grassmannians in Section \ref{GGG_I}. 
In Section \ref{SECT7}, we illustrate  how the MN formula over Grassmannians can be  applied to the enumeration of   rim hook tableaux. Some positivity conjectures  are discussed  in Section  \ref{CCC_I}.
  

\subsection*{Acknowledgement}
We are grateful to Changzheng Li, Leonardo  Mihalcea,  Changjian Su, Kirill Zainoulline, and Paul Zinn-Justin for valuable discussions and suggestions. This work was supported by the National Natural Science Foundation of China (11971250, 12071320).
R.X. acknowledges the partial support from the NSERC Discovery grant RGPIN-2015-04469, Canada.

\section{Main Results}\label{UP}

In this section, we give detailed descriptions of our main results: two Pieri formulas and a MN formula for equivariant CSM classes of Schubert 
cells   in the classical flag variety. 
Throughout this paper, let $n$ be a fixed positive integer, and  $k$ be an integer belonging to the set $[n]:=\{1,2,\ldots, n\}$. 

Let $\lambda=(\lambda_1,\lambda_2,\ldots, \lambda_k)$ be a  partition, namely, $\lambda_1\geq \lambda_2\geq \cdots\geq \lambda_k\geq 0$. We do not distinguish $\lambda$ with its Young diagram, a left-justified array with $\lambda_i$ boxes in row $i$.
Let $s_{\lambda}(x_1,\ldots,x_k)$ be the Schur polynomial associated to   $\lambda$, see Subsection \ref{BBBBB} for  definition. 
When $\lambda$ has exactly  one column (resp., one row) with $r$ boxes,  $s_{\lambda}(x_1,\ldots,x_k)$ is  the elementary symmetric polynomial $e_r(x_1,\ldots,x_k)$  (resp., the complete homogeneous symmetric polynomial $h_r(x_1,\ldots,x_k)$). Note that $e_r$ and  $h_r$ may serve as, up to a sign, representatives of the $r$-th Chern class and Segre class  of the $k$-th tautological bundle  over  $\Fl(n)$, respectively, see Fulton \cite[\textsection 14.6]{Fulton}. On the other hand, they are representatives of Schubert classes associated to the Grassmannian permutations corresponding respectively to the one column and one row partition.  These two ways of understanding lead to two different  generalizations of the classical  Pieri rule to equivariant  CSM classes, which are dealt with  in Theorem \ref{LLLL-1} and Theorem \ref{LUUU-1}, respectively. As mentioned in Introduction, 
our   formulas are  valid for the multiplication  by a Schur polynomial or an equivariant Schubert class associated to a partition of hook shape. 

Theorem \ref{LLLL-1} and Theorem \ref{LUUU-1} are closely related to paths in the {\it Bruhat graph}  on $\S_n$. 
As usual, write $t_{ab}$ $(1\leq a<b\leq n)$ for the transpositions in $\S_n$, and $\ell(w)$ for the length of $w\in \S_n$.
Note that $\ell(w)$ equals the number of inversion pairs (namely, $(w(a),w(b))$ with $1\leq a<b\leq n$ and $w(a)>w(b)$) of $w$.
The { Bruhat graph} on $\S_n$ is a directed graph whose vertices are permutations in $\S_n$ such that there is a directed edge from $u$ to $w$, denoted $u\longrightarrow w$, if $w=ut_{ab}$ for some  $t_{ab}$ and $\ell(w)\geq \ell(u)+1$ (equivalently, $u(a)<u(b)$) \cite{BB,Dyer}. For our purpose, we often label the edges  in the Bruhat graph by writing 
\[
u\stackrel{\tau}\longrightarrow w, \ \  \text{if} \  u\longrightarrow w, \  w=ut_{ab},\ \text{and} \ \tau=u(a).
\]   
An edge $u\stackrel{\tau}\longrightarrow w$ with $w=u t_{ab}$
is called a {\it $k$-edge} if $a\leq k<b$. The {\it $k$-Bruhat graph} on $\S_n$ 
is the subgraph induced by all $k$-edges. 
Figure \ref{fig:EgS3} illustrates the $1$-Bruhat graph and the $2$-Bruhat graph on $\S_3$, where we use dashed arrows to emphasize the edges $u\stackrel{\tau}\longrightarrow w$ with $\ell(w)>\ell(u)+1$. 

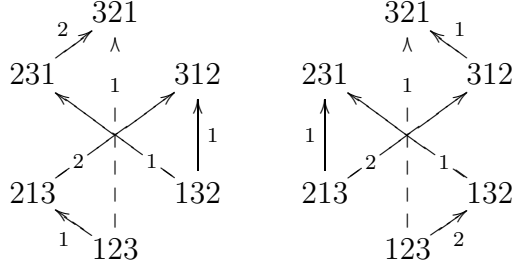
\begin{figure}[h]\label{fig:EgS3}\center
$$
\xymatrix@C=0.55pc{
\ar@{}[rr]|{\displaystyle\rule[-0.5pc]{0pc}{1.5pc}321}="321"
\phantom{123}&\phantom{123}&\phantom{123}\\
\ar@{}[u]|{\displaystyle\rule[-0.5pc]{0pc}{1.5pc}231}="231"
\ar@{}[d]|{\displaystyle\rule[-0.5pc]{0pc}{1.5pc}213}="213"&\rule{0pc}{2pc}&
\ar@{}[u]|{\displaystyle\rule[-0.5pc]{0pc}{1.5pc}312}="312"
\ar@{}[d]|{\displaystyle\rule[-0.5pc]{0pc}{1.5pc}132}="132"\\
\ar@{}[rr]|{\displaystyle\rule[-0.5pc]{0pc}{1.5pc}123}="123"&&
\ar"123";"213"^{1}
\ar"213";"312"|<<<<{\,\,2\,\,}
\ar"132";"231"|<<<<{\,\,1\,\,}
\ar"132";"312"_{1}
\ar"231";"321"^{2}
\ar@{-->}"123";"321"|>>>>>>{\rule[-0.5pc]{0pc}{1.5pc}1}
}\qquad
\xymatrix@C=0.55pc{
\ar@{}[rr]|{\displaystyle\rule[-0.5pc]{0pc}{1.5pc}321}="321"
\phantom{123}&\phantom{123}&\phantom{123}\\
\ar@{}[u]|{\displaystyle\rule[-0.5pc]{0pc}{1.5pc}231}="231"
\ar@{}[d]|{\displaystyle\rule[-0.5pc]{0pc}{1.5pc}213}="213"&\rule{0pc}{2pc}&
\ar@{}[u]|{\displaystyle\rule[-0.5pc]{0pc}{1.5pc}312}="312"
\ar@{}[d]|{\displaystyle\rule[-0.5pc]{0pc}{1.5pc}132}="132"\\
\ar@{}[rr]|{\displaystyle\rule[-0.5pc]{0pc}{1.5pc}123}="123"&&
\ar"123";"132"_{2}
\ar"213";"312"|<<<<{\,\,2\,\,}
\ar"132";"231"|<<<<{\,\,1\,\,}
\ar"213";"231"^{1}
\ar"312";"321"_{1}
\ar@{-->}"123";"321"|>>>>>>{\rule[-0.5pc]{0pc}{1.5pc}1}
}$$
\caption{The $1$-Bruhat graph and the $2$-Bruhat graph on $\S_3$}
\end{figure}

A path $\gamma$ of length $m$ from $u$ to $w$ means a sequence of edges 
\[\gamma\colon\ \  u
\stackrel{\tau_1}\longrightarrow w_1
\stackrel{\tau_2}\longrightarrow w_2\stackrel{\tau_3}\longrightarrow \cdots \stackrel{\tau_m}\longrightarrow w_m=w.\]
We say that $\gamma$ is {\it increasing} (resp., {\it decreasing}) if $\tau_1<\tau_2<\cdots<\tau_m$ (resp., $\tau_1>\tau_2>\cdots>\tau_m$), and   is {\it peakless}    if there exists $1\leq i\leq m$
such that 
$\tau_1>\cdots>\tau_i<\cdots<\tau_m$.
We use $\mathrm{in}(\gamma)=m-i$   (resp., $\mathrm{de}(\gamma)=i-1$) to denote one less than the length of the increasing  (resp., decreasing) segment of $\gamma$. Clearly, a peakless path is increasing (resp., decreasing)  if $\mathrm{de}(\gamma)=0$ 
  (resp., $\mathrm{in}(\gamma)=0$).

Define the {\it extended $k$-Bruhat order}, denoted $\leq_k$, on $\S_n$ as the order generated by all $k$-edges, that is, $u\leq_k w$ whenever there is a path in the $k$-Bruhat graph from $u$ to $w$. 
If restricting the $k$-edges to the edges in the Hasse diagram of the Bruhat order (namely, edges $u\longrightarrow w$ with $\ell(w)=\ell(u)+1$), then the extended $k$-Bruhat order reduces to the well-studied  ordinary $k$-Bruhat order \cite{LRS, BS, Sottile1}.

For $u, w\in \S_n$ and a subset $A\subseteq [n]$,  we adopt the following notation 
\begin{align}
uA &:=\{u(i)\colon i\in A\},\nonumber\\[3pt]
\Delta_A(u, w) &:=\{u(i)\colon i\in A\}  \setminus \{u(i)\colon u(i)\neq w(i)\} ,\label{deltaa}\\[3pt]
\Sigma_A(u, w) & :=\{u(i)\colon i\in A\}  \cup \{u(i)\colon u(i)\neq w(i)\}. 
\label{sigmaa}
\end{align}
When $A=[k]$, we  simply use $\Sigma_k(u, w)$ and 
$\Delta_k(u, w)$  to represent $\Sigma_{[k]}(u, w)$ and 
$\Delta_{[k]}(u, w)$, respectively. 
Moreover, assuming that $A=\{a_1<a_2<\cdots<a_m\}$ and $f(x_1,\ldots, x_m)$ is any given polynomial, we denote by $f(x_A)$   the polynomial obtained
from $f(x_1,\ldots, x_m)$ by substituting $x_i$ with $x_{a_i}$ for $1\leq i\leq m$. With this notation, the Schur polynomial $s_{\lambda}(x_1,\ldots,x_k)$  
can be briefly written as $s_{\lambda}(x_{[k]})$.

A hook shape partition  with arm length $\alpha$ and leg length $\beta$ will be denoted $\Gamma=(\alpha+1,1^\beta)$, that is,   $\Gamma$ has  one row with $\alpha+1$ boxes and one column with $\beta+1$ boxes. 

\begin{Thm}[see Theorem \ref{eqPierirulehook}]\label{LLLL-1}
Let $u\in \S_n$, and   $\Gamma=(1+\alpha,1^{\beta})$ be a hook shape. Then we have  
\[
c^T_{\mathrm{SM}}(Y(u)^\circ)\cdot s_{\Gamma}(x_{[k]}) 
=s_{\Gamma}( t_{u[k]})\cdot c^T_{\mathrm{SM}}(Y(u)^\circ)+
\sum_{u<_k w\in \S_n} c_{u,\Gamma}^w( t)\cdot c_{\mathrm{SM}}^T(Y(w)^\circ),
\]
where 
\begin{equation*}\label{eq:equivhook-12}
c_{u,\Gamma}^w( t)=
\sum_{\gamma}h_{\alpha-\mathrm{in}(\gamma)}( t_{\Sigma_k(u, w)})\cdot e_{\beta-\mathrm{de}(\gamma)}( t_{\Delta_k(u, w)})
\end{equation*}
with the sum taken over all peakless paths  from $u$ to $w$ in the extended $k$-Bruhat order.  
\end{Thm}

Letting $\alpha=0$ or $\beta=0$,
Theorem \ref{LLLL-1} reduces to a Pieri formula for equivariant CSM classes, see Corollary \ref{Pieri-CSM-1}. 
Particularly, if $\alpha$ and $\beta$ are both  $0$,
then Theorem \ref{LLLL-1} specifies  to  the Chevalley formula
\cite[Theorem 4.2]{MNS}. On the other hand, if taking the lowest degree component of $c^T_{\mathrm{SM}}(Y(u)^\circ)$, then Theorem \ref{LLLL-1} becomes an expansion for equivariant Schubert  classes, see Corollary \ref{EqSchubertPieri}.

Theorem \ref{LUUU-1}  replaces the Schur polynomial of a hook shape in Theorem \ref{LLLL-1} by an equivariant Schubert class of a Grassmannian permutation of hook shape. 
For a partition $\lambda=(\lambda_1,\ldots, \lambda_k)$
inside the $k\times (n-k)$ rectangle, let $w_{\lambda}$ denote the Grassmannian permutation in $\S_n$ associated to   $\lambda$ with descent at position $k$. That is, the values  of $w_\lambda$ at the first $k$ positions are determined by 
$\lambda_i=w(k-i+1)-(k-i+1)$ for $1\leq i\leq k$. For example, $w_\lambda=13524$ for the partition  $\lambda=(2,1,0)$. 


\begin{Thm}[see Theorem \ref{LRruleforhookshapeCSM}]\label{LUUU-1}
Let $u\in \S_n$, and let $\Gamma=(1+\alpha,1^{\beta})$ be a hook shape inside the $k\times (n-k)$ rectangle.  Then we have
\[
c_{\mathrm{SM}}^T(Y(u)^\circ)\cdot [Y(w_{\Gamma})]_T 
= [Y(w_{\Gamma})]_T|_{u}\cdot c_{\mathrm{SM}}^T(Y(u)^\circ)
+ \sum_{u<_k w\in \S_n} \mathfrak{c}_{u,\Gamma}^w( t)\cdot c_{\mathrm{SM}}^T(Y(w)^\circ),
\]
where 
\begin{equation*}\label{eq:LRruleforhookshapeCSM-II}
\mathfrak{c}_{u,\Gamma}^w( t) = \sum_{\gamma}\sum_{
\begin{subarray}{c}
\alpha_1+\alpha_2=\alpha-\mathrm{in}(\gamma)\\
\beta_1+\beta_2=\beta-\mathrm{de}(\gamma)
\end{subarray}} 
(-1)^{\alpha_2+\beta_2}\cdot
h_{\alpha_1}( t_{\Sigma_k(u, w)})\cdot e_{\beta_1}( t_{\Delta_k(u, w)})\cdot
e_{\alpha_2}( t_{[k+\alpha]})\cdot h_{\beta_2}( t_{[k-\beta]})
\end{equation*}
with the first  sum over all  peakless paths  from $u$ to $w$ in the extended $k$-Bruhat order. 
Here,  $[Y(w_{\Gamma})]_T|_{u}$ means the localization of $[Y(w_{\Gamma})]_T$ at $u$, as will be defined  in Section \ref{Sec-2}.

\end{Thm}

Setting $\alpha=0$ or $\beta=0$, we obtain our second Pieri formula for equivariant CSM classes, and in this case the expression 
for $\mathfrak{c}_{u,\Gamma}^w( t)$ could be dramatically  simplified to the localization of Schubert classes, see Theorem \ref{LRruleforehCSM}. 
Taking the lowest degree part in Theorem \ref{LUUU-1} recovers the Pieri formula for equivariant Schubert classes obtained in \cite{Robinson} using an algebraic approach (see \cite{LSY} for a geometric proof), see Remark \ref{REM-I}.

The classical MN formula computes the product of  a Schur polynomial by a power sum symmetric polynomial, which arises naturally in the computation of irreducible characters of $\S_n$, see Sagan \cite{Sagan} or Stanley \cite[Chapter 7]{stanley2}. 
Its nonequivariant Schubert generalization was found by Morrison and Sottile \cite{MS2}. 
Geometrically, a power sum symmetric polynomial can be viewed as a component of Chern characters of tautological bundles up to a scalar, see Subsection \ref{BBBBB}. 

We vastly lift the MN  formula of Morrison and Sottile from nonequivariant Schubert classes to equivariant CSM classes. 

\begin{Thm}[see Theorem \ref{MNforequivariantCSMSchur}]\label{UUU-1}
Let $u\in \S_n$. For $r\geq 1$,  
\[
c_{\mathrm{SM}}^T(Y(u)^\circ)\cdot p_{r}(x_{[k]})
=p_{r}( t_{u[k]})\cdot c_{\mathrm{SM}}^T(Y(u)^\circ)
+ \sum_{\eta\in \S_n} d_{u,r}^{u\eta}( t)\cdot  c_{\mathrm{SM}}^T(Y(u\eta)^\circ),
\]
where the sum runs over $(r'+1)$-cycles $\eta\in \S_n$
with $1\leq r'\leq r$ such that $u\leq_k u\eta$ in the extended 
$k$-Bruhat order, and
\[
d_{u,r}^{u\eta}( t) = (-1)^{\operatorname{ht}_k(\eta)}\cdot h_{r-r'}( t_{uM(\eta)}).
\]
Here,    $M(\eta)= \{1\leq i\leq n\colon \eta(i)\neq i\}$ is  the set of 
non-fixed points of $\eta$, and 
$
\operatorname{ht}_k(\eta)= \texttt{\#}\{i\leq k\colon   i\in M(\eta)\}-1.
$
\end{Thm}

Taking the lowest degree component in Theorem \ref{UUU-1} gives a MN formula for equivariant Schubert classes, see Corollary \ref{LRrulehookshapeSchubert}, which will be applied in  Section \ref{SECT7} to the enumeration of standard rim hook tableaux. If further specializing all   $t_i=0$, then we are led to the MN formula for Schubert classes    by Morrison and Sottile \cite{MS2}.

\section{Chern--Schwartz--MacPherson classes}\label{Sec-2}
 
 In this section, we shall briefly review the   geometric background   of Chern--Schwartz--MacPherson classes of Schubert cells in flag varieties. Some 
 properties required in this paper are also included.

\subsection{Chern--Schwartz--MacPherson Classes}

For a complex variety $X$,  denote by $\Fun(X)$ the space of {\it constructible functions} over $X$ with coefficients in $\mathbb{Q}$, namely,  the space of functions $f\colon X\to \mathbb{Q}$ which can be written as a finite sum 
\[
\sum c_W\mathbf{1}_W,
\]
where $c_W\in \mathbb{Q}$ and $\mathbf{1}_W$ is the characteristic function of a constructible subset $W$ of $X$.
It is well known that  any proper morphism $f\colon Y\to X$ induces a functorial pushforward $f_*\colon \Fun(Y)\to \Fun(X)$ given by 
\[
f_*\varphi(x) = \sum_{a\in \mathbb{Q}}\chi_c\big(f^{-1}(x)\cap \varphi^{-1}(a)\big),
\]
where $\chi_c$ is the Euler characteristic with compact  supports. 
It was conjectured by Deligne and Grothendieck \cite{Sullivan} and proved by MacPherson \cite{MacPherson} that there exists a linear map
\[
c_{\mathrm{SM}}\colon \Fun(X) \to H_*(X)
\]
from $\Fun(X)$ to the Borel--Moore homology of $X$ which is functorial in $X$ (with respect to proper pushforward). 
That is, for any proper morphism $f\colon Y\to X$,   
the following diagram commutes 
\[
\xymatrix{
\Fun(Y)\ar[d]_{f_*}\ar[r]^{c_{\mathrm{SM}}}& H_*(Y)\ar[d]^{f_*}\\
\Fun(X)\ar[r]^{c_{\mathrm{SM}}}& H_*(X),}
\]
where the right $f_*$ is the usual proper pushforward of Borel--Moore homology. 
For any constructible subset $W\subseteq X$,  define the {\it Chern--Schwartz--MacPherson class} (CSM class for abbreviation) to be 
\[
c_{\mathrm{SM}}(W) = c_{\mathrm{SM}}(\mathbf{1}_W). 
\]
Actually, the CSM class is the unique natural transformation characterized by the following property. For $X$ smooth,
\[
c_{\mathrm{SM}}(X) =  c_1(\mathscr{T}_X)\frown [X] \in H_\bullet(X),
\]
where  $\mathscr{T}_X$ is the tangent bundle of $X$,  {$c_1(\mathscr{T}_X)$ denotes the first  Chern class of $\mathscr{T}_X$}, the symbol $\frown$ stands for the cap product, and $[X]$ is the fundamental class of $X$.

The equivariant setting of  CSM classes was developed  by Ohmoto \cite{Ohmoto}. Let $T$ be a torus, and $X$ be a $T$-variety. 
For any $T$-invariant constructible subset $W\subseteq X$, denote by $c_{\mathrm{SM}}^T(W)$   its corresponding equivariant CSM class. 

In this paper, we are only concerned with the case when $X$ is smooth so that the (equivariant) cohomology can be naturally identified with (equivariant) Borel--Moore homology of $X$ under Poincar\'e duality. { Hence one may think of $c_{\mathrm{SM}}(W)$ and $c_{\mathrm{SM}}^T(W)$   as cohomology classes in $H_T^\bullet(X)$.}

\subsection{Flag Varieties}\label{AAAC}

The  {\it flag variety} $\Fl(n)$ is the variety of (complete) flags of $\mathbb{C}^n$, namely,  chains of subspaces  
\[
0=V_0\subseteq V_1\subseteq \cdots \subseteq V_n=\mathbb{C}^n
\]
with $\dim V_i=i$ for  $0\leq i\leq n$. 
For   $1\leq i\leq n$, define the {\it $i$-th tautological bundle} $\mathcal{V}_i$ over $\Fl(n)$ to be the vector bundle whose fiber at a flag $(V_\bullet)\in \Fl(n)$ is $V_i$. 

Let $T$ be the subgroup of diagonal matrices of $GL_n$. 
There is a natural action of $T$ on $\Fl(n)$. 
For $1\leq i\leq n$,  denote by $\pt=\operatorname{Spec} \mathbb{C}$   a point with trivial $T$-action. 
Let 
\[
 t_i = c_1(\mathbb{C}_{-t_i}) \in H^2_T(\pt)
\]
be the first Chern class of the equivariant bundle $\mathbb{C}_{-t_i}$, {the one-dimensional representation corresponding to the character sending $\operatorname{diag}(t_1,\ldots,t_n)\in T$ to $t_i^{-1}$. }
By  Borel \cite{Borel}, there is an isomorphism
\[
H_T^\bullet(\pt)\cong \mathbb{Q}[ t_1,\ldots, t_n].
\]
Note that $H_T^\bullet(\Fl(n))$ is a module over $H_T^\bullet(\pt)$.
For $1\leq i\leq n$, let 
\[
x_i=c_1^T((\mathcal{V}_i/\mathcal{V}_{i-1})^\vee)\in H_T^2(\Fl(n))
\]
be the first equivariant Chern class of  the dual bundle of the quotient bundle $\mathcal{V}_{i}/\mathcal{V}_{i-1}$. 
%
The following seminal  result is due to Borel \cite{Borel}, see also Anderson and Fulton \cite[\textsection 10.6]{AF}. 

\begin{Th}[Borel \cite{Borel}]\label{Boreliso}
We have 
\begin{equation*}\label{eq:Borelisoeq}
H_T^\bullet(\Fl(n)) \cong \frac{\mathbb{Q}[x_1,\ldots,x_n, t_1,\ldots, t_n]}{\left<f(x)-f( t): f\in \Lambda_n\right>}
\end{equation*}
where $\Lambda_n$ is the ring of symmetric polynomials in $n$ variables. Moreover,  
\begin{equation*}\label{eq:Borelisononeq}
H^\bullet(\Fl(n)) \cong \frac{\mathbb{Q}[x_1,\ldots,x_n]}{\left<f(x)-f(0): f\in \Lambda_n\right>}
\end{equation*}
with the forgetful map $\epsilon: H_T^\bullet(\Fl(n))\to H^\bullet(\Fl(n))$ given by 
\begin{equation*}\label{eq:augmentmap}
\epsilon \big(f(x, t)\big)= f(x,0). 
\end{equation*}
\end{Th}

The $T$-fixed points of $\Fl(n)$ are in bijection with the symmetric group $\S_n$. Precisely, for   $w\in \S_n$, the flag \[\phi_w=((\phi_w)_1,\ldots,(\phi_w)_n)\] is fixed by $T$, where 
$
(\phi_w)_i = \operatorname{span}(\mathbf{e}_{w(1)},\ldots,\mathbf{e}_{w(i)})
$
and $\mathbf{e}_i$'s are the standard basis of $\mathbb{C}^n$. 
Denote the {\it localization map} at $w$ by
\begin{equation*}
-|_{w}: H_T^\bullet(\Fl(n)) \to  H_T^\bullet(\phi_w)
\cong \mathbb{Q}[ t_1,\ldots, t_n].
\end{equation*}
Under the isomorphism in Theorem \ref{Boreliso}, it is easy to see  that
\begin{equation}
f(x, t)|_{w} = f(w t, t), 
\end{equation}
where $wt$ means $(t_{w(1)}, t_{w(2)}, \ldots, t_{w(n)})$ (so $f(w t, t)$ is obtained from $f(x, t)$ by replacing $x_i$ with $t_{w(i)}$).

For   $w\in\S_n$, define the  {\it Schubert cell} $Y(w)^\circ$ to be the constructible subset of flags $V_\bullet$ such that for   $0\leq i,j\leq n$,
\[
\dim(V_i\cap V^0_j)=\texttt{\#}\{(a,b): a\leq i,\,b\leq j,\,w(a)+b=n+1\},
\]
where $V^0_\bullet$ is the reversed standard flag, i.e., for  $1\leq i\leq n$, 
\[
V^0_i = \operatorname{span}(\mathbf{e}_{n},\ldots,\mathbf{e}_{n-i+1}).
\]
Notice that $Y(w)^\circ$ is $T$-invariant and $\phi_w$ is the unique $T$-fixed point over it.
In this paper, we shall concentrate on  the (equivariant) CSM classes of Schubert cells 
\begin{align*}
c_{\mathrm{SM}}^T(Y(w)^\circ) \in H_T^\bullet(\Fl(n)),\ \ \ \ \ 
c_{\mathrm{SM}}(Y(w)^\circ)  = \epsilon\big(c_{\mathrm{SM}}^T(Y(w)^\circ)\big) \in H^\bullet(\Fl(n)), 
\end{align*}
and  simply call them (equivariant) CSM classes. 
The {\it Schubert variety} $Y(w)$ is the closure of $Y(w)^\circ$, and   is a closed $T$-subvariety of $\Fl(n)$. So one can define the  {\it  (equivariant) Schubert class} to be
its fundamental class
\begin{align*}
[Y(w)]_T\in H_T^\bullet(\Fl(n)), \ \ \ \ \ 
[Y(w)] = \epsilon\big([Y(w)]_T\big)\in H^\bullet(\Fl(n)).
\end{align*}

\subsection{Characteristic Classes}\label{BBBBB}
 
For any $T$-equivariant vector bundle $\mathcal{V}$ of rank $k$ over a $T$-variety $X$, there is a  homomorphism from the ring $\Lambda_k$ of symmetric polynomials in $k$ variables to $H_T^\bullet(X)$,
defined by sending the $r$-th elementary symmetric polynomial 
\[e_r(x_{[k]})=\sum_{1\leq i_1<i_2<\cdots<i_r\leq k}x_{i_1}x_{i_2}\cdots x_{i_r}\]
to  the $r$-th $T$-equivariant Chern class $c^T_r(\mathcal{V})$.
Denote by $\mathcal{V}^\vee$   the dual of $\mathcal{V}$. 
For $r\geq 1$, the $r$-th $T$-equivariant Segre class of $\mathcal{V}^\vee$ is  the image of the {$r$-th complete homogeneous symmetric polynomial} 
\[
h_r(x_{[k]})=\sum_{1\leq i_1\leq i_2\leq \cdots\leq i_r\leq k}x_{i_1}x_{i_2}\cdots x_{i_r},
\]  
\cite{AF,Fulton}. After suitable completion, we have a well-defined $T$-equivariant Chern character
\begin{equation*}
\operatorname{ch}^T(\mathcal{V}) = k + \frac{1}{1!}p_1+\frac{1}{2!}p_2+\frac{1}{3!}p_3+\cdots,
\end{equation*}
where $p_r$ is the image of the  $r$-th power  sum   symmetric polynomial  
\[p_r(x_{[k]})=x_1^r+x_2^r+\cdots+x_k^r.\]

Recall that $\mathcal{V}_k$ is the $k$-th tautological bundle  over $\Fl(n)$. Denote its dual bundle by $\mathcal{V}_k^\vee$. 
By the defining properties of Chern classes, we have 
\begin{equation*}
c^T_r(\mathcal{V}_k^\vee) = e_r(x_{[k]}).
\end{equation*}
In particular, $h_r(x_{[k]})$ is the $r$-th $T$-equivariant Segre class of $\mathcal{V}_k$, and  $p_r(x_{[k]})$ appears as a summand of $T$-equivariant Chern character up to a constant. 


For a partition $\lambda=(\lambda_1,\ldots,\lambda_k)$, the associated  {\it Schur polynomial} $s_{\lambda}(x_{[k]})$ is defined as 
\[
s_{\lambda}(x_{[k]}) = \frac{\det\big(x_i^{\lambda_j+j-1}\big)_{1\leq i,j\leq k}}{\prod_{1\leq i<j\leq k}(x_i-x_j)}.
\]
Schur polynomials can be interpreted as the classes of degenerate loci of generic rational sections of $\mathcal{V}$, see \cite{Manivel,KL} and \cite[Chapter 14]{Fulton}.
When  $\lambda$ has exactly one row or one column with $r$ boxes, there hold the following relations 
\[ h_r(x_{[k]})=s_{(r)}(x_{[k]}) \ \ \ \ \text{and}\ \ \ \ 
e_r(x_{[k]})=s_{(1^r)}(x_{[k]}).
\]

\subsection{Schubert Classes and Schubert Polynomials}

The (equivariant) Schubert classes admit a remarkable choice of  polynomial representatives called  {\it (double) Schubert polynomials} as introduced by Lascoux and  Sch\"utzenberger \cite{LS}, see also \cite{Ma,Manivel}. 
For $1\leq i\leq n-1$,  the {\it BGG Demazure operator}   $\partial_{i}$   \cite{BGG} acts  on $\mathbb{Q}[x]=\mathbb{Q}[x_1,\ldots,x_n]$  by letting
\begin{equation*}
\partial_i f = \frac{f-s_if}{x_i-x_{i+1}},
\end{equation*}
where $s_i=t_{i\, i+1}$ is the simple transposition, and $s_i f$ is obtained from $f$ by  swapping  $x_i$ and $x_{i+1}$. 
The double Schubert polynomials $\mathfrak{S}_{w}(x, t)$ for $w\in \S_n$ can be defined recursively by
\begin{align*}
\mathfrak{S}_{w_0}(x, t) = \prod_{i+j\leq n}(x_i- t_j),\qquad 
\partial_i \mathfrak{S}_{w}(x, t)
=\begin{cases}
\mathfrak{S}_{ws_i}(x, t), & \ell(ws_i)<\ell(w),\\
0, & \ell(ws_i)>\ell(w),\\
\end{cases}
\end{align*}
where $w_0=n\cdots 2 1$ refers to the longest permutation of $\S_n$. 
Setting all $t_i=0$ in $\mathfrak{S}_{w}(x, t)$ gives the single Schubert polynomial  
$\mathfrak{S}_w(x)=\mathfrak{S}_{w}(x,0)$. 
For combinatorial models of (double) Schubert polynomials, 
see for example \cite{FandS,KnMi,LLS, MiSt}. 
As  the original motivation, the equivariant Schubert class $[Y(w)]_T$ is represented by $\mathfrak{S}_w(x, t)$ under the Borel isomorphism in Theorem \ref{Boreliso},  see for example \cite[Theorem 6.4]{AF}. 

We collect some properties concerning Schubert polynomials,
which will be used in the proof of Theorem \ref{LUUU-1}
in Section \ref{Sect4}. 

\begin{Prop}\label{PropertiesofSchubertpolynomials}
\begin{enumerate}[\rm(i)]
\item\label{itii}  For a Grassmannian permutation $w_{\lambda}$ with descent at $k$, the single Schubert polynomial $\mathfrak{S}_{w_{\lambda}}(x)$ coincides with the Schur polynomial $s_{\lambda}(x_{[k]})$. 

\item\label{itiii}  There holds the following {Giambelli formula} 
\[
\mathfrak{S}_{w}(x, t) = \sum_{\begin{subarray}{c}
w= v^{-1}  u\\
\ell(w)=\ell(u)+\ell(v)
\end{subarray}}
\mathfrak{S}_{u}(x)
\mathfrak{S}_{v}(- t).
\]

\item\label{itiv}  For $u, v\in \S_n$ with $M(u)\subseteq [k]$ and $M(v)\cap [k]=\emptyset$, 
\[\mathfrak{S}_{uv}(x)=\mathfrak{S}_{u}(x)\mathfrak{S}_{v}(x).\]
Here, $M(w)$ denotes the set of non-fixed points of $w\in \S_n$ as defined in \eqref{MMM}. 
\end{enumerate}
\end{Prop}
 
The above statements   in  \eqref{itii},  \eqref{itiii}  and \eqref{itiv}
can be found in \cite[(4.8)]{Ma}, \cite[pp 87--88]{Ma} and  \cite[(4.6)]{Ma}, respectively.


\subsection{Properties of CSM Classes}

The computation of (equivariant) CSM classes has received attention  in a series of work,
see for example \cite{AM0,AM,AMSS17}.  
For $1\leq i\leq n-1$, the {\it (nonhomogeneous) Demazure--Lusztig type operator} is defined as 
\begin{equation*}
\mathcal{T}_i=-s_i+\partial_i. 
\end{equation*}
It can be directly checked that
\begin{align*}
\mathcal{T}_i^2&=\operatorname{id},\\[5pt] 
\mathcal{T}_i\mathcal{T}_j&=\mathcal{T}_j\mathcal{T}_i,\quad |i-j|\geq 2,\\[5pt]
\mathcal{T}_i\mathcal{T}_{i+1}\mathcal{T}_i&=\mathcal{T}_{i+1}\mathcal{T}_i\mathcal{T}_{i+1}. 
\end{align*}
So,   for  $w\in \S_n$, one may define unambiguously that  
\[\mathcal{T}_w = \mathcal{T}_{i_1}\mathcal{T}_{i_2}\cdots \mathcal{T}_{i_\ell},\]
where   $s_{i_1}s_{i_2}\cdots s_{i_\ell}$ is  any  decomposition (not necessarily reduced) of $w$.
The operators $\mathcal{T}_i$'s as well as $x_i$'s generate the {  degenerate affine Hecke algebra}  introduced by Lusztig \cite{Lusztig} in his study of the representation theory of affine Hecke algebras. 

Under the Borel isomorphism, both $\partial_i$ and $\mathcal{T}_i$ are well defined over $H_T^\bullet(\Fl(n))$ by operating on the variables $x_1,\ldots,x_n$.
A crucial property we will use is the following. 

\begin{Th}[\cite{AM,AMSS17}]\label{recrusionCSM}
Let   $w\in \S_n$. For any $1\leq i \leq n-1$, 
\begin{equation*}
\mathcal{T}_i\big(c_{\mathrm{SM}}^T(Y(w)^\circ)\big)=c_{\mathrm{SM}}^T(Y(ws_i)^\circ). 
\end{equation*}
Therefore, for $u\in \S_n$, we have 
\begin{equation*}
\mathcal{T}_u\big(c_{\mathrm{SM}}^T(Y(w)^\circ)\big)=c_{\mathrm{SM}}^T(Y(wu^{-1})^\circ). 
\end{equation*}
\end{Th}

Taking advantage of the above recurrence relation, a number of properties of CSM classes have been established   \cite{AM,AMSS17}. 
Here we list some that we need in this paper.  

\begin{Prop}[\cite{AM,AMSS17}]
\label{propofCSM} 

\begin{enumerate}[\rm(i)]

\item \label{propofCSMi}
The CSM classes $c_{\mathrm{SM}}^T(Y(w)^\circ)$ form an $\mathbb{F}$-basis of $H_T^\bullet(\Fl(n))_{\mathbb{F}}$, where $\mathbb{F}=\mathbb{Q}(t_1,\ldots,t_n)$ is the fraction field of $H_T^\bullet(\pt)=\mathbb{Q}[t_1,\ldots,t_n]$.

\vspace{5pt}

\item 
The Schubert class $[Y(w)]_T$ is the lowest degree term of   $c^T_{\mathrm{SM}}(Y(w)^\circ)$.

 \vspace{5pt}

\item \label{propofCSMiii}
We have 
\[
\left.c_{\mathrm{SM}}^T(Y(w)^\circ)\right|_{\operatorname{id}} = 
\begin{cases}
{\prod\limits_{i<j} (1+ t_i- t_j)}, & w = \operatorname{id},\\[7pt]
0, & \text{otherwise}.
\end{cases}
\]
\end{enumerate}
\end{Prop}


\begin{Rmk}\label{RmkNoPolyCSM}
Schubert polynomials are stable in the sense that  they are independent of $n$, and can be used to compute the structure constants of Schubert classes, see \cite[\textsection 10.10.2 and \textsection 10.10.4]{AF}. 
However, there have been no known polynomial representatives  for CSM classes with the stability property. 
Stability is a necessary condition to lift the structure constants in the multiplication of CSM classes to the level of polynomials.
Liu \cite{Liu} introduced the notion of twisted Schubert polynomials, which represent CSM classes up to a sign, and showed that this family of polynomials enjoy many interesting combinatorial properties. However, as pointed out in \cite{Liu}, twisted Schubert polynomials 
depend on $n$, and so are not stable.

\end{Rmk}

\section{Rigidity Theorem}\label{Secc-3}

\def\Cminus{\Delta}
\def\Cplus{\Sigma}

In this section, we establish  the Rigidity Theorem for the proof of Theorem  \ref{LLLL-1}, which  bridges the gap between  the Pieri  formulas  for nonequivariant  and equivariant CSM classes.  
 The analogous   idea will be used in  Section \ref{Sect5}
 to establish the   Rigidity Theorem (see Theorem \ref{noneqimplieEqhookMN}) required in the proof of  Theorem  \ref{UUU-1}.

Let $u\in \S_n$, and   $\Gamma=(\alpha+1,1^{\beta})$  be a hook shape. For a subset $A$ of  $[n]$, suppose that 
\[
c_{\mathrm{SM}}(Y(u)^\circ)\cdot s_{\Gamma}(x_{A}) = 
\sum_{w\in \S_n} c_{u,\Gamma}^w\cdot  c_{\mathrm{SM}}(Y(w)^\circ).
\]
Although the coefficients $c_{u,\Gamma}^w$ depend  on 
$A$, we simplify the notation  since there will cause no confusion 
from the context.

\begin{Th}[Rigidity Theorem]\label{noneqimplieEqhook}
Let $u\in \S_n $,    $\Gamma=(\alpha+1,1^{\beta})$ and $A\subseteq[n]$. Suppose that 
\begin{equation*}
c_{\mathrm{SM}}^T(Y(u)^\circ)\cdot s_{\Gamma}(x_{A}) = 
\sum_{w\in \S_n} c_{u,\Gamma}^w( t)\cdot  c_{\mathrm{SM}}^T(Y(w)^\circ).
\end{equation*} 
Then we have
\begin{equation}\label{X-1}
c_{u,\Gamma}^w( t)=\begin{cases}
s_{\Gamma}( t_{uA}), & w=u,\\[8pt]
\displaystyle
\sum_{\begin{subarray}{c}
\alpha'\leq \alpha,\
\beta'\leq \beta\\
\Gamma'=(1+\alpha',1^{\beta'})
\end{subarray}}c_{u,\Gamma'}^w \cdot h_{\alpha-\alpha'}( t_{\Cplus_A(u, w)})\cdot e_{\beta-\beta'}( t_{\Cminus_A(u, w)}) , &w\neq u. 
\end{cases}
\end{equation}
(See \eqref{deltaa} and \eqref{sigmaa} for the definitions of  $\Delta_A(u, w)$ and $\Sigma_A(u, w)$.)
\end{Th}

If we set all $ t_i=0$ in \eqref{X-1}, then the summands  on the right-hand side are nonzero only in  the case $\alpha'=\alpha$ and $\beta'=\beta$, and hence we are led to the expected relationship
\[c_{u,\Gamma}^w(0)=c_{u,\Gamma}^w.\]  
A combinatorial description of  $c_{u,\Gamma}^w$ will be 
given in Theorem \ref{NoneqPierirulehook}, which  combined with Theorem \ref{noneqimplieEqhook}  allows   us to reach a proof of  Theorem  \ref{LLLL-1} (see   Theorem \ref{eqPierirulehook}).

The remaining of this section is devoted to a proof of 
Theorem \ref{noneqimplieEqhook}. We first  investigate the action of the   Demazure operators on a specific family of polynomial quotients related to the generating  function of Schur polynomials of hook shapes.

\subsection{Demazure Operators and Polynomials $\Q$ and $\Z$}\label{OI-1}

For a subset $A\subseteq [n]$, consider 
\begin{equation*}
\Q(x_A) = \prod_{a\in A} (1+\q x_a)\quad \text{and}\quad
\Z(x_A) = \prod_{a\in A} (1-\z x_a), 
\end{equation*}
which are regarded as elements in   
the ring of formal power  series in $q, z$ over   $\mathbb{Q}[x]$. Here, when $A=\varnothing$,
we set $\Q(x_A)=\Z(x_A)=1$. 
Clearly,
\[\Q(x_A)=\sum_{r\geq 0} \q^r  e_r(x_A)
\ \ \ \ \text{and}\ \ \ \ \frac{1}{\Z(x_A)}=\sum_{r\geq 0} z^r  h_r(x_A).\]
Note that  $e_0(x_A)=h_0(x_A)=1$ for any $A\subseteq [n]$, and $e_r(x_\varnothing)=h_r(x_\varnothing)=0$ for $r\geq 1$.


\begin{Lemma}\label{OWW-1}
For $A\subseteq [n]$, 
let
\[E(\q,\z,x_A)=\frac{1}{\q+\z}\left(\frac{\Q(x_A)}{\Z(x_A)}-1\right).\]
Then we have
\[E(\q,\z,x_A)=\sum_{\alpha,\beta\geq 0} \z^\alpha\q^\beta s_{(1+\alpha,1^{\beta})}(x_A).\]
\end{Lemma}

\begin{proof}
Notice that
\[
E(\q,\z,x_A)= \frac{1}{\q+\z}\left(\left(\sum_{r=0}^\infty \z^r h_r(x_A)\right)\left(\sum_{s=0}^\infty \q^s e_s(x_A)\right)-1\right).
\]
Applying  the classical Pieri  rule to the multiplication of 
$h_r(x_A)=s_{(r)}(x_A)$ by $e_s(x_A)$ (see   \cite[\textsection 7.15]{stanley2}), the right-hand side 
becomes 
\[
 \frac{1}{\q+\z}\sum_{r,s\geq 0} \z^{r}\q^s 
\left(s_{(1+r,1^{s-1})}(x_A) + s_{(r,1^{s})}(x_A)\right)
= \sum_{\alpha,\beta\geq 0} \z^\alpha\q^\beta s_{(1+\alpha,1^{\beta})}(x_A),
\]
where we used the assumption  that $ s_{(1+r,1^{-1})}= s_{(0,1^s)}= 0$ for $r, s\geq 0$. 
\end{proof}
 
Another advantage we consider $E(\q,\z,x_A)$ is that 
as $q$ tends to $-z$, $E(\q,\z,x_A)$ becomes  the generating function of power sum symmetric polynomials. 
This allows us to invoke the results established  in this section to 
prove the Rigidity  Theorem for power sum symmetric polynomials as given  in Theorem \ref{noneqimplieEqhookMN}.

We extend the Demazure operators  defined on polynomials in Section \ref{Sec-2} to rational functions. 
For two distinct integers  $a,b\in [n] $ and a rational function $f$, let
\begin{equation}\label{eq:demazuredef}
\partial_{ab}f = \frac{f-t_{ab}f}{x_a-x_b},
\end{equation}
where   $t_{ab} f$ is obtained from $f$ by interchanging $x_a$ and $x_b$.  When $f$ is symmetric in $x_a$ and $x_b$, 
it is easily seen  that $\partial_{ab}(f)=0$, and  for any polynomial $g$,
\[\partial_{ab}(fg)=f\,\partial_{ab}g. \]

\begin{Lemma}\label{QVrelation1}
For   distinct  $a,b\in [n]$ and two subsets $A\subseteq B$ of $[n]$, we have 
\begin{equation}\label{eq:QVrelation1}
\partial_{ab} \frac{\Q(x_A)}{\Z(x_B)} 
= \big(\delta_{a\in A}\q-\delta_{b\in A}\q
-\delta_{a\notin B}\z + \delta_{b\notin B}\z\big)
\frac{\Q(x_{A\setminus \{a,b\}})}{\Z(x_{B\cup \{a,b\}})},
\end{equation}
where $\delta$ is the Kronecker delta, that is,  $\delta_\lozenge$ is $1$ if the condition $\lozenge$ is satisfied and   $0$ otherwise. 
\end{Lemma}

\begin{proof}
If any one  of the following conditions is satisfied:
\[a,b\in A,\qquad 
a,b\in B\setminus A,\qquad \text{or}\ \ 
a,b\in [n]\setminus B, 
\]
then $\frac{\Q(x_A)}{\Z(x_B)} $ is symmetric in $x_a$ and $x_b$, and in these cases both sides of \eqref{eq:QVrelation1} are zero.
Noticing that $\partial_{ab}=-\partial_{ba}$, it remains  to verify 
the following three cases. 

Case 1. $a\in A$ and $b\in B\setminus A$. In this case, 
\begin{align*}
 \partial_{ab} \frac{\Q(x_A)}{\Z(x_B)}&=
\frac{\Q(x_{A\setminus \{a\}})}{\Z(x_{B\setminus \{a,b\}})}\, \partial_{ab}\frac{1+\q x_a}{(1-\z x_a)(1-\z x_b)}\\[5pt]
&=\frac{\Q(x_{A\setminus \{a\}})}{\Z(x_{B\setminus \{a,b\}})}\,\frac{\q}{(1-\z x_a)(1-\z x_b)},
\end{align*}
which, along with the fact $A\setminus \{a\}=A\setminus \{a, b\}$, agrees with \eqref{eq:QVrelation1}.

Case 2. $a\in A$ and $b\in [n]\setminus B$. 
In this case,
\begin{align*}
 \partial_{ab} \frac{\Q(x_A)}{\Z(x_B)}=
\frac{\Q(x_{A\setminus \{a\}})}{\Z(x_{B\setminus \{a\}})}\, \partial_{ab}\frac{1+\q x_a}{(1-\z x_a)}=\frac{\Q(x_{A\setminus \{a\}})}{\Z(x_{B\setminus \{a\}})}\,\frac{\q+\z}{(1-\z x_a)(1-\z x_b)},
\end{align*}
which is  the same as \eqref{eq:QVrelation1}.

Case 3.  $a\in B\setminus A$ and $b\in [n]\setminus B$. In this case,
\begin{align*}
 \partial_{ab} \frac{\Q(x_A)}{\Z(x_B)}=
\frac{\Q(x_{A})}{\Z(x_{B\setminus \{a\}})}\, \partial_{ab}\frac{1}{(1-\z x_a)}=\frac{\Q(x_{A})}{\Z(x_{B\setminus \{a\}})}\,\frac{\z}{(1-\z x_a)(1-\z x_b)},
\end{align*}
which also coincides with  \eqref{eq:QVrelation1}.
\end{proof}


For   $w\in \S_n$, let 
\begin{align}\label{MMM}
M(w) = \{1\leq i\leq n\colon w(i)\neq i\} 
\end{align} 
represent  the set of non-fixed points of $w$.

\begin{Lemma}\label{nonfixinglemma}
For $m\geq 0$, assume that 
$w=t_{a_1b_1}\cdots t_{a_mb_m}\in \S_n.$
If 
\begin{equation}\label{XX-1}
\partial_{a_1b_1}\cdots \partial_{a_mb_m} \frac{\Q(x_A)}{\Z(x_B)} \neq 0
\end{equation}
for some subsets $A\subseteq B$ of $[n]$, then
\begin{equation}\label{XX-2}
M(w) = \{a_1,b_1,\ldots,a_m,b_m\}.
\end{equation}
\end{Lemma}

\begin{proof}
It is obvious that $M(w)\subseteq \{a_1,b_1,\ldots,a_m,b_m\}$. We next prove the reverse inclusion 
by induction on $m$. When $m=0$, $w$ is the identity permutation,  both sides of \eqref{XX-2} are  empty, and we are done. 

Now consider the case  $m>0$. 
Suppose  to the contrary that $w(x)=x$ for some $x\in \{a_1,b_1,\ldots,a_m,b_m\}$. 
For $0\leq j\leq m-1$, let $w_j=t_{a_{j+1}b_{j+1}}\cdots t_{a_mb_m}$.  
Let $i$ be the smallest $j$ such that $w_j(x) \neq x$. 
Clearly, we have $i>0$.  Since $w_{i-1}(x)=t_{a_i b_i}w_i(x)=x$ and $w_i(x)\neq x$, we obtain  that
\begin{equation}\label{XX-7}
\{a_i,b_i\}=\{x,w_{i}(x)\}.    
\end{equation} 
On the other hand, by \eqref{XX-1},  we have
\[\partial_{a_{i+1}b_{i+1}}\cdots \partial_{a_mb_m} \frac{\Q(x_A)}{\Z(x_B)}\neq 0,\]
and so it follows by induction that  
\begin{equation*}\label{XX-77}
M(w_i)=\{a_{i+1}, b_{i+1},\ldots, a_m, b_m\},
\end{equation*}
which along with Lemma \ref{QVrelation1} implies that 
\begin{equation}\label{XX-777}
\partial_{a_{i+1}b_{i+1}}\cdots \partial_{a_mb_m} \frac{\Q(x_A)}{\Z(x_B)}\in \mathbb{Z}[\q,\z]\frac{\Q(x_{A\setminus M(w_i)})}{\Z(x_{B\cup M(w_i)})}. 
\end{equation}
Since $w_i(x)\neq x$, we  see that $w_i(w_i(x))\neq w_i(x)$. Thus  both $x$ and $w_{i}(x)$  belong to $M(w_i)$. 
In view of  \eqref{XX-7}, both $a_i$ and $b_i$ belong to $M(w_i)$.
Consequently, 
\[\partial_{a_i b_i}\frac{\Q(x_{A\setminus M(w_i)})}{\Z(x_{B\cup M(w_i)})}=\frac{\Q(x_{A\setminus M(w_i)})}{\Z(x_{B\cup M(w_i)\setminus\{a_i,  b_i\}})} \partial_{a_i b_i} \frac{1}{(1-zx_{a_i})(1-zx_{b_i})}=0,\] 
which together  with \eqref{XX-777} would yield 
\[\partial_{a_{i}b_{i}}\partial_{a_{i+1}b_{i+1}}\cdots \partial_{a_mb_m} \frac{\Q(x_A)}{\Z(x_B)}=0,\]
contrary to the assumption in \eqref{XX-1}. This completes the proof. 
\end{proof}





\subsection{Demazure--Lusztig Operators and Proof of Theorem \ref{noneqimplieEqhook}}

For  $w\in \S_n$, fix a reduced word of $w$:
\begin{equation*}\label{eq:reducedwordforw}
w= s_{i_1}\cdots s_{i_\ell}. 
\end{equation*}
For any subset $J\subseteq [\ell]$, define
\begin{align}\label{w_J}
w_J=\prod_{j \in J} s_{i_j} 
\quad \text{and}\quad
\nabla_J = \prod_{j=1}^\ell\begin{cases}
s_{i_j}, & j\in J,\\[5pt]
\partial_{i_j},& j \notin J,
\end{cases}
\end{align}
where the factors in the product are multiplied  from left to right as $j$ increases.  
By direct calculation,    $\mathcal{T}_i=-s_i+\partial_i$ satisfies the  Leibniz rule:
\begin{equation}\label{eq:Leibnizrule}
\mathcal{T}_i(fg)= \mathcal{T}_i f\cdot s_ig + f\cdot\partial_i g.
\end{equation}
It should be noticed  that the twisted operator $T_i=s_i+\partial_i$, which has been used by Liu  \cite{Liu} to define twisted Schubert polynomials,   possesses analogous  properties to $\mathcal{T}_i$ \cite[Proposition 3.3]{Liu}.
Iterating \eqref{eq:Leibnizrule}, it is not hard to check that 
\begin{equation}\label{X-11}
\mathcal{T}_w(fg) = \sum_{J\subseteq [\ell]} \mathcal{T}_{w_J}f \cdot \nabla_J g
=\sum_{v\in \S_n}  \mathcal{T}_vf \cdot \sum_{\begin{subarray}{c}{J\subseteq [\ell]}\\{w_J=v}\end{subarray}} \nabla_{J}g. 
\end{equation}
For $v, w\in \S_n$,   define the {skew operator $\mathcal{T}_{w/ v}$}  as 
\begin{equation}\label{eq:skewoperators}
\mathcal{T}_{w/ v} = \sum_{\begin{subarray}{c}J\subseteq [\ell]\\ w_J=v\end{subarray}}\nabla_J. 
\end{equation}
Note that since  $v$ appears as a subword in the reduced  word of $w$,  $\mathcal{T}_{w/v}$ is zero unless $v\leq w$ in the Bruhat order. With the notation in \eqref{eq:skewoperators},  \eqref{X-11} can be rewritten as
\begin{equation}\label{X-111}
\mathcal{T}_w(fg)=\sum_{v\in \S_n} \mathcal{T}_vf\cdot  \mathcal{T}_{w/ v} g. 
\end{equation}

Let $\alpha$ be a class in $H_T^\bullet(\Fl(n))$. 
By Proposition \ref{propofCSM}, one has the following expansion  
\begin{equation*}
\alpha = \sum_{w\in \S_n} c^w_\alpha( t)\cdot  c_{\mathrm{SM}}^T(Y(w)^\circ),
\end{equation*}
where the coefficients  $c^w_\alpha( t)$ are rational functions in $t$.

\begin{Lemma} 
We have
\begin{equation}\label{eq:inversion formula}
c^w_\alpha( t) = \frac{1}{\prod_{1\leq i<j\leq n}(1+ t_i- t_j)}\mathcal{T}_w(\alpha)\big|_{\operatorname{id}}.
\end{equation}
 \end{Lemma}

\begin{proof}
 By Theorem \ref{recrusionCSM} and Proposition \ref{propofCSM},   $c^w_\alpha( t)$ could be computed   by applying the operator
$\mathcal{T}_w$ to $\alpha$ and then invoking   the localization map $|_{\operatorname{id}}$.
\end{proof}

This combined with \eqref{X-111} leads to the following important observation.

\begin{Lemma}\label{PP-1}
Suppose that for $u\in \S_n$  and a class  $\alpha\in H_T^\bullet(\Fl(n))$,
\begin{equation*}
c_{\mathrm{SM}}^T(Y(u)^\circ)\cdot \alpha = \sum_{w\in \S_n} c^w_{u,\alpha}( t)\cdot  c_{\mathrm{SM}}^T(Y(w)^\circ). 
\end{equation*}
Then we have
\begin{align}
c_{u,\alpha}^w( t) = \left.\mathcal{T}_{w/ u} (\alpha) \right|_{\operatorname{id}}.
\end{align}
\end{Lemma}

 \begin{proof}

Replacing $\alpha$ by $c_{\mathrm{SM}}^T(Y(u)^\circ)\cdot \alpha$   in \eqref{eq:inversion formula} gives 
\begin{equation}\label{X-X1}
c^w_{u,\alpha}( t) = \frac{1}{\prod_{1\leq i<j\leq n}(1+ t_i- t_j)}\mathcal{T}_w \left.(c_{\mathrm{SM}}^T(Y(u)^\circ)\cdot \alpha) \right|_{\operatorname{id}}. 
\end{equation}
Plugging \eqref{X-111} into \eqref{X-X1} and   using Theorem \ref{recrusionCSM}  and  Proposition 
\ref{propofCSM} , we deduce that
\begin{align*} 
c^w_{u,\alpha}( t)& = \frac{1}{\prod_{1\leq i<j\leq n}(1+ t_i- t_j)}\sum_{v\in \S_n} \left.\mathcal{T}_v(c_{\mathrm{SM}}^T(Y(u)^\circ)\right|_{\operatorname{id}}\cdot  \left.\mathcal{T}_{w/ v} (\alpha) \right|_{\operatorname{id}} \\[5pt]
& = \frac{1}{\prod_{1\leq i<j\leq n}(1+ t_i- t_j)}\sum_{v\in \S_n} \left.c_{\mathrm{SM}}^T(Y(uv^{-1})^\circ)\right|_{\operatorname{id}}\cdot  \mathcal{T}_{w/ v} (\alpha) \big|_{\operatorname{id}}\\[5pt]
& = \frac{1}{\prod_{1\leq i<j\leq n}(1+ t_i- t_j)}  \left.c_{\mathrm{SM}}^T(Y(\operatorname{id})^\circ)\right|_{\operatorname{id}}\cdot  \left.\mathcal{T}_{w/ u} (\alpha) \right|_{\operatorname{id}}\\[5pt]
&=\left.\mathcal{T}_{w/ u} (\alpha) \right|_{\operatorname{id}},
\end{align*}
as required.
\end{proof}


The last lemma concerns the action of   skew operators on the generating function $E(\q,\z,x_A)$ of Schur polynomials of hook shapes as defined in Lemma \ref{OWW-1}. 

\begin{Lemma}\label{NonimplieEquLemma}
For $u, w\in \S_n$  with $u\neq w$, we have
\begin{equation}\label{O-1}
\mathcal{T}_{w/ u} E(\q,\z,x_A) \in \frac{1}{\q+\z}\mathbb{Z}[\q,\z] \frac{\Q(x_{\Cminus_A(u, w)})}{\Z(x_{\Cplus_A(u, w)})}.
\end{equation}
\end{Lemma}

\begin{proof}
For any $v\in \S_n$, it is easy to check that 
\begin{equation*}\label{conjofpartialab}
v\partial_{ab} = \partial_{v(a)v(b)}v. 
\end{equation*}
Hence, for  $J\subseteq [\ell]$ such that  $w_J=u$,
we can interchange the Demazure operators appearing in $\nabla_J$ defined in \eqref{w_J} one by one to the rightmost side  of $w_J$, and so  we may  assume that $\nabla_J$ takes the  form
\begin{equation*}\label{eq:assumewuss}
\nabla_J=u\partial_{a_1b_1}\cdots \partial_{a_rb_r}.
\end{equation*}
On the other hand, since  $vt_{ab} = t_{v(a)v(b)}v$ for any $v\in \S_n$, we can use exactly the same procedure  with  $\nabla_J$ to deduce that
\[
w= u t_{a_1b_1}\cdots t_{a_rb_r}, 
\]
or equivalently,
\[u^{-1} w=t_{a_1b_1}\cdots t_{a_rb_r}.\]
Combining  Lemma \ref{QVrelation1} and Lemma \ref{nonfixinglemma}, we obtain that
\begin{equation}\label{OO-2}
\partial_{a_1b_1}\cdots \partial_{a_rb_r} \frac{\Q(x_A)}{\Z(x_A)} \in \mathbb{Z}[\q,\z] \frac{\Q(x_{A\setminus M(u^{-1} w)})}{\Z(x_{A\cup M(u^{-1} w)})}.
\end{equation}
Since $u\neq w$, we have $r\geq 1$ and so 
\[\partial_{a_1b_1}\cdots \partial_{a_rb_r} \frac{1}{\q+\z}=0, \]
which together with \eqref{OO-2} leads to
\begin{align*}
\partial_{a_1b_1}\cdots \partial_{a_rb_r} E(\q,\z,x_A) 
&=\partial_{a_1b_1}\cdots \partial_{a_rb_r}\frac{1}{\q+\z}
\frac{\Q(x_A)}{\Z(x_A)}-\partial_{a_1b_1}\cdots \partial_{a_rb_r}\frac{1}{\q+\z}\\[5pt]
&\in \frac{1}{\q+\z}\mathbb{Z}[\q,\z] \frac{\Q(x_{A\setminus M(u^{-1} w)})}{\Z(x_{A\cup M(u^{-1} w)})}.
\end{align*}
Therefore, 
\begin{align*}
\nabla_J  E(\q,\z,x_A) &= u\partial_{a_1b_1}\cdots \partial_{a_rb_r}E(\q,\z,x_A) \\[5pt]
&\in   \frac{1}{\q+\z}\mathbb{Z}[\q,\z] \frac{\Q(x_{u(A\setminus M(u^{-1} w))})}{\Z(x_{u(A\cup M(u^{-1} w))})} \\[5pt]
&=
\frac{1}{\q+\z}\mathbb{Z}[\q,\z] \frac{\Q(x_{\Cminus_A(u, w)})}{\Z(x_{\Cplus_A(u, w)})},
\end{align*}
yielding \eqref{O-1}.
\end{proof}

We are finally in a position to complete the proof of   Theorem \ref{noneqimplieEqhook}.

\begin{proof}[Proof of Theorem \ref{noneqimplieEqhook}]
Recall that \begin{equation*}\label{X-3}
c_{\mathrm{SM}}^T(Y(u)^\circ)\cdot s_{\Gamma}(x_{A}) = 
\sum_{w\in \S_n} c_{u,\Gamma}^w( t)\cdot  c_{\mathrm{SM}}^T(Y(w)^\circ).
\end{equation*} 
Let
\begin{align}\label{cuwt}
c_u^w(\q,\z, t) = \sum_{\begin{subarray}{c}{\alpha,\beta\ge0}\\{\Gamma=(1+\alpha,1^\beta)}\end{subarray}} \z^\alpha\q^\beta \cdot c_{u,\Gamma}^w( t).
\end{align}
Then by Lemma \ref{OWW-1},
\begin{align*}
c_{\mathrm{SM}}^T(Y(u)^\circ)\cdot E(\q,\z,x_{A})
&=c_{\mathrm{SM}}^T(Y(u)^\circ)\cdot \sum_{\alpha,\beta\geq 0} \z^\alpha\q^\beta \cdot s_{(1+\alpha,1^\beta)}(x_A)\\[5pt]
&=\sum_{w\in \S_n} c_u^w(\q,\z, t)\cdot c_{\mathrm{SM}}^T(Y(w)^\circ).  
\end{align*}
By Lemma \ref{PP-1}, we obtain that
\begin{equation}\label{RRRW}
c_u^w(\q,\z, t) = \mathcal{T}_{w/ u} E(\q,\z,x_{A})\big|_{x_i= t_i}. 
\end{equation}
If $w=u$, then the skew operator $\mathcal{T}_{u/u}$ is nothing but $u$ itself, and so  
\[
c_u^u(\q,\z, t)=E(\q,\z, t_{uA})=\sum_{\alpha,\beta\geq 0} \z^{\alpha}\q^\beta \cdot s_{(1+\alpha,1^\beta)}(t_{uA}),
\]
which gives $c_{u,\Gamma}^u( t)=s_{(1+\alpha,1^\beta)}(t_{uA})$.
If $w\neq u$, then, by \eqref{RRRW} and Lemma \ref{NonimplieEquLemma}, we see that 
\[c_u^w(\q,\z, t)=f(q,z) \frac{\Q( t_{\Cminus_A(u,w)})}{\Z( t_{\Cplus_A(u,w)})},\]
where $f(q,z) \in \frac{1}{\q+\z}\mathbb{Z}[\q,\z].$
Setting all $t_i=0$ on both sides, we obtain that $c_u^w(\q,\z,0)=f(q,z)$, and hence,  
\begin{align}
c_u^w(\q,\z, t) &= c_u^w(\q,\z,0)\cdot  
\sum_{r,s\geq 0} \z^{r}\q^s \cdot h_r( t_{\Cplus_A(u, w)})\cdot e_s( t_{\Cminus_A(u, w)})\nonumber\\
&=\left(\sum_{\begin{subarray}{c}\alpha',\beta'\ge0\\ \Gamma=(1+\alpha',1^{\beta'})\end{subarray}} \z^{\alpha'}\q^{\beta'} \cdot c_{u,\Gamma}^w(0)\right)\cdot \sum_{r,s\geq 0} \z^r\q^s \cdot h_r( t_{\Cplus_A(u, w)})\cdot e_s( t_{\Cminus_A(u, w)}).\label{cuwtt}
\end{align}
Comparing  the coefficients of $\z^\alpha\q^\beta$  in \eqref{cuwt} and \eqref{cuwtt}, we are led to \eqref{X-1}.
\end{proof}

\section{Pieri Type Rules}\label{Sect4}

In this section, we will prove  Theorems \ref{LLLL-1} and \ref{LUUU-1}.
To this end, we first establish the expansion formula for multiplying  a nonequivariant CSM class  by a Schur polynomial of   hook shape, see Theorem \ref{NoneqPierirulehook}. Combining  Theorem \ref{NoneqPierirulehook} with the Rigidity Theorem \ref{noneqimplieEqhook} gives a proof  of Theorem  \ref{LLLL-1}. Based on Theorem  \ref{LLLL-1} and properties of double Schubert polynomials, we 
finish the proof of Theorem \ref{LUUU-1}.

\subsection{Nonequivariant Case}

We begin by proving   a Pieri formula  for nonequivariant CSM classes. 
The proof relies on the property of the action of the skew operators $\mathcal{T}_{w/u}$   on $e_{r}(x_{[k]})$ or $h_{r}(x_{[k]})$, which has been investigated  by Liu \cite{Liu}.

\begin{Th}[CSM Pieri formula]\label{CSMPieriRuleforeh}
Let $u\in \S_n$. 
We have the following identities in $H^\bullet(\Fl(n)):$
\begin{enumerate}[\rm(i)]
\item   \label{thm411}

For $r\geq 0$, 
\begin{equation}\label{CMSeP}
c_{\mathrm{SM}}(Y(u)^\circ)\cdot e_{r}(x_{[k]}) = \sum_{w\in \S_n} c_{\mathrm{SM}}(Y(w)^\circ),
\end{equation}
where the sum ranges over all $w\in \S_n$ such that there is path
\[
u\longrightarrow
ut_{a_1b_1}
\longrightarrow
ut_{a_1b_1}t_{a_2b_2}\longrightarrow\cdots\longrightarrow w=ut_{a_1b_1}\cdots t_{a_rb_r}
\]
 from $u$ to $w$ in the extended  $k$-Bruhat order and $a_1,a_2,\ldots,a_r$ are distinct.

\item  \label{thm412}
For $r\geq 0$, 
\begin{equation}\label{CMShP}
c_{\mathrm{SM}}(Y(u)^\circ)\cdot h_{r}(x_{[k]}) = \sum_{w\in \S_n} c_{\mathrm{SM}}(Y(w)^\circ),
\end{equation}
where the sum ranges over all $w$ as in \eqref{thm411}, except that now the 
integers $b_1,\ldots, b_r$ are distinct.
\end{enumerate}

\end{Th}

\begin{proof}
We only give a proof of \eqref{CMSeP}, and the arguments  for 
\eqref{CMShP} are similar. 
By Lemma \ref{PP-1}, we need to show that for $w\in \S_n$,
\begin{equation}\label{eq:erelation4}
 \mathcal{T}_{w/u}\big(e_r(x_{[k]})\big)\big|_{x_i=0}\neq 0   
\end{equation}
if and only if $w$ satisfies the  conditions in \eqref{thm411},
and in this case, the nonzero value in \eqref{eq:erelation4}  is exactly equal to $1$.

Set $\tilde{\partial}_{w/u}=u^{-1}\,\mathcal{T}_{w/u}$.
We say that $w$ is legal if 
 $w$   satisfies the conditions in \eqref{thm411}, except 
that we replace $r$ by any nonnegative  integer $m$. 
If $w$ is not legal,
then it follows from \cite[Theorem 3.10]{Liu} that 
$\tilde{\partial}_{w/u}\big(e_r(x_{[k]})\big)=0$.
If $w$   is legal, then \cite[Theorem 3.10]{Liu} gives 
\[\tilde{\partial}_{w/u}\big(e_r(x_{[k]})\big)=e_{r-|A|}(x_{[k]\setminus A}),\]
where $A=\{a_1,\ldots, a_m\}$. 
So, 
\[\mathcal{T}_{w/u}\big(e_r(x_{[k]})\big)=u\,\tilde{\partial}_{w/u}=e_{r-|A|}(x_{u([k]\setminus A)}),\]
implying that    
the value in \eqref{eq:erelation4} vanishes if $m\neq r$, and 
equals 1 if $m=r$. This completes the proof. 
\end{proof}

The  permutations appearing in Theorem \ref{CSMPieriRuleforeh} can be alternatively characterized  
in terms of  increasing or decreasing    paths in the extended 
$k$-Bruhat order.

\begin{Lemma}\label{erelation3}
We have the following  statements.
\begin{itemize}
\item[(1)] A permutation $w\in \S_n$ satisfies the conditions in 
\eqref{thm411} of Theorem \ref{CSMPieriRuleforeh} if and only if there exists a   decreasing  path of length $r$ from $u$ to $w$ in the extended $k$-Bruhat order. Moreover, the decreasing path is unique. 

\vspace{5pt}
  
 \item[(2)]    
 A permutation $w\in \S_n$ satisfies the conditions in 
\eqref{thm412} of Theorem \ref{CSMPieriRuleforeh} if and only if there exists an  increasing  path  of length $r$ from $u$ to $w$ in the extended $k$-Bruhat order.  Moreover, the increasing path is unique. 
\end{itemize}

\end{Lemma}

\begin{proof}
We only give a proof of the  statement in (1), and a similar analysis applies to the statement in (2). 
Let us first verify  the necessity.  Suppose that 
$w$ satisfies the conditions in 
\eqref{thm411} of Theorem \ref{CSMPieriRuleforeh}. 
The existence of a decreasing path is implied by the following easily checked  claim.

\begin{Claim} 
Assume that there is a path
\[
v
\stackrel{\tau}\longrightarrow
vt_{a_1b_1}
\stackrel{\sigma}\longrightarrow
vt_{a_1b_1}t_{a_2b_2}=v'
\]
from $v$ to $v'$ in the extended $k$-Bruhat order
with $a_1\neq a_2$ and $\tau<\sigma$. Then there is
an alternative  path from $v$ to $v'$:
\[
v 
\stackrel{\sigma}\longrightarrow
vt_{a_2b_2}
\stackrel{\tau}\longrightarrow
vt_{a_2b_2}t_{a_1b_1}=v'.
\]
\end{Claim}
In fact, 
since $$\tau=v(a_1)=vt_{a_1b_1}(b_1)<\sigma = vt_{a_1b_1}(a_2)<vt_{a_1b_1}(b_2),$$
we have $b_1\neq b_2$. So, $t_{a_1b_1}$ and $t_{a_2b_2}$ commute, and the claim follows. 
Given $u$ and $w$ as in \eqref{thm411} of Theorem  \ref{CSMPieriRuleforeh}, we have a
path 
\[
u
\stackrel{\sigma_1}\longrightarrow ut_{a_1b_1} 
\stackrel{\sigma_2}\longrightarrow\cdots 
\stackrel{\sigma_r}\longrightarrow ut_{a_1b_1}\cdots  t_{a_rb_r}=w
\]
from  $u$ to $w$
in the extended $k$-Bruhat order.
If this path is not decreasing, we can iterate the above Claim and eventually obtain a decreasing path 
of length $r$
from $u$ to $w$  in the extended $k$-Bruhat order.

We  next
prove the sufficiency.
Let 
\[
\gamma'\colon u=w_0'
\stackrel{\tau_1'}\longrightarrow w_1' 
\stackrel{\tau_2'}\longrightarrow\cdots 
\stackrel{\tau_r'}\longrightarrow w_{r}'=w
\]
be a decreasing path 
of length $r$
from $u$ to $w$  in the extended $k$-Bruhat order. 
Assume that for $1\leq i\leq r$, 
\[w_i'=ut_{a_1' b_1'}\cdots t_{a_i' b_i'}.\]
To conclude the sufficiency, we show that the integers 
$a_1', a_2', \ldots, a_r'$ are distinct. We first explain that $a_1'\neq a_2'$. Notice that $\tau_1'=u(a_1')<w_1'(a_1')$
and $w_1'(j)=u(j)$ for $j\in [k]\setminus \{a_1'\}$. 
Since $w_1'(a_2')=\tau_2'<\tau_1'$, we have $ w_1'(a_2')< w_1'(a_1')$. This implies that $a_2'\neq a_1'$. Moreover,  we see that 
\[w_2'(a_1')=w_1'(a_1')>u(a_1'),\  w_2'(a_2')=w_1'(b_2')>w_1'(a_2')=u(a_2'),\]
and 
 $w_2'(j)=u(j)$ for $j\in [k]\setminus \{a_1', a_2'\}$.

We proceed to check that $a_3'$ is different from $a_1'$ and $a_2'$.  Since $w_2'(a_1')>u(a_1')=\tau_1'$, $w_2'(a_2')>u(a_2')=\tau_2'$ and $w_2'(a_3')=\tau_3'<\tau_2'<\tau_1'$, we obtain that $w_2'(a_3')$ is smaller than $w_2'(a_1')$ and $w_2'(a_2')$, and so $a_3'\neq a_1', a_2'$. Moreover, we have 
\[w_3'(a_1')=w_2'(a_1')>u(a_1'),\  w_3'(a_2')=w_2'(a_2')=w_1'(b_2')>u(a_2'),\  w_3'(a_3')>u(a_3'),\]
and  $w_3'(j)=u(j)$ for $j\in [k]\setminus \{a_1', a_2', a_3'\}$.
Using a similar analysis, we can deduce  that for $i=4,\ldots, r$, $a_i'$ is 
different from $a_1',\ldots,  a_{i-1}'$, and this verifies  the sufficiency.

Finally, we show that the decreasing path $\gamma'$ is unique. Since $a_1',\ldots,  a_{r}'$ are distinct, we know that for $1\leq i\leq r$,
$\tau_i=u(a_i')$,
or equivalently, $a_i'=u^{-1}(\tau_i)$. 
On the other hand, since $u(b'_1)=w_1'(a_1')=w(a_1')$, we see that $b_1'$ is uniquely located. Similarly, since $w_1'(b'_2)=w_2'(a_2')=w(a_2')$ and $w_1'=ut_{a_1'b_1'}$, we see that $b_2'$ is uniquely located. Iterating the same process, we can eventually obtain that $b_1',b_2',\ldots, b_r'$ are  uniquely determined. This verifies the uniqueness.
\end{proof}

\begin{Rmk}\label{REM-1}
Let $u, w\in \S_n$. Suppose that there is a (unique) decreasing (resp., increasing) path $\gamma$
from $u$ to $w$ in the extended $k$-Bruhat order. By the proof of Lemma \ref{erelation3}, the $a_i$'s (resp., $b_i$'s) are distinct, thus the length of $\gamma$ is equal to
\[\texttt{\#}\, [k]\cap M(u^{-1}w) \ \ \text{(resp., $\texttt{\#}\, ([n]\setminus [k])\cap M(u^{-1}w)$)}.\]
\end{Rmk}

 For a path $\gamma$ in the extended $k$-Bruhat order, we  use  $\mathrm{end}(\gamma)$ to denote the endpoint permutation  of 
 $\gamma$.
By Lemma \ref{erelation3}, we can reformulate  Theorem \ref{CSMPieriRuleforeh}   as follows.

\begin{Coro}[CSM Pieri Formula]\label{REW}
Let $u\in \S_n$. 
We have the following identities in $H^\bullet(\Fl(n))$:
\begin{enumerate}[\rm(i)]
\item \label{xjlabl}
For $r\geq 0$, 
\begin{equation}\label{CMSePpath}
c_{\mathrm{SM}}(Y(u)^\circ)\cdot e_{r}(x_{[k]}) = \sum_{\gamma} c_{\mathrm{SM}}(Y(\mathrm{end}(\gamma))^\circ),
\end{equation}
where the sum ranges  over all   decreasing paths of length $r$ starting at $u$ 
in the extended $k$-Bruhat order. 

\item \label{xjlabl2}
For $r\geq 0$, 
\begin{equation}\label{CMShPpath}
c_{\mathrm{SM}}(Y(u)^\circ)\cdot h_{r}(x_{[k]}) = \sum_{\gamma} c_{\mathrm{SM}}(Y(\mathrm{end}(\gamma))^\circ),
\end{equation}
where the sum ranges  over all   increasing paths of length $r$ starting at $u$ 
in the extended $k$-Bruhat order.
\end{enumerate}
\end{Coro}

We can now exhibit the expansion formula   for the multiplication of a CSM class by a Schur polynomial  of hook shape. 

\begin{Th}\label{NoneqPierirulehook}
For a hook shape partition $\Gamma=(\alpha+1,1^{\beta})$, we have the following identity in $H^\bullet(\Fl(n)):$
\begin{equation}\label{eq:CSMhook}
c_{\mathrm{SM}}(Y(u)^\circ)\cdot s_{\Gamma}(x_{[k]}) = \sum_{\gamma} c_{\mathrm{SM}}(Y(\mathrm{end}(\gamma))^\circ),
\end{equation}
where the sum runs  over all peakless paths $\gamma$  starting at  $u$   in the extended $k$-Bruhat order with $\mathrm{in}(\gamma)=\alpha$ and $\mathrm{de}(\gamma)=\beta$. 
\end{Th}

\begin{proof}
We make induction on $\beta$.
The case $\beta=0$ is nothing but \eqref{xjlabl2} of Corollary \ref{REW}. 
Now we  assume that $\beta\geq 1$.
By Corollary \ref{REW}, we see that
\[
c_{\mathrm{SM}}(Y(u)^\circ)\cdot e_\beta(x_{[k]})\cdot h_{\alpha+1}(x_{[k]}) = \sum_{\gamma'} c_{\mathrm{SM}}(Y(w)^\circ),
\]
where the sum is taken over all paths 
\[
\gamma'\colon u =w_0
\stackrel{\tau_1}\longrightarrow w_1 
\longrightarrow\cdots 
\stackrel{\tau_{\alpha+\beta+1}}\longrightarrow w_{\alpha+\beta+1}=w
\]
 in the extended $k$-Bruhat order such that 
\[
\tau_1> \cdots >\tau_{\beta} \ \ \text{and} \ \ \tau_{\beta+1}<\cdots < \tau_{\alpha+\beta+1}.
\]
Since $\tau_\beta=w_\beta(i)$ 
for some $i>k$ and $\tau_{\beta+1}=w_\beta(j)$ for some $j\leq k$, we have $\tau_\beta\neq \tau_{\beta+1}$.
Thus  either $\tau_\beta>\tau_{\beta+1}$ or $\tau_\beta<\tau_{\beta+1}$.

If $\tau_\beta<\tau_{\beta+1}$, then $\gamma'$ is exactly a peakless path with $\mathrm{in}(\gamma')=\alpha+1$ and $\mathrm{de}(\gamma')=\beta-1$. On the other hand, by induction, 
\[
c_{\mathrm{SM}}(Y(u)^\circ)\cdot {s_{(2+\alpha,1^{\beta-1})}}(x_{[k]})
= \sum_{\gamma''} c_{\mathrm{SM}}(Y(\mathrm{end}(\gamma''))^\circ),
\]
where the sum is taken  over all peakless paths $\gamma''$ with $\mathrm{in}(\gamma'')=\alpha+1$ and $\mathrm{de}(\gamma'')=\beta-1$. 
Therefore,  
\begin{align*}
&c_{\mathrm{SM}}(Y(u)^\circ)\cdot h_{\alpha+1}(x_{[k]})\cdot e_\beta(x_{[k]})
- c_{\mathrm{SM}}(Y(u)^\circ)\cdot {s_{(2+\alpha,1^{\beta-1})}}(x_{[k]})\\[5pt]
&=c_{\mathrm{SM}}(Y(u)^\circ)\cdot ( h_{\alpha+1}(x_{[k]})\cdot e_\beta(x_{[k]}) 
-  {s_{(2+\alpha,1^{\beta-1})}}(x_{[k]}))\\[5pt]
&=\sum_{\gamma} c_{\mathrm{SM}}(Y(\mathrm{end}(\gamma))^\circ),
\end{align*}
where the sum runs  over all peakless paths $\gamma$  starting at  $u$   in the extended $k$-Bruhat order with $\mathrm{in}(\gamma)=\alpha$ and $\mathrm{de}(\gamma)=\beta$. 
We arrive at  \eqref{eq:CSMhook} by noticing the identity 
\[h_{\alpha+1}(x_{[k]})\cdot e_\beta(x_{[k]})=s_{(1+\alpha,1^{\beta})}(x_{[k]})
+{s_{(2+\alpha,1^{\beta-1})}}(x_{[k]}),\]
which  follows from  the classical Pieri rule for Schur polynomials.
\end{proof}

We may alternatively use $h_{\alpha+1}(x_{[k]})\cdot e_\beta(x_{[k]})$ instead of $e_\beta(x_{[k]})\cdot h_{\alpha+1}(x_{[k]})$ in the proof of Theorem \ref{NoneqPierirulehook}, which leads to a dual version of Theorem \ref{NoneqPierirulehook}.
We say that a path
\[
\gamma\colon  u =w_0
\stackrel{\tau_1}\longrightarrow w_1 
\stackrel{\tau_2}\longrightarrow\cdots 
\stackrel{\tau_m}\longrightarrow w_{m}=w
\]
in the extended $k$-Bruhat order is {\it unimodal} if 
$\tau_1< \cdots <\tau_{i}>\cdots > \tau_{m}$  for some $1\leq i\leq m$. Analogously, we use  $\mathrm{in}(\gamma)=i-1$  (resp., $\mathrm{de}(\gamma)=m-i$) to denote one less than the length of the increasing  (resp., decreasing) segment of $\gamma$.

\begin{Th}\label{NoneqPierirulehook-R}
For a hook shape partition $\Gamma=(1+\alpha,1^{\beta})$, we have the following identity in $H^\bullet(\Fl(n))$:
\[
c_{\mathrm{SM}}(Y(u)^\circ)\cdot s_{\Gamma}(x_{[k]}) = \sum_{\gamma} c_{\mathrm{SM}}(Y(\mathrm{end}(\gamma))^\circ),
\]
where the sum runs  over all unimodal  paths $\gamma$ starting at $u$  in the extended $k$-Bruhat order with $\mathrm{in}(\gamma)=\alpha$ and $\mathrm{de}(\gamma)=\beta$.
\end{Th}

If we take the lowest degree component in  Theorems \ref{NoneqPierirulehook} or  \ref{NoneqPierirulehook-R}, then we recover the  formula for the product 
$[Y(u)]\cdot s_{\Gamma}(x_{[k]})$ due to Sottile  \cite[Theorem 8]{Sottile1}, where in this case the sum ranges over  peakless or unimodal  paths  in the ordinary $k$-Bruhat order.

 Theorems  \ref{NoneqPierirulehook} and   \ref{NoneqPierirulehook-R} together lead to  the following equidistribution.  

\begin{Coro}\label{peaklessieunimodal}
For   permutations $u,w\in \S_n$ and nonnegative integers $\alpha$ and $\beta$, 
the number of peakless paths from $u$ to $w$ 
in the extended $k$-Bruhat order 
with $\mathrm{in}(\gamma)=\alpha$ and $\mathrm{de}(\gamma)=\beta$ 
is equal to the number of unimodal paths from $u$ to $w$ 
in the extended $k$-Bruhat order 
with $\mathrm{in}(\gamma)=\alpha$ and $\mathrm{de}(\gamma)=\beta$.
\end{Coro}

\subsection{Equivariant Case I: Theorem  \ref{LLLL-1}}

\begin{Th}[=Theorem  \ref{LLLL-1}]\label{eqPierirulehook}
For $u\in \S_n$ and a hook shape   $\Gamma=(1+\alpha,1^{\beta})$, we have the following identity in $H_T^\bullet(\Fl(n))$:
\[
c^T_{\mathrm{SM}}(Y(u)^\circ)\cdot s_{\Gamma}(x_{[k]}) 
=s_{\Gamma}( t_{u[k]})\cdot c^T_{\mathrm{SM}}(Y(u)^\circ)+
\sum_{u\neq w\in \S_n} c_{u,\Gamma}^w( t)\cdot c_{\mathrm{SM}}^T(Y(w)^\circ),
\]
where 
\begin{equation}\label{eq:equivhook}
c_{u,\Gamma}^w( t)=
\sum_{\gamma}h_{\alpha-\mathrm{in}(\gamma)}( t_{\Cplus_k(u, w)})\cdot 
e_{\beta-\mathrm{de}(\gamma)}( t_{\Cminus_k(u, w)})
\end{equation}
with the sum taken over all peakless paths from $u$ to $w$ in 
the extended $k$-Bruhat order. 
\end{Th}

\begin{proof}
This  immediately follows from Theorem \ref{NoneqPierirulehook} and the Rigidity Theorem \ref{noneqimplieEqhook}. 
\end{proof}

Restricting $\alpha=0$ or $\beta=0$ in  Theorem \ref{eqPierirulehook}, we obtain the following Pieri  formula for equivariant CSM classes.

\begin{Coro}[Equivariant CSM Pieri Formula I]\label{Pieri-CSM-1}
Let $u\in \S_n$. 
We have the following identities in $H^\bullet(\Fl(n))$:
\begin{itemize}
    \item [(1)] For $r\geq 1$, 
\[
c^T_{\mathrm{SM}}(Y(u)^\circ)\cdot e_{r}(x_{[k]}) 
=
\sum_{\gamma} e_{r-\ell(\gamma)}(t_{\Cminus_k(u, w)})\cdot  c_{\mathrm{SM}}^T(Y(\mathrm{end}(\gamma))^\circ),
\] 
where the sum is over all  decreasing paths  from $u$ to $w$ in 
the extended $k$-Bruhat order. Here, $\ell(\gamma)$
denotes the length of $\gamma$.
    
    \item [(2)]
    For $r\geq 1$, 
\[
c^T_{\mathrm{SM}}(Y(u)^\circ)\cdot h_{r}(x_{[k]}) 
=\sum_{\gamma} h_{r-\ell(\gamma)}( t_{\Sigma_k(u, w)})\cdot  c_{\mathrm{SM}}^T(Y(\mathrm{end}(\gamma))^\circ),
\]    
where the sum is over all  increasing paths  from $u$ to $w$ in the extended $k$-Bruhat order.
    
\end{itemize}
\end{Coro}

Taking the lowest degree component of $c^T_{\mathrm{SM}}(Y(u)^\circ)$, we obtain  a formula for multiplying  an equivariant Schubert class 
by a Schur polynomial of hook shape. 

\begin{Coro}\label{EqSchubertPieri}
For a hook shape partition $\Gamma=(1+\alpha,1^{\beta})$, we have the following identity in $H_T^\bullet(\Fl(n))$:
\begin{equation}
[Y(u)]_T\cdot s_{\Gamma}(x_{[k]}) 
=s_{\Gamma}( t_{u[k]})\cdot [Y(u)]_T+
\sum_{u\neq w\in \S_n} \overline{c}_{u,\Gamma}^w( t)\cdot [Y(w)]_T,
\end{equation}
where $\overline{c}_{u,\Gamma}^w( t)$ has the same expression as $c_{u,\Gamma}^w( t)$ in \eqref{eq:equivhook}, except that now  the peakless paths are restricted to be in the ordinary $k$-Bruhat order. 
\end{Coro}

We give an example to illustrate Theorem \ref{eqPierirulehook} and Corollary \ref{EqSchubertPieri}.

\begin{Eg}\label{SSSCC}
Let $u=23154$ and $k=2$. Choose $\alpha=1$ and $\beta=1$, i.e., $\Gamma=(2,1)$. 
For simplicity,  denote \[\zeta_{w}=c^T_{\mathrm{SM}}(Y(w)^\circ)\quad \text{and} \quad \sigma_w=[Y(w)]_T.\]
Table \ref{tab:csmschub23154} lists all the peakless paths $\gamma$ in the $k$-Bruhat graph of $\S_5$ starting from $u=23154$ with $\mathrm{in}(\gamma)\leq 1$ and $\mathrm{de}(\gamma)\leq 1$, where we use dashed arrows to distinguish the $k$-edges $u\stackrel{\tau}\kto w$ with $\ell(w)>\ell(u)+1$. See also Figure  \ref{itervl} for an illustration.
\begin{table}[hhh]\center
\begin{tabular}{c|c}\hline
peakless paths & coefficients\\\hline
$23154 \stackrel{2}\kkto 53124$&
    $e_1(t_{\{3\}})h_1(t_{\{1,3,5\}})=
    t_2t_3 + t_3^2 + t_3t_5$\\
$23154 \stackrel{2}\kkto 43152$&
    $e_1(t_{\{3\}})h_1(t_{\{2,3,4\}})=
    t_2t_3 + t_3^2 + t_3t_4$\\
$23154 \stackrel{3}\kto 25134$&
    $e_1(t_{\{2\}})h_1(t_{\{2,3,5\}})=
    t_2^2 + t_2t_3 + t_2t_5$\\
$23154 \stackrel{3}\kto 24153$&
    $e_1(t_{\{2\}})h_1(t_{\{2,3,4\}})=
    t_2^2 + t_2t_3 + t_2t_4$\\
$23154 \stackrel{2}\kkto 53124 \stackrel{3}\kto 54123$&
    $e_1(t_{\varnothing})h_0(t_{\{2,3,5,4\}})=
    0$\\
$23154 \stackrel{2}\kkto 43152 \stackrel{4}\kto 53142$&
    $e_1(t_{\{3\}})h_0(t_{\{2,3,5,4\}})=
    t_3$\\
$23154 \stackrel{2}\kkto 43152 \stackrel{3}\kto 45132$&
    $e_1(t_{\varnothing})h_0(t_{\{2,3,5,4\}})=
    0$\\
$23154 \stackrel{3}\kto 25134 \stackrel{2}\kto 35124$&
    $e_0(t_{\varnothing})h_1(t_{\{2,3,5\}})=
    t_2 + t_3 + t_5$\\
$23154 \stackrel{3}\kto 25134 \stackrel{2}\kkto 45132$&
    $e_0(t_{\varnothing})h_1(t_{\{2,3,5,4\}})=
    t_2 + t_3 + t_4 + t_5$\\
$23154 \stackrel{3}\kto 24153 \stackrel{2}\kkto 54123$&
    $e_0(t_{\varnothing})h_1(t_{\{2,3,5,4\}})=
    t_2 + t_3 + t_4 + t_5$\\
$23154 \stackrel{3}\kto 24153 \stackrel{2}\kto 34152$&
    $e_0(t_{\varnothing})h_1(t_{\{2,3,4\}})=
    t_2 + t_3 + t_4$\\
$23154 \stackrel{3}\kto 24153 \stackrel{4}\kto 25143$&
    $e_1(t_{\{2\}})h_0(t_{\{2,3,5,4\}})=
    t_2$\\
$23154 \stackrel{3}\kto 25134 \stackrel{2}\kto 35124 \stackrel{3}\kto 45123$&
    $e_0(t_{\varnothing})h_0(t_{\{2,3,5,4\}})=
    1$\\
$23154 \stackrel{3}\kto 24153 \stackrel{2}\kto 34152 \stackrel{3}\kkto 54132$&
    $e_0(t_{\varnothing})h_0(t_{\{2,3,5,4\}})=
    1$\\
$23154 \stackrel{3}\kto 24153 \stackrel{2}\kto 34152 \stackrel{4}\kto 35142$&
    $e_0(t_{\varnothing})h_0(t_{\{2,3,5,4\}})=
    1$\\\hline
\end{tabular}
\caption{Computing $\zeta_{23154}\cdot s_{(2,1)}(x_1,x_2)$ and $\sigma_{23154}\cdot s_{(2,1)}(x_1,x_2)$}
\label{tab:csmschub23154}
\end{table}

By Theorem \ref{eqPierirulehook}, we have
\begin{align*}
\zeta_{23154} \cdot s_{(2,1)}(x_1,x_2)&= s_{(2,1)}(t_2,t_3)\cdot\zeta_{23154}
+(t_2t_3 + t_3^2 + t_3t_5)\cdot\zeta_{53124}\\
&\quad+(t_2t_3 + t_3^2 + t_3t_4)\cdot\zeta_{43152}
+(t_2^2 + t_2t_3 + t_2t_5)\cdot\zeta_{25134}\\
&\quad+(t_2^2 + t_2t_3 + t_2t_4)\cdot\zeta_{24153}
+t_3\cdot\zeta_{53142}
+(t_2 + t_3 + t_5)\cdot\zeta_{35124}\\
&\quad+(t_2 + t_3 + t_4 + t_5)\cdot\zeta_{45132}
+(t_2 + t_3 + t_4 + t_5)\cdot\zeta_{54123}\\
&\quad+(t_2 + t_3 + t_4)\cdot\zeta_{34152}
+t_2\cdot\zeta_{25143}
+\zeta_{45123}+\zeta_{54132}+\zeta_{35142}.
\end{align*}
By Corollary \ref{EqSchubertPieri}, all the paths in  Table \ref{tab:csmschub23154} with solid edges are what we need   in the expansion of  $\sigma_{23154} \cdot s_{(2,1)}(x_1,x_2)$, to wit, 
\begin{align*}
\sigma_{23154} \cdot s_{(2,1)}(x_1,x_2)
&= s_{(2,1)}(t_2,t_3)\cdot\sigma_{23154} 
+ (t_2^2 + t_2t_3 + t_2t_5)\cdot\sigma_{25134}\\
&\quad+(t_2^2 + t_2t_3 + t_2t_4)\cdot\sigma_{24153}
+(t_2 + t_3 + t_5)\cdot\sigma_{35124}\\
&\quad+(t_2 + t_3 + t_4)\cdot\sigma_{34152}+t_2\cdot\sigma_{25143}
+\sigma_{45123}+\sigma_{35142}.
\end{align*}
\end{Eg}

\begin{figure}[hhhh]
$$
\def\labelstyle#1{\scriptstyle\rule[-0.25pc]{0pc}{1pc}\,\,#1\,\,}
\xymatrix@!R=1.5pc@C=0pc{
&&&&{45132}&&&&\\&
{35142}
    \kar[urrr]|{{3}}&&&{45123}&&&{54132}\\
{25143}
    \kar[ur]|{{2}}
    \kar[urrrr]|{{2}}&&
{34152}
    \kar[ul]|>>>>>{{4}}
    \kkar[urrrrr]|>>>>>>>>>>>{{3}}&&
{35124}
    \kar[u]|>>>>>{{3}}&&
{53142}
    \kar[ur]|{{3}}&&{54123}\\&
{24153}
    \kar[ul]|{{4}}
    \kar[ur]|{{2}}
    \kkar[urrrrrrr]|{{2}}&&
{25134}
    \kar[ur]|{{2}}
    \kkar@/^1pc/[uuur]|<<<<<<<<<<<<{{2}}&&
{43152}
    \kar@/_1pc/[uuul]|<<<<<<<<<<<<<<{{2}}
    \kar[ur]|{{4}}&&
{53124}
    \kar[ur]|{{3}}\\
&&&&{23154}
    \kar[ulll]|{{3}}
    \kar[ul]|{{3}}
    \kkar[ur]|{{2}}
    \kkar[urrr]|{{2}}
}
$$
\caption{An illustration of Example \ref{SSSCC}.}
\label{itervl}
\end{figure}
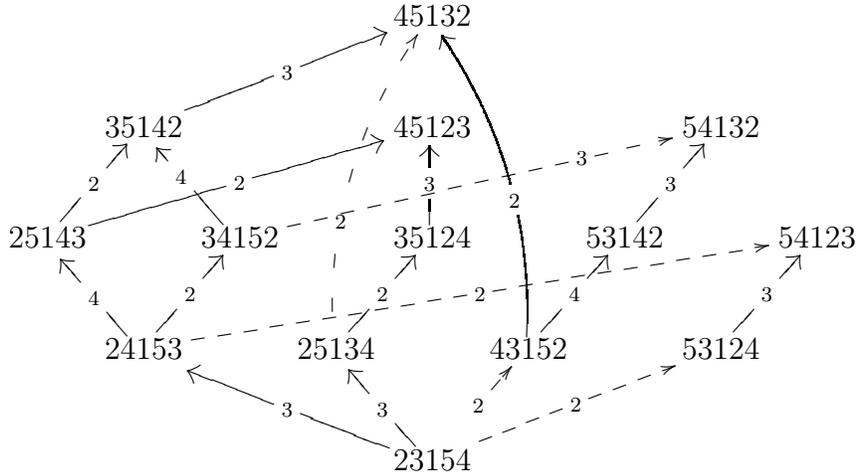


\subsection{Equivariant Case II: Theorem \ref{LUUU-1}}

Recall  the Giambelli formula for Schubert polynomials in Proposition \ref{PropertiesofSchubertpolynomials}.
In the case when the permutation is  Grassmannian   corresponding to a hook shape,  the Giambelli formula has   an   explicit  expression. 

\begin{Lemma}\label{GiambelliLemma} 
For a hook shape $\Gamma=(1+\alpha,1^{\beta})$, assume that
the Grassmannian permutation $w_{\Gamma}$   belongs to $\S_n$ and has descent at position $k$.
We have
the following identity: 
\begin{equation}\label{eq:GiambelliLemma}
[Y(w_{\Gamma})]_T = \mathfrak{S}_{w_{\Gamma}^{-1}}(- t)+\sum_{
\begin{subarray}{c}
\Gamma'=(1+\alpha',1^{\beta'})\\
 \alpha'\leq \alpha,\,
 \beta'\leq \beta
\end{subarray}} s_{\Gamma'}(x_{[k]})\cdot e_{\alpha-\alpha'}(- t_{[k+\alpha]})\cdot h_{\beta-\beta'}(- t_{[k-\beta]}). 
\end{equation}
\end{Lemma}

\begin{proof}
By  the Giambelli formula in   Proposition \ref{PropertiesofSchubertpolynomials}, 
\begin{align*} 
[Y(w_{\Gamma})]_T = \mathfrak{S}_{w_{\Gamma}}(x, t) = \sum_{\begin{subarray}{c}
w_{\Gamma}= v^{-1}  u\\
\ell(w_{\Gamma})=\ell(u)+\ell(v)
\end{subarray}}
\S_{u}(x)
\mathfrak{S}_{v}(- t).
\end{align*}
Since $w_{\Gamma}$ has descent at $k$, it is easily verified  that
\begin{equation}\label{99pp}
w_{\Gamma}=s_{k-\beta}\cdots s_{k-1}\,s_{k+\alpha}\cdots s_{k+1}s_k
\end{equation}
is a reduced word of $w_{\Gamma}$.
If $u=\operatorname{id}$, then $v^{-1}=w_\Gamma$, and in this case 
$\S_{u}(x)\S_{v}(- t)$ contributes the term $\mathfrak{S}_{w_{\Gamma}^{-1}}(- t)$.

We now consider the case when $u\neq\operatorname{id}$. By analyzing the reduced word in \eqref{99pp}, it is not hard to check that    $u$ is the product of a latter half (possibly  empty) of $s_{k-\beta}\cdots s_{k-1}$ and a latter half (cannot be empty)  of $s_{k+\alpha}\cdots s_{k+1} s_k$. That is,  there exist $0\le \alpha'\le \alpha$ and $0\le \beta'\le \beta$  such that 
\begin{equation}\label{00qq}
u = s_{k-\beta'}\cdots s_{k-1}\,s_{k+\alpha'}\cdots s_{k+1}s_k
\end{equation}
and
\begin{equation}\label{00qqq}
v^{-1}=s_{k-\beta}\cdots s_{k-\beta'-1}\,s_{k+\alpha}\cdots s_{k+\alpha'+1}.
\end{equation}
By \eqref{00qq}, we see that $u=w_{\Gamma'}$ where  $\Gamma'=(1+\alpha',1^{\beta'})$ is hook shape with  $0\le \alpha'\le \alpha$ and $0\le \beta'\le \beta$. 
By \eqref{itii} of Proposition \ref{PropertiesofSchubertpolynomials}, 
\[
\mathfrak{S}_{u}(x)= s_{\Gamma'}(x_{[k]}). 
\]

We still need to evaluate $\S_v(-t)$.
From \eqref{00qqq}, it follows that $v$
admits a factorization $v=v_1v_2$ with
\[v_1= s_{k+\alpha'+1}\cdots s_{k+\alpha},\qquad  v_2= s_{k-\beta'-1}\cdots s_{k-\beta},\] 
which clearly satisfy the condition in \eqref{itiv} of Theorem \ref{PropertiesofSchubertpolynomials}. 
Note that $v_1$ (resp., $v_2$) is a Grassmannian permutation with descent at  $k+\alpha$ (resp., $k-\beta$) corresponding to the one column partition $(1^{\alpha-\alpha'})$ (resp., the one row 
partition $(\beta-\beta')$),  whose Schubert polynomial is  $e_{\alpha-\alpha'}(x_{[k+\alpha]})$ (resp.,  $h_{\beta-\beta'}(x_{[k-\beta]})$). 
So we have 
\[
\mathfrak{S}_{v}(- t) =\mathfrak{S}_{v_1}(- t)\cdot \mathfrak{S}_{v_2}(- t)= e_{\alpha-\alpha'}(- t_{[k+\alpha]})\cdot  h_{\beta-\beta'}(- t_{[k-\beta]}). 
\]

Combining the above gives the desired identity in \eqref{eq:GiambelliLemma}.
\end{proof}

We are now ready  to complete the proof of Theorem  \ref{LUUU-1}.

\begin{Th}[=Theorem   \ref{LUUU-1}]\label{LRruleforhookshapeCSM}
Let $u\in \S_n$, and $\Gamma=(1+\alpha,1^{\beta})$ be a hook shape to which the corresponding permutation $w_\Gamma\in \S_n$  has descent at position $k$. Then we  have the following identity in $H_T^\bullet(\Fl(n))$:
\[
c_{\mathrm{SM}}^T(Y(u)^\circ)\cdot [Y(w_{\Gamma})]_T 
= [Y(w_{\Gamma})]_T|_{u}\cdot c_{\mathrm{SM}}^T(Y(u)^\circ)
+ \sum_{u\neq w\in \S_n} \mathfrak{c}_{u,\Gamma}^w( t)\cdot  c_{\mathrm{SM}}^T(Y(w)^\circ),
\]
where 
\begin{equation}\label{eq:LRruleforhookshapeCSM}
\mathfrak{c}_{u,\Gamma}^w( t) = \sum_{\gamma}\sum_{
\begin{subarray}{c}
\alpha_1+\alpha_2=\alpha-\mathrm{in}(\gamma)\\
\beta_1+\beta_2=\beta-\mathrm{de}(\gamma)
\end{subarray}} 
h_{\alpha_1}( t_{\Cplus_k(u, w)})\cdot e_{\beta_1}( t_{\Cminus_k(u, w)})\cdot 
e_{\alpha_2}(- t_{[k+\alpha]})\cdot h_{\beta_2}(- t_{[k-\beta]})
\end{equation}
with the sum over all  peakless paths from $u$ to $w$ in the extended $k$-Bruhat order. 
\end{Th}

\begin{proof}
Replacing $[Y(w_{\Gamma})]_T$ by the right-hand side of \eqref{eq:GiambelliLemma}, we have
\begin{align*}
&c_{\mathrm{SM}}^T(Y(u)^\circ)\cdot [Y(w_{\Gamma})]_T\\[5pt]
&\ =c_{\mathrm{SM}}^T(Y(u)^\circ)\cdot \mathfrak{S}_{w_{\Gamma}^{-1}}(- t)\\[5pt]
&\ \ \ \quad+\sum_{
\begin{subarray}{c}
\Gamma'=(1+\alpha',1^{\beta'})\\
 \alpha'\leq \alpha,\,
 \beta'\leq \beta
\end{subarray}} c_{\mathrm{SM}}^T(Y(u)^\circ)\cdot s_{\Gamma'}(x_{[k]})\cdot e_{\alpha-\alpha'}(- t_{[k+\alpha]})\cdot h_{\beta-\beta'}(- t_{[k-\beta]}). 
\end{align*}
By Theorem \ref{eqPierirulehook},
the coefficient of $c_{\mathrm{SM}}^T(Y(u)^\circ)$ is 
\[
\mathfrak{S}_{w_\Gamma^{-1}}(- t)+\sum_{
\begin{subarray}{c}
\Gamma'=(1+\alpha',1^{\beta'})\\
 \alpha'\leq \alpha,\,
 \beta'\leq \beta
\end{subarray}} s_{\Gamma'}( t_{u[k]})\cdot e_{\alpha-\alpha'}(- t_{[k+\alpha])}\cdot h_{\beta-\beta'}(- t_{[k-\beta]}),
\]
which, according to  Lemma \ref{GiambelliLemma} and the definition of localization map, is precisely $[Y(w_{\Gamma})]_T|_{u}$.

For $w\neq u$, using Theorem  \ref{eqPierirulehook} again,  we obtain that
\begin{align*}
\mathfrak{c}_{u,\Gamma}^w( t)& = \sum_{
\begin{subarray}{c}
\Gamma'=(1+\alpha',1^{\beta'})\\
 \alpha'\leq \alpha,\,
 \beta'\leq \beta
\end{subarray}} c_{u,\Gamma'}^w( t)\cdot e_{\alpha-\alpha'}(- t_{[k+\alpha]})\cdot h_{\beta-\beta'}(- t_{[k-\beta]})\\
&=\sum_{
\begin{subarray}{c}
\Gamma'=(1+\alpha',1^{\beta'})\\
 \alpha'\leq \alpha,\,
 \beta'\leq \beta
\end{subarray}}\left(\sum_{\gamma}h_{\alpha'-\mathrm{in}(\gamma)}( t_{\Cplus_k(u, w)})\cdot 
e_{\beta'-\mathrm{de}(\gamma)}( t_{\Cminus_k(u, w)})\right)\\
&\hspace{3cm}\times e_{\alpha-\alpha'}(- t_{[k+\alpha]})\cdot h_{\beta-\beta'}(- t_{[k-\beta]}),
\end{align*}
which, after exchanging indices, coincides with \eqref{eq:LRruleforhookshapeCSM}.
\end{proof}

Taking the lowest degree part of   CSM classes leads to a formula for the multiplication of an equivariant Schubert class 
by an equivariant  Schubert class of hook shape.  

\begin{Coro}\label{LRrulehookshapeSchubert}
Adopting the notation in Theorem  \ref{LRruleforhookshapeCSM},
we  have
\begin{equation}
[Y(u)]_T\cdot [Y(w_{\Gamma})]_T 
= [Y(w_{\Gamma})]_T|_{u}\cdot [Y(u)]_T
+ \sum_{w} \overline{\mathfrak{c}}_{u,\Gamma}^w( t)\cdot  [Y(w)]_T,
\end{equation}
where $\overline{\mathfrak{c}}_{u,\Gamma}^w( t)$ has the same expression as $\mathfrak{c}_{u,\Gamma}^w( t)$ in 
\eqref{eq:LRruleforhookshapeCSM}, except that now we assume   the peakless paths are restricted to be in the ordinary 
$k$-Bruhat order. 
\end{Coro}

When $\alpha=0$ or $\beta=0$  in Theorem \ref{LRruleforhookshapeCSM},  we obtain our second Pieri   formula  for
equivariant CSM classes. In this case, the expression in \eqref{eq:LRruleforhookshapeCSM} for the structure constants could be further simplified. 
For $r\geq 0$, let
\begin{gather*}
c[k,r]=s_{k-r+1}\cdots s_k \ \ \ \text{and}\ \ \ 
c'[k,r]=s_{k+r-1}\cdots s_k.
\end{gather*}
Note that  $c[k,r]$ (resp., $c'[k,r]$) is the Grassmannian permutation corresponding to the one column partition $(1^r)$ (resp., the one row partition $(r)$) with descent at $k$.

\begin{Lemma}\label{Doubleeflemma} 
For  $r\geq 0$,  we have
\[
[Y(c[k,r])]_T
=\sum_{i+j=r}e_i(x_{[k]})\cdot h_j(- t_{[k-r+1]}) 
= \sum_{1\leq i_1<\cdots<i_r\leq k}\prod_{j=1}^r (x_{i_j}- t_{i_j-j+1})
\]
and 
\[
[Y(c'[k,r])]_T
= \sum_{i+j=r}h_i(x_{[k]})\cdot e_j(- t_{[k+r-1]}) 
= \sum_{1\leq i_1\leq \cdots\leq i_r\leq k}\prod_{j=1}^r (x_{i_j}- t_{i_j+j-1}). 
\]
\end{Lemma}
\begin{proof}
The first equality in each expression is the special case  of $\alpha=0$ or $\beta=0$ in Lemma \ref{GiambelliLemma}. The second equality can be found in \cite[Equations (1.2) and (1.3)]{Molev}. 
\end{proof}

It turns out that the structure constants in our second Pieri formula  may be characterized by the localization of Schubert classes. 

\begin{Th}[Equivariant CSM Pieri Formula II]\label{LRruleforehCSM}
Let $u\in \S_n$, and let $c[k,r]$ and $c'[k,r]$
be permutations in $\S_n$ with descent at position $k$.
We  have the following identities in $H_T^\bullet(\Fl(n))$:
\begin{enumerate}[\rm(i)]
    \item \label{item:LRruleforehCSMi}
    For  $r\geq 0$,  
    \[c_{\mathrm{SM}}^T(Y(u)^\circ) \cdot [Y(c[k,r])]_T = \sum_{w\in \S_n}\mathfrak{c}^w_{u,c[k,r]}( t) \cdot  c_{\mathrm{SM}}^T(Y(w)^\circ),\]
    where the sum ranges  over all $w\in \S_n$ such that there exists a decreasing   path  of length $r' \leq r$  from $u$ to $w$ in the extended $k$-Bruhat order, and 
\begin{equation}\label{eq:frakcuw}
\mathfrak{c}^w_{u,c[k,r]}( t) 
=\left. [Y(c[k-r',r-r'])]_T\right|_{\delta(u, w)}.
\end{equation}
Here, $\delta(u, w)$ can be  taken as any permutation in $\S_n$ such that  the image set of $[k-r']$ is exactly 
$\Cminus_k(u, w)=\{u(i)\colon i\in [k]\}  \setminus \{u(i)\colon u(i)\neq w(i)\}$.

\item \label{item:LRruleforehCSMii} For  $r\geq 0$,  
\[c_{\mathrm{SM}}^T(Y(u)^\circ) \cdot [Y(c'[k,r])]_T = \sum_{w\in \S_n}\mathfrak{c}^w_{u,c'[k,r]}( t) \cdot c_{\mathrm{SM}}^T(Y(w)^\circ),\]
    where the sum ranges  over all $w\in \S_n$ such that there exists an increasing   path  of length $r' \leq r$ from $u$ to $w$ in the extended $k$-Bruhat order, and  \begin{equation}\label{eq:frakcuw-h}
        \mathfrak{c}^w_{u,c'[k,r]}( t) 
=\left. [Y(c'[k+r',r-r'])]_T\right|_{\sigma(u,w)}.
    \end{equation}
  Here, $\sigma(u, w)$ can be taken as any permutation 
 in $\S_n$ such that  the image set of $[k+r']$ is exactly 
$\Cplus_k(u, w)=\{u(i)\colon i\in [k]\}  \cup \{u(i)\colon u(i)\neq w(i)\}$.   
\end{enumerate}

\end{Th}

\begin{proof}
We give a proof of (\ref{item:LRruleforehCSMi}), and the arguments of (\ref{item:LRruleforehCSMii}) can be carried out similarly. 
Let   $w\in \S_n$ be such that there is a decreasing path from $u$ to $w$ of length $r' \leq r$ in the extended $k$-Bruhat order. Notice that $c[k,r]$ corresponds to the one column hook shape $(1+\alpha, 1^\beta)$ with $\alpha=0$ and $\beta=r-1$.  Hence, by Theorem \ref{LRruleforhookshapeCSM},
if $w=u$, then $r'=0$ and  the coefficient of $c_{\mathrm{SM}}^T(Y(u)^\circ)$ is
\begin{equation}\label{JH-1}
\mathfrak{c}_{u,c[k,r]}^{u}( t)=[Y(c[k,r])]_T|_{u},
\end{equation}
and if $w\neq u$, then $r'>0$ and  the coefficient of
$c_{\mathrm{SM}}^T(Y(w)^\circ)$ is \begin{equation}\label{JH-2}
\mathfrak{c}_{u,c[k,r]}^{w}( t)  = \sum_{\beta_1+\beta_2=r-r'} e_{\beta_1}( t_{\Cminus_k(u, w)})\cdot h_{\beta_2}(- t_{[k-r+1]}).
\end{equation}

On the other hand, by the first equality in  Lemma \ref{Doubleeflemma}, we have
\[
[Y(c[k-r',r-r'])]_T=\sum_{i+j=r-r'}e_i(x_{[k-r']})\cdot h_j(- t_{[k-r+1]}). 
\]
By Remark \ref{REM-1}, we know that $\texttt{\#}\Cminus_k(u, w)=k-r'$. 
Moreover,  it follows from  Lemma \ref{Doubleeflemma} that
$[Y(c[k-r',r-r'])]_T$ is symmetric in $x_1,\ldots,x_{k-r'}$. Therefore, for any   chosen permutation $\delta(u,w)\in\S_n$ such that  its image set of $[k-r']$ is 
$\Delta_k(u, w)$, the localization 
\[
\left.[Y(c[k-r',r-r'])]_T\right|_{\delta(u,w)}
\]
equals \eqref{JH-1} when $w=u$, and
equals \eqref{JH-2} when $w\neq u$. This concludes the proof.
\end{proof}

\begin{Rmk}\label{REM-I}
In the above proof, we only used the first equality in Lemma \ref{Doubleeflemma}. Applying the second equality, we can get an explicit formula for $\mathfrak{c}^w_{u,c[k,r]}( t)$. 
Precisely, suppose that 
\[\Delta_k(u,w)=\left\{\delta_1<\delta_2<\cdots<\delta_{k-r'}\right\}.\]
Take $\delta(u, w)$ as the permutation in $\S_n$
sending $i$ to $\delta_i$ for $1\leq i\leq k-r'$, and
fixing all the remaining elements in $[n]\setminus [k-r']$. By means  of the second equality in Lemma \ref{Doubleeflemma}, the coefficient $\mathfrak{c}^w_{u,c[k,r]}( t)$
in \eqref{eq:frakcuw} becomes 
\[
\mathfrak{c}^w_{u,c[k,r]}( t) 
= \sum_{1\leq i_1<\cdots<i_{r-r'}\leq k-r'}\prod_{j=1}^{r-r'} ( t_{\delta_{i_j}}- t_{i_j-j+1}).
\]
Similarly, if we assume 
that
\[
\Sigma_k(u,w)=\{\sigma_1<\sigma_2<\cdots<\sigma_{k+r'}\},
\]
then  the coefficient   $\mathfrak{c}^w_{u,c'[k,r]}( t)$ in \eqref{eq:frakcuw-h}  
can be read as
\[
\mathfrak{c}^w_{u,c'[k,r]}( t) 
= \sum_{1\leq i_1\leq \cdots\leq i_{r-r'}\leq k+r'}\prod_{j=1}^{r-r'} ( t_{\sigma_{i_j}}- t_{i_j+j-1}).
\]
If restricting the decreasing/increasing  path from $u$ to $w$ to be in the ordinary $k$-Bruhat order in Theorem \ref{LRruleforehCSM}, 
we recover the Pieri formula for equivariant Schubert classes, as established by Robinson \cite{Robinson}, see also Li,   Ravikumar,  Sottile  and  Yang  \cite{LSY}.  
\end{Rmk}

\section{Murnaghan--Nakayama Type Rules}\label{Sect5}

Our goal in this section is to establish the MN   formula  for   equivariant CSM classes, 
as described in Theorem \ref{UUU-1}.  
To do this, we begin by deducing a MN formula
for nonequivariant CSM classes in Theorem \ref{noneqMNruleCSM}. Then we derive  the  Rigidity Theorem for power sum symmetric functions in Theorem \ref{noneqimplieEqhookMN}, which together with Theorem \ref{noneqMNruleCSM}  
completes the proof of Theorem \ref{UUU-1}.

\subsection{Nonequivariant CSM MN Formula}

For a permutation $w\in \S_n$, define  its {\it $k$-height} to be
 one less than the number of non-fixed points of $w$ 
at the first $k$ positions, namely,
\begin{equation}\label{L-1}
\operatorname{ht}_k(w)= \texttt{\#}\{i\leq k\colon w(i)\neq i\}-1.
\end{equation}

Our MN formula for nonequivariant CSM classes 
can be stated as follows.

\begin{Th}[CSM MN Formula]\label{noneqMNruleCSM}
Let $u\in \S_n$. For $r\geq 1$, we have the following identity in $H^\bullet(\Fl(n))$:
\begin{gather}
c_{\mathrm{SM}}(Y(u)^\circ)\cdot p_{r}(x_{[k]})
=\sum_{\eta\in \S_n} (-1)^{\operatorname{ht}_k(\eta)}\cdot  c_{\mathrm{SM}}(Y(u\eta)^\circ),
\end{gather}
where the sum ranges over all $(r+1)$-cycles $\eta\in \S_n$ such that  $u\leq_k u\eta$ in the extended $k$-Bruhat order.
\end{Th}

Theorem \ref{noneqMNruleCSM} is a direct  consequence of 
Theorem \ref{QQQ_I} and  Theorem  \ref{mutiplicitylemma}.

\begin{Th}\label{QQQ_I}
Let $u\in \S_n$. For $r\geq 1$, suppose that 
\begin{equation*}
c_{\mathrm{SM}}(Y(u)^\circ)\cdot p_r(x_{[k]})
=\sum_{w\in \S_n}d^w_{u, r}\cdot  c_{\mathrm{SM}}(Y(w)^\circ).
\end{equation*}
Then, 
\begin{equation}\label{eq:noneqMNruleCSMproof-R}
d^w_{u, r}
=\sum_{\gamma} (-1)^{\mathrm{de}(\gamma)}, 
\end{equation}
where the sum runs over all unimodal paths  $\gamma$  of length $r$ from $u$ to $w$ in the extended $k$-Bruhat order. 
Moreover, if $d^w_{u, r}\neq 0$, then  $ w=u\eta$ for some ($m+1$)-cycle $\eta$, where $m\geq 1$. 
\end{Th}

\begin{proof}
To prove \eqref{eq:noneqMNruleCSMproof-R},
we need the following expansion 
\begin{equation}\label{LLL-1}
p_r(x_{[k]})=\sum_{\alpha+\beta+1=r} (-1)^\beta s_{(1+\alpha,1^{\beta})}(x_{[k]}),   
\end{equation}
which is a special case of   \cite[Theorem 7.17.3]{stanley2}.
Applying the Pieri formula in Theorem \ref{NoneqPierirulehook-R}
to the right-hand side of \eqref{LLL-1}, we obtain that
\begin{equation*}\label{eq:noneqMNruleCSMproof}
c_{\mathrm{SM}}(Y(u)^\circ)\cdot p_r(x_{[k]})
=\sum_{\gamma} (-1)^{\mathrm{de}(\gamma)} c_{\mathrm{SM}}(Y(\mathrm{end}(\gamma))^\circ) 
\end{equation*}
with  $\gamma$  running  over all unimodal paths of length $r$
starting at $u$ in the extended $k$-Bruhat order, 
which leads to \eqref{eq:noneqMNruleCSMproof-R}.

Let $w\in \S_n$ be  such that $d^w_{u, r}\neq 0$.
In this case,    we see from \eqref{eq:noneqMNruleCSMproof-R}  that $\ell(w)>\ell(u)$.
By Lemma \ref{PP-1}, $d^w_{u, r}\neq 0$ is equivalent to 
\[
\mathcal{T}_{w/u}\left(p_r(x_{[k]}\right)\big|_{x_i=0}\neq 0.
\] 
We aim to show that $w=u\eta$ for some $(m+1)$-cycle $\eta$.

For a subset $A$ of $[n]$,    write  
\begin{equation}
p(x_A) =\sum_{a\in A}\frac{x_a}{1-\z x_a} = 
p_1(x_A) +\z p_2(x_A) +\cdots .
\end{equation}
Recall that 
\[
\mathcal{T}_{w/ u} = \sum_{\begin{subarray}{c}
J\subseteq [\ell]\\ w_J=u\end{subarray}}\nabla_J. 
\]
From the proof of Lemma \ref{NonimplieEquLemma},   it can be seen that  
\[
\nabla_J=u\,\partial_{a_1b_1}\cdots \partial_{a_mb_m},
\]
and moreover,
\begin{equation}\label{87uu}
w=u\,t_{a_1b_1}\cdots t_{a_mb_m}.
\end{equation}
Since $\ell(w)>\ell(u)$, we have $m\geq 1$. 

By direct computation,  it is routine to check that
\begin{equation}\label{FFF-4}
\partial_{ab}\,p(x_A) = 
(\delta_{a\in A}\delta_{b\notin A}-\delta_{b\in A}\delta_{a\notin A}) \cdot \frac{1}{(1-\z x_a)(1-\z x_{b})}.
\end{equation}
We have the following claim:

\begin{Claim}
For $m\geq 1$, if
\begin{equation}\label{eq:erelation2}
\partial_{a_1b_1}\cdots \partial_{a_mb_m} p(x_{A}) \neq 0,
\end{equation}
then the set $\{a_1, b_1,\ldots, a_m, b_m\}$ has cardinality $m+1$, and $t_{a_1b_1}\cdots t_{a_mb_m}$ forms an $(m+1)$-cycle on 
  $\{a_1, b_1,\ldots, a_m, b_m\}$. 
\end{Claim}

We verify the above Claim by induction on $m$. The case $m=1$ is obvious. Assume now that $m\geq 2$. 
By the assumption that $\partial_{a_1b_1}\cdots \partial_{a_mb_m} p(x_{A}) \neq 0$, we have  $\partial_{a_2b_2}\cdots \partial_{a_mb_m} p(x_{A})\neq 0$, and thus by induction,
the set $B=\{a_2, b_2,\ldots, a_m, b_m\}$ has cardinality $m$, and $t_{a_2b_2}\cdots t_{a_mb_m}$ forms an $m$-cycle on   $B$.
By \eqref{FFF-4} and Lemma \ref{QVrelation1}, $$\partial_{a_2b_2}\cdots \partial_{a_mb_m} p(x_{A})
\in \mathbb{Q}[\z]
\left(\partial_{a_2b_2}\cdots \partial_{a_{m-1}b_{m-1}} \frac{1}{\Z(x_{\{a_r,b_r\}})}\right)
\subseteq \mathbb{Q}[\z]\frac{1}{\Z(x_{B})}.$$
To ensure that $\partial_{a_1b_1}(\partial_{a_2b_2}\cdots \partial_{a_mb_m} p(x_{A}))\neq 0$, 
we necessarily have
$\partial_{a_1b_1}\frac{1}{\Z(x_{B})}\neq 0$, 
which requires  that
$\texttt{\#}\, \{a_1,b_1\}\cap B=1$ by Lemma \ref{QVrelation1}.
So  $\{a_1, b_1,\ldots, a_m, b_m\}$ contains $m+1$
elements. 
 Since $t_{a_2b_2}\cdots t_{a_mb_m}$ forms an  $m$-cycle on  $B$, it is easily checked that $t_{a_1b_1}\cdots t_{a_mb_m}$ forms  an $(m+1)$-cycle on  $\{a_1, b_1,\ldots, a_m, b_m\}$. 
 
Recall that if $d^w_{u, r}\neq 0$, then $\mathcal{T}_{w/ u}(p(x_A))\neq 0$, implying that $\partial_{a_1b_1}\cdots \partial_{a_mb_m} p(x_{A}) \neq 0$ for some $m\geq 1$. Together with the above Claim and \eqref{87uu}, there must exist some $m\geq 1$ such that 
$w=u\eta$ with $\eta=t_{a_1b_1}\cdots t_{a_mb_m}$
an $(m+1)$-cycle.
\end{proof}

\begin{Th}\label{mutiplicitylemma}
Let $u\in \S_n$, and    $\eta\in \S_n$ be an $(m+1)$-cycle ($m\geq 1$) such that $u\leq_k u\eta$ in the extended $k$-Bruhat order.
Then there exists a unique unimodal path $\gamma$  from $u$ to $u\eta$ in the extended $k$-Bruhat order. 
Moreover,  $\gamma$ is a path of length $m$ and with $\mathrm{de}(\gamma)=\operatorname{ht}_k(\eta)$. 
\end{Th}

We remark that an analogous  statement to Theorem  \ref{mutiplicitylemma} for the Grassmannian Bruhat order   appeared  in \cite[\textsection 6.2]{Sottile}.
The proof of Theorem \ref{mutiplicitylemma}
is quite technical, and will occupy the whole bulk of 
Subsection \ref{EEE-I}.

\subsection{Proof of Theorem \ref{mutiplicitylemma}}\label{EEE-I}

To prove Theorem \ref{mutiplicitylemma}, 
we need several  lemmas. The first lemma is a  nonrecursive criterion for the extended $k$-Bruhat order, which is quite useful in the comparison of two permutations in the extended $k$-Bruhat order. An analogous criterion for the ordinary $k$-Bruhat order has appeared in   \cite[Theorem A]{BS}.

\begin{Lemma}\label{extkBruhatorder}
For two permutations $u, w\in \S_n$, $u\leq_k w$ in the extended $k$-Bruhat order if and only if for any $a\leq k<b$, $u(a)\leq w(a)$ and $u(b)\geq w(b)$. 
\end{Lemma}

\begin{proof}
The proof is  along a similar line to that of  \cite[Theorem A]{BS}. 
Denote by $\leq'_k$   the relation defined by the condition: for $a\leq k<b$, $u(a)\leq w(a)$ and $u(b)\geq w(b)$. Clearly, $\leq'_k$ is a partial order on $\S_n$. 
We need to verify that $\leq_k$ and $\leq'_k$ are the same partial order on $\S_n$.

Assume that $u\longrightarrow w$, that is, $w=ut_{ab}$ with $a\leq k<b$, and $u(a)<u(b)$. It is obvious  that $u\leq'_k w$. So, if $u\longrightarrow u_1\longrightarrow\cdots\longrightarrow w$, then $u\leq'_k u_1\leq'_k\cdots\leq'_k w$. This implies that if $u\leq_k w$, then $u\leq'_k w$.
 
Conversely, assume that $u\leq'_k w$. 
We use induction on $\ell(w)$ to prove $u\leq_k w$.
This is clear  in the case $\ell(w)=0$ since both $u$ and $w$ are the identity permutation. We now consider  $\ell(w)>0$ and $u\neq w$. 
To apply induction, we   show that there exist  $a\leq k<b$ such that   $wt_{ab}\longrightarrow w$ and  $u\leq_k' wt_{ab}$.

Choose the index  $a\leq k$ such that $u(a)$ is minimal 
subject to
$u(a)<w(a)$.
Such a choice exists since otherwise $u(i)\geq w(i)$ for all $i\leq k$ would lead to $u=w$. Once $a$ is chosen,  
locate any  index $b>k$ such that
\[
w(b)<w(a)\leq u(b).
\]
We explain that the above  choice of $b$ exists.
Suppose to the contrary that there does not exist such an index $b$. 
We claim that for $i\in [n]$, if $w(i)<w(a)$, then $u(i)<w(a)$. 
In fact, if $i>k$ and  $w(i)<w(a)$, then we must have $u(i)<w(a)$, since otherwise $u(i)\ge w(a)$ would imply the existence of an index $b$.  If $i\le k$ and $w(i)<w(a)$, then the claim clearly  holds since $u(i)\le w(i)<w(a)$.   By this claim, we see that 
\[\{u(i)\colon i\in [n], \   w(i)<w(a)\}=[w(a)-1]=\{1,2,\ldots, w(a)-1\},\]
which contradicts $u(a)<w(a)$ since $a\not\in \{i\in [n]\colon w(i)<w(a)\}$. 
So the assumption is false, that is, such an index $b$ always exists.


It is obvious that $wt_{ab}\longrightarrow w$ since $w(b)<w(a)$. 
We proceed to  show that $u\le_{k}' wt_{ab}$. By the choices of $a$ and $b$, it suffices to check that
\[ 
u(a)\leq w(b).
\]
Still, we use contradiction. Suppose otherwise that $w(b)<u(a)$. 
We shall construct  an infinite sequence $b_1,b_2,\ldots$ such that 
\[u(a)>u(b_1)>u(b_2)>\cdots,\]
which is absurd.

Let $b_1$ be the position such that  $u(b_1)=w(b).$  Then $u(a)>w(b)=u(b_1)$.  
Let   $b_2$  be such that   $u(b_2)=w(b_1).$
Since $u(a)>w(b)$ and $u(b)>w(b)$, it follows that $b_1\neq a, b$, and 
so $u(b_1)=w(b)\neq w(b_1)$. By the minimality of $u(a)$, it follows that $b_{1}>k$ and $u(b_1)>w(b_{1})=u(b_2).$ 
Choose the index  $b_3$ by letting 
$u(b_3)=w(b_2)$. 
Using the same analysis, we may deduce that $b_2>k$ and $u(b_2)>w(b_2)=u(b_3)$. Continuing this procedure,
we are led to 
$u(a)>u(b_1)>u(b_2)>\cdots$, a contradiction.


Applying   induction, we get 
$u\leq_k wt_{ab}$, which along with $wt_{ab}\leq_k w$ gives $u\leq_k w$.  
\end{proof}

\begin{Coro}\label{pathmovingset}
For $u, w\in \S_n$, assume that 
\[\gamma\colon u=w_0\longrightarrow w_1\longrightarrow \cdots \longrightarrow w_m=w\]
  is a path in the extended $k$-Bruhat order
with $w_i=w_{i-1}t_{a_ib_i}$ for $1\leq i\leq m$. 
Then 
\begin{equation}\label{EQ-1}
\{a_1,b_1,\ldots, a_m, b_m\} = M(u^{-1}w)=\{i\in [n]\colon u(i)\neq w(i)\}.
\end{equation}
Precisely,  by Lemma \ref{extkBruhatorder},
\begin{equation*}\label{EQ-2}
\{a_1,\ldots, a_m\}=\{a\in [n]\colon u(a)<w(a)\}
\end{equation*}
and 
\begin{equation*}\label{EQ-3}
\{b_1,\ldots, b_m\}=\{b\in [n]\colon u(b)>w(b)\}. 
\end{equation*}
\end{Coro}
\begin{proof}
It is clear that $M(u^{-1}w)\subseteq \{a_1,b_1,\ldots, a_m, b_m\}$.
It remains to check the reverse inclusion.
For   $1\leq i\leq m$, we have 
$u\leq_k w_{i-1}\longrightarrow w_{i}\leq_k w$. Combining  
this with Lemma \ref{extkBruhatorder} gives
\begin{align*}
&u(a_i)\leq w_{i-1}(a_i)<w_{i}(a_i)\leq w(a_i),\\[5pt]
&u(b_i)\geq w_{i-1}(b_i)>w_{i}(b_i)\geq w(b_i). 
\end{align*}
This implies  that both $a_i$ and $b_i$ are not fixed points of $u^{-1}w$, and thus they belong to $M(u^{-1}w)$, as required. 
\end{proof}

\begin{Lemma}\label{mutiplicityllemma1} 
For  $u\neq w\in \S_n$ and any path 
\[\gamma\colon u=w_0\stackrel{\tau_1}\longrightarrow w_1\stackrel{\tau_2}\longrightarrow \cdots \stackrel{\tau_m}\longrightarrow w_m=w\]
from $u$ to $w$ in the extended $k$-Bruhat order, we have
\begin{equation}\label{T-4}
\min\{u(i)\colon u(i)\neq w(i)\}=\min\{\tau_1,\ldots,\tau_m\}.
\end{equation}
Moreover, letting $a\in M(u^{-1}w)$ be the index exactly attaining the minimum value in \eqref{T-4}, we have $a\leq k$.
\end{Lemma}

{
\begin{proof}
Let $a$ be the index attaining the minimum value on the left-hand side of \eqref{T-4}. 
We first verify that $a\leq k$. Suppose to the contrary that 
$a>k$. By Lemma \ref{extkBruhatorder},
\begin{equation}\label{XXPPP-1}
u(a)\geq w(a)=u(u^{-1}w(a)).
\end{equation}
Since  $a$ is not a fixed point of $u^{-1}w$, $u^{-1}w(a)$ is also not a fixed point of $u^{-1}w$. So we have $u(a)\leq u(u^{-1}w(a))$, which together with \eqref{XXPPP-1} implies that $a=u^{-1}w(a)$, contrary to the fact that $u(a)\neq w(a)$. This concludes that $a\leq k$.

It remains to show that $u(a)=\min\{\tau_1,\ldots,\tau_m\}$. 
Let $i$ be the minimum index such that $u(a)\neq w_i(a)$. Assume that 
$$w_{i-1}\stackrel{\tau_i}\longrightarrow w_i=w_{i-1}t_{a_ib_i}$$
for some $a_i\leq k<b_i$. 
Since $i$ is minimum  and $a\leq k$, we have $a_i=a$ and $\tau_i=w_{i-1}(a)=u(a)$, meaning that $u(a)$ appears in the labels $\tau_1,\ldots, \tau_m$.
If there were some $j$ such that $\tau_{j}<u(a)$, say
$$w_{j-1}\stackrel{\tau_{j}}\longrightarrow w_{j}=w_{j-1}t_{a_jb_j} $$
for some $a_j\leq k<b_j$,  
we would deduce from Lemma \ref{extkBruhatorder} that  $u(a_j)\leq w_{j-1}(a_j)=\tau_{j}<u(a)$. 
However, it follows from  Corollary \ref{pathmovingset} that $a_j\in M(u^{-1}w)$,  which is contrary to the choice of $a$. This completes the proof.
\end{proof}
}



To give a proof of Theorem \ref{mutiplicitylemma}, we shall 
construct a unimodal path from $u$ to $u\eta$ in the extended $k$-Bruhat order. To do this, we first determine the edge with the minimum label (namely, the first or the last edge) in the unimodal path, as will be done in Lemmas \ref{mutiplicityllemma3} and 
\ref{mutiplicityllemma3-2}.
 
\begin{Lemma}\label{mutiplicityllemma3}
Let $u\in \S_n$, and $\eta\in \S_n$ be an $(m+1)$-cycle with $m\geq 2$ such that $u\leq_k u\eta$. Assume that $a\in M(\eta)$ is the index attaining the  minimum value $\min\{u(i)\colon i\in M(\eta)\}$.
We have the following equivalent statements:
\begin{equation}\label{I-1}
u(\eta^{-1}(a))<u(\eta(a))\iff u\stackrel{u(a)}\longrightarrow ut_{a\,\eta^{-1}(a)}\leq_k u\eta,
\end{equation}
where $a\leq k<\eta^{-1}(a)$.
Moreover, if we write $ut_{a\,\eta^{-1}(a)}=u'$ and $u\eta=u'\eta'$,   then $\eta'$ is an $m$-cycle with $\operatorname{ht}_k(\eta')=\operatorname{ht}_k(\eta)$, and 
$u'M(\eta')=uM(\eta)\setminus \{u(a)\}$. 
\end{Lemma}

\begin{proof}
Since $m\geq 2$, we have $\eta(a)\neq \eta^{-1}(a)$, and both $\eta(a)$ and $\eta^{-1}(a)$
belong to $M(\eta)$. Let us proceed to check that  $a\leq k<\eta^{-1}(a)$.  Since $u\leq_k u\eta$, by Lemma  \ref{mutiplicityllemma1}, we have $a\leq k$. Suppose otherwise that $\eta^{-1}(a)\leq k$. By Lemma \ref{extkBruhatorder},
we deduce that 
\[u(a)=u\eta(\eta^{-1}(a))\geq u(\eta^{-1}(a)),\]
which, along with the choice of $a$, yields the contradiction that $a=\eta^{-1}(a)$. 
This verifies   $\eta^{-1}(a)>k$.

We next prove the equivalence in \eqref{I-1}.
By Lemma \ref{extkBruhatorder}, the right-hand side of \eqref{I-1} 
is equivalently saying that 
\begin{equation}\label{I-2}
u(a)<u(\eta^{-1}(a))\ \ \ \ \ \  (\iff u\stackrel{u(a)}\longrightarrow ut_{a\,\eta^{-1}(a)})
\end{equation}
and 
\begin{equation}\label{I-3}
u(\eta^{-1}(a))\leq u\eta(a) \ \ \text{and} \ \ u(a)\geq u\eta(\eta^{-1}(a))=u(a)\ \ \ (\iff ut_{a\,\eta^{-1}(a)}\leq u\eta).    
\end{equation}
The choice of $a$ directly implies \eqref{I-2}.
Since $\eta^{-1}(a)\neq \eta(a)$,  condition \eqref{I-3} is the same as the left-hand side of \eqref{I-1}. This concludes  \eqref{I-1}.  
 
The fact that $\eta'=t_{a \,\eta^{-1}(a)}\eta$ is an $m$-cycle follows from direct computation. 
Actually, $\eta'$ is obtained from  $\eta$ by removing   the value $\eta^{-1}(a)$. Since $\eta^{-1}(a)>k$, recalling the definition in \eqref{L-1}, we obtain that  $\operatorname{ht}_k(\eta')=\operatorname{ht}_k(\eta)$.  Moreover, it is easily checked  that
$$u'M(\eta')=u'(M(\eta)\setminus \{\eta^{-1}(a)\})
=u'M(\eta)\setminus \{u(a)\}=uM(\eta)\setminus \{u(a)\}.$$
This completes the proof. 
\end{proof}

A dual statement to Lemma \ref{mutiplicityllemma3} can be derived  by  similar arguments.

\begin{Lemma}\label{mutiplicityllemma3-2}
Let $u$, $\eta$, and $a$ be as given in Lemma \ref{mutiplicityllemma3}. 
We have the following equivalent statements:
\begin{equation}\label{AAEESS}
u(\eta^{-1}(a))>u(\eta(a))\iff u\leq_k u\eta t_{a\,\eta^{-1}(a)} \stackrel{u(a)}\longrightarrow u\eta,
\end{equation}
where $a\leq k<\eta^{-1}(a)$.
Moreover, if we write 
$u=u'$ and $u\eta t_{a\,\eta^{-1}(a)}=u'\eta'$, then
$\operatorname{ht}_k(\eta')=\operatorname{ht}_k(\eta)-1$, and 
$u'M(\eta')=uM(\eta)\setminus \{u(a)\}$. 
\end{Lemma}

The last two lemmas will be used to prove the uniqueness of the unimodal path from 
$u$ to $u\eta$  in the extended $k$-Bruhat order.

\begin{Lemma}\label{mutiplicityllemma2}
Let $u$, $\eta$, and $a$ be as given in Lemma \ref{mutiplicityllemma3}. 
Assume that {  $u(a)$ appears as the first label of a path  }
    \[\gamma\colon u\stackrel{u(a)}\longrightarrow ut_{ab}\leq_k  u\eta \]
     from $u$ to $u\eta$ in the extended $k$-Bruhat order. Then we have $b=\eta^{-1}(a)$.
{ Moreover, we have $u(\eta^{-1}(a))<u(\eta(a))$ by \eqref{I-1}.}
\end{Lemma}

\begin{proof}
By Lemma \ref{extkBruhatorder}, we have
\begin{equation}\label{P-1}
u(a)=ut_{ab}(b)\geq u\eta(b)=u(\eta(b)).
\end{equation}
On the other hand, by Corollary \ref{pathmovingset}, we have  $b\in   M(\eta)$, namely, $\eta(b)\neq b$,  implying $\eta(b)\in M(\eta)$. Together with  \eqref{P-1} and the choice of $a$, we are given   $\eta(b)=a$. 
\end{proof}

Using similar analysis to Lemma \ref{mutiplicityllemma2}, we obtain the following dual assertion. 

\begin{Lemma}\label{mutiplicityllemma2-2}
Let $u$, $\eta$, and $a$ be as given in Lemma \ref{mutiplicityllemma3}. 
Assume that {  $u(a)$ appears as the last label of a path }
\[\gamma\colon u\leq_k u\eta t_{a'b} \stackrel{u(a)}\longrightarrow u\eta\]
  from $u$ to $u\eta$ in the extended $k$-Bruhat order.
    Then we have $a'=a$ and $b=\eta^{-1}(a)$.
 Moreover, we have $u(\eta^{-1}(a))>u(\eta(a))$ 
 by \eqref{AAEESS}.
\end{Lemma}

\begin{proof}
Since $u(a)$ appears on the last edge, we have $u\eta t_{a'b}(a')=u\eta(b)=u(a)$,  and so $b=\eta^{-1}(a)$. We next verify  $a'=a$.
By  Lemma \ref{extkBruhatorder},
\begin{equation}\label{P-2}
u(a')\leq u\eta t_{a'b}(a')=u\eta(b)=u(a).
\end{equation}
By Corollary \ref{pathmovingset}, we see that  $a'\in M(\eta)$, which, together with \eqref{P-2} and the choice  of $a$, forces  $a'=a$, as desired. 
\end{proof}

We are finally ready to present a proof of Theorem \ref{mutiplicitylemma}.

\begin{proof}[Proof of Theorem \ref{mutiplicitylemma}]

In the case when $m=1$, $\eta$ is a transposition, say, $\eta=t_{ab}$ 
with $a\leq k<b$. We explain that $u\to ut_{ab}$ is the 
unique path from $u$ to $ut_{ab}$ in the extended $k$-Bruhat order. Suppose that $\gamma$ is a path  from $u$ to $ut_{ab}$ in the extended $k$-Bruhat order. Since $M(t_{ab})=\{a,b\}$, by Corollary \ref{pathmovingset}, 
  $t_{a b}$ is the only  transposition appearing in the construction of $\gamma$, and
  so $\gamma$ is exactly  the path $u\to ut_{ab}$. 
Clearly, the path $u\to ut_{ab}$ is of length one and with $\mathrm{de}(\gamma)=\operatorname{ht}_k(t_{ab})=0$. 

We next consider the case when $m\geq 2$. 
Let $a\in M(\eta)$ be the index reaching the minimum value  $\min\{u(i)\colon i\in M(\eta)\}$.  Since $m\geq 2$, we have $\eta(a)\neq \eta^{-1}(a)$.
The discussion is divided into two cases. 


Case 1.  $u(\eta^{-1}(a))<u(\eta(a))$. By Lemma \ref{mutiplicityllemma3}, we see that  $u\stackrel{u(a)}\longrightarrow ut_{a\,\eta^{-1}(a)}\leq_k u\eta$, where $a\leq k<\eta^{-1}(a)$.  Let $u'=ut_{a\,\eta^{-1}(a)}$ and $u'\eta'=u\eta$ be as defined  in   Lemma \ref{mutiplicityllemma3}. By induction on $m$, we   assume that there is a unique unimodal path $\gamma'$ from $u'$ to $u'\eta'$ of length $m-1$, which satisfies  
$\mathrm{de}(\gamma')=\operatorname{ht}_k(\eta')$. 
Consider the path
$$\gamma\colon u\stackrel{u(a)}\longrightarrow u'\to \stackrel{\gamma'}\cdots \to u'\eta'.$$
By Lemma \ref{mutiplicityllemma1}, 
the minimum label among the edges in $\gamma'$ is  
$\min\{u'(i)\colon i \in M(\eta')\}=\min\{u'M(\eta')\}$. 
By Lemma \ref{mutiplicityllemma3}, $u'M(\eta')=uM(\eta)\setminus \{u(a)\}$, and by the minimality of $u(a)$, we have 
\[\min\{u'(i)\colon i \in M(\eta')\}>
u(a).\]
Thus $\gamma$ is a  unimodal path from $u$ to $u\eta$ 
of length $m$ and
 with $\mathrm{de}(\gamma)=\mathrm{de}(\gamma')=\operatorname{ht}_k(\eta')$. By Lemma \ref{mutiplicityllemma3}, we have $\operatorname{ht}_k(\eta)=\operatorname{ht}_k(\eta')$, and so we obtain that $\mathrm{de}(\gamma)= \operatorname{ht}_k(\eta)$.

It remains to show that $\gamma$ is  unique.
Let  $\overline{\gamma}$ be any unimodal path from $u$ to $u\eta$ in the extended $k$-Bruhat order. By Lemma \ref{mutiplicityllemma1}, the minimum label in $\overline{\gamma}$ is $u(a)$. Since  $\overline{\gamma}$ is unimodal, $u(a)$  appears on the first or the last edge.
Because  $u(\eta^{-1}(a))<u(\eta(a))$, it follows from  Lemma \ref{mutiplicityllemma2}  that the label $u(a)$ must appear on the first edge, and the  transposition  corresponding to this  edge is  $t_{a\,\eta^{-1}(a)}$. 
Since the subpath $ \overline{\gamma}'$ from $u'=ut_{a\,\eta^{-1}(a)}$ to $u'\eta'=u\eta$ is still unimodal, by induction on $m$, we can assume that $ \overline{\gamma}'$ is unique. 
This enables us to conclude that  $\overline{\gamma}$   coincides with $\gamma$.  

Case 2.  $u(\eta^{-1}(a))>u(\eta(a))$. The arguments rely on Lemmas \ref{mutiplicityllemma1}, \ref{mutiplicityllemma3-2} and \ref{mutiplicityllemma2-2}, which are nearly the same as Case 1, and  so are omitted. 
\end{proof}

\subsection{Equivariant CSM MN Formula:  Theorem \ref{UUU-1}}

To attain  a proof of Theorem \ref{UUU-1}, we  establish the following Rigidity Theorem for multiplying a CSM class by a power sum symmetric function.

\begin{Th}[Rigidity Theorem]\label{noneqimplieEqhookMN}
Let  $u\in \S_n$ and  $A\subseteq [n]$. For $r\geq 1$, 
suppose  that 
\begin{equation}
c_{\mathrm{SM}}(Y(u)^\circ) \cdot p_r(x_A) = 
\sum_{w\in \S_n} d_{u,r}^w \cdot 
c_{\mathrm{SM}}(Y(w)^\circ).
\end{equation}
Then, we have
\begin{equation}\label{eq:abstractCSMMNrule}
c_{\mathrm{SM}}^T(Y(u)^\circ) \cdot p_r(x_A) = 
\sum_{w\in \S_n} d_{u,r}^w( t) \cdot 
c_{\mathrm{SM}}^T(Y(w)^\circ),
\end{equation}
where
\begin{equation}
d_{u,r}^w( t)=\begin{cases}
p_r( t_{uA}), & w=u,\\[5pt]
\sum\limits_{1\leq r'\leq r}d_{u,r'}^w\cdot  h_{r-r'}( t_{u M(u^{-1}w)}),& w\neq u. \label{99qw}
\end{cases}
\end{equation}
Here, recall that
\[u M(u^{-1}w)= \{u(i)\colon w(i)\neq u(i)\}.\]
\end{Th}

\begin{proof}
The proof is similar to that of  Theorem \ref{noneqimplieEqhook}, and is outlined below.
Recall that 
$$p(x_{A})=\sum_{a\in A}\frac{x_a}{1-\z x_a} = 
p_1(x_A) +p_2(x_A)\z  +\cdots.$$
Denote
$$d_u^w(z,t) = \sum_{r\ge1}d_{u,r}^w(t)z^{r-1}.$$
Then we have
$$c_{\mathrm{SM}}^T(Y(u)^\circ)\cdot p(x_{A}) = \sum_{w\in \S_n} d_u^w( z,t) \cdot c_{\mathrm{SM}}^T(Y(w)^\circ).$$
By Lemma \ref{PP-1}, we have 
\begin{align*}
d^w_{u}( z,t)=\mathcal{T}_{w/u}\big(p(x_{A})\big)\big|_{x_i= t_i}.
\end{align*}

If   $u=w$,  then $\mathcal{T}_{w/u}$ is just $u$, so $d^u_{u}( z,t)=p(t_{uA})$, which implies $d_{u,r}^u( t)=p_r(t_{uA}).$
Now we consider the case $u\neq w.$
Recall the notation $E(\q,\z,x_A)$ defined in Lemma \ref{OWW-1}:
\[E(\q,\z,x_A)=\frac{1}{\q+\z}\left(\frac{\Q(x_A)}{\Z(x_A)}-1\right).\]
By evaluating   $\lim_{q \rightarrow -z}E(\q,\z,x_A)$ using   the L’Hospital rule, 
we obtain the relation  
\[E(-\z,\z,x_{A})=\sum_{a\in A}\frac{x_a}{1-\z x_a}=p(x_{A}).\]
In Lemma \ref{NonimplieEquLemma}, we showed  that for  $A\subseteq [n]$, there exists $f(q,z)\in \mathbb{Z}[q,z]$ such that 
\begin{equation}\label{eq:prelimit}
\mathcal{T}_{w/ u} E(\q,\z,x_A)= \frac{1}{\q+\z}f(\q,\z) \frac{\Q(x_{\Cminus_A(u, w)})}{\Z(x_{\Cplus_A(u, w)})}.
\end{equation}
Notice the following  observation: for $A\subseteq B$,
\[
\left.\frac{\Q(x_A)}{\Z(x_B)}\right|_{\q=-\z} =
\frac{1}{\Z(x_{B\setminus A})}=
\sum_{j=0}^\infty h_j(x_{B\setminus A}) \z^j .
\]
Moreover, notice that
\[\Sigma_{A}(u,w)\setminus \Delta_A(u,w)=uM(u^{-1}w).\] 
Taking limit $q\to -z$ on both sides of \eqref{eq:prelimit}, we get 
$$\mathcal{T}_{w/u}\big(p(x_A)\big) =\left(\lim_{q\to -z} \frac{f(q,z)}{q+z}\right)\frac{1}{\Z(x_{uM(u^{-1}w)})},$$
and so we have 
\[d_u^w(z,t)=\left(\lim_{q\to -z} \frac{f(q,z)}{q+z}\right)\frac{1}{\Z(t_{uM(u^{-1}w)})}.\]
Letting all $t_i=0$ on both sides, we see that 
$$\lim_{q\to -z} \frac{f(q,z)}{q+z}= d_u^w(z,0),$$ 
and hence
\begin{align*}
d_u^w(z,t)
& = d_u^w(z,0)\cdot\sum_{j=0}^\infty  h_j( t_{uM(u^{-1}w)}) \z^j
 \\[3pt]
&= \sum_{i=0}^\infty d_{u,i+1}^w \z^i  \cdot \sum_{j=0}^\infty h_j(t_{uM(u^{-1}w)}) \z^j.
\end{align*}
Equating  the coefficient of $\z^{r-1}$  
  gives \eqref{99qw}, as desired. 
%
\end{proof}

Using  Theorems \ref{noneqMNruleCSM} and \ref{noneqimplieEqhookMN}, we  
reach a proof of Theorem \ref{UUU-1}. 

\begin{Th}[=Theorem \ref{UUU-1}: Equivariant CSM MN Formula]\label{MNforequivariantCSMSchur}
Let $u\in \S_n$. For $r\geq 1$, we have the following identity in $H_T^\bullet(\Fl(n))$: 
\[
c_{\mathrm{SM}}^T(Y(u)^\circ)\cdot p_{r}(x_{[k]})
=p_{r}( t_{u[k]})\cdot c_{\mathrm{SM}}^T(Y(u)^\circ)
+ \sum_{\eta\in \S_n} d_{u,r}^{u\eta}( t)\cdot  c_{\mathrm{SM}}^T(Y(u\eta)^\circ),
\]
where  the sum runs over $(r'+1)$-cycles $\eta\in \S_n$
with $1\leq r'\leq r$ such that $u\leq_k u\eta$ in the extended 
$k$-Bruhat order, and
\begin{equation}\label{eq:MNforequivariantCSMSchur}
d_{u,r}^{u\eta}( t) = (-1)^{\operatorname{ht}_k(\eta)}\cdot h_{r-r'}( t_{u M(\eta)}).
\end{equation}
\end{Th}

Taking the lowest degree part in Theorem \ref{MNforequivariantCSMSchur} leads to the following  MN formula for equivariant Schubert classes. 

\begin{Coro}[Equivariant Schubert MN Formula]\label{MNforequivariantSchubert}
Let $u\in \S_n$.  For $r\geq 1$, we have the following identity in $H_T^\bullet(\Fl(n))$: 
\begin{equation}
[Y(u)]_T\cdot p_{r}(x_{[k]})
=p_{r}( t_{u[k]})\cdot [Y(u)]_T
+ \sum_{\eta\in \S_n} \overline{d}_{u,r}^{u\eta}( t)\cdot [Y(w)]_T,
\end{equation}
where  the sum runs over $(r'+1)$-cycles $\eta\in \S_n$
with $1\leq r'\leq r$ such that $\ell(u\eta)=\ell(u)+r'$ and  $u\leq_k u\eta$ in the extended   $k$-Bruhat order, and $\overline{d}_{u,r}^{u\eta}( t)$ has the same expression as $d_{u,r}^{u\eta}( t)$ in \eqref{eq:MNforequivariantCSMSchur}. 
\end{Coro}

We remark that in  Corollary \ref{MNforequivariantSchubert},
the condition  $\ell(u\eta)=\ell(u)+r'$, together with Theorem \ref{mutiplicitylemma},
implies that $u\leq_k u\eta$ is actually in the
ordinary $k$-Bruhat order. 

If further taking all $t_i=0$ in Corollary \ref{MNforequivariantSchubert}, we arrive at the MN rule for 
nonequivariant  Schubert classes, as deduced by Morrison and Sottile \cite[Theorem 1]{MS2}.

We end this section with an example to illustrate Theorem \ref{MNforequivariantCSMSchur} and Corollary \ref{MNforequivariantSchubert}.

\begin{Eg}
{ With the same setting as in Example \ref{SSSCC}, let us compute 
\[
\zeta_{23154}\cdot p_3(x_1,x_2)\ \ \text{and}\ \ 
\sigma_{23154}\cdot p_3(x_1,x_2).
\]
In the first column in Table \ref{tab:cyclesfrom23154},
we list all $w\in \S_n$ such that $w=u\eta$ for some $(r'+1)$-cycle $\eta$ ($1\leq r'\leq r=3$), and $u\leq_k w$. The requirement  $u\leq_k w$  can be tested  efficiently by means of Lemma \ref{extkBruhatorder}. 
}

\begin{table}[h]\center
\begin{tabular}{c|c|c|c|c}\hline
$w$ &$\eta$ & $\operatorname{ht}_2+1$ & $d_{u,3}^{u\eta}( t)$ & $\ell(w)$\\\hline
$23\,154$ & $\texttt{-}$ & $0$& 
    $p_{3}( t_2, t_3)$ & $3$\\
$\underline{5}3\,1\underline{2}4$ & $(14)$ & $1$& 
    $h_{2}( t_{\{5,2\}})= t_2^2+ t_2 t_5+ t_5^2$ & $6$\\
$\underline{4}3\,15\underline{2}$ & $(15)$ & $1$& 
    $h_{2}( t_{\{4,2\}})= t_2^2+ t_2 t_4+ t_4^2$ & $6$\\
$2\underline{5}\,1\underline{3}4$ & $(24)$ & $1$& 
    $h_{2}( t_{\{5,3\}})= t_3^2+ t_3 t_5+ t_5^2$ & $\underline{4}$ \\
$2\underline{4}\,15\underline{3}$ & $(25)$ & $1$ & 
    $h_{2}( t_{\{4,3\}})= t_3^2+ t_3 t_4+ t_4^2$ & $\underline{4}$\\
$\underline{5}3\,1\underline{42}$ & $(145)$ & $1$ &
    $h_{1}( t_{\{5,4,2\}})= t_2+ t_4+ t_5$ & $7$\\
$\underline{35}\,1\underline{2}4$ & $(124)$& $2$ &
    $-h_{1}( t_{\{3,5,2\}})=- t_2- t_3- t_5$ & $\underline{5}$\\
$\underline{34}\,15\underline{2}$ & $(125)$& $2$ & 
    $-h_{1}( t_{\{3,4,2\}})=- t_2- t_3- t_4$ & $\underline{5}$\\
$2\underline{5}\,1\underline{43}$ & $(245)$& $1$ & 
    $h_{1}( t_{\{5,4,3\}})= t_3+ t_4+ t_5$ & $\underline{5}$\\
$\underline{54}\,1\underline{32}$ & $(1425)$& $2$ & 
    $-h_{0}( t_{\{5,4,3,2\}})=-1$ & $8$ \\
$\underline{35}\,1\underline{42}$ & $(1245)$& $2$ & 
    $-h_{0}( t_{\{3,5,4,2\}})=-1$ & $\underline{6}$\\
$\underline{45}\,1\underline{23}$ & $(1524)$& $2$ &
    $-h_{0}( t_{\{4,5,2,3\}})=-1$ & $\underline{6}$\\\hline
\end{tabular}
\caption{Computing $\zeta_{23154}\cdot p_{3}(x_1,x_2)$ and
$\sigma_{23154}\cdot p_3(x_1,x_2)$}
\label{tab:cyclesfrom23154}
\end{table}

By Theorem \ref{MNforequivariantCSMSchur}, we have 
\begin{align*}
\zeta_{23154}\cdot p_3(x_1,x_2) 
& =p_3( t_2, t_3)\cdot\zeta_{23154}\\
&\quad +( t_2^2+ t_2 t_5+ t_5^2)\cdot\zeta_{53124}
+ ( t_2^2+ t_2 t_4+ t_4^2)\cdot\zeta_{43152}\\
&\quad+ ( t_3^2+ t_3 t_5+ t_5^2)\cdot\zeta_{25134}
+( t_3^2+ t_3 t_4+ t_4^2)\cdot\zeta_{24153}\\
&\quad+ ( t_2+ t_4+ t_5)\cdot\zeta_{53142}
-( t_2+ t_3+ t_5)\cdot\zeta_{35124}\\
&\quad-( t_2+ t_3+ t_4)\cdot\zeta_{34152}
+( t_3+ t_4+ t_5)\cdot\zeta_{25143}\\
&\quad-\zeta_{54132}
-\zeta_{35142}
-\zeta_{45123}.
\end{align*}
To compute $\sigma_{23154}\cdot p_3(x_1,x_2) $, by Corollary \ref{MNforequivariantSchubert}, we need to 
pick out the permutations in the first column satisfying  that 
$\ell(w)=\ell(u)+r'$. Such $\ell(w)$ are underlined in the last column in Table \ref{tab:cyclesfrom23154}, and so we obtain that 
\begin{align*}
\sigma_{23154}\cdot p_3(x_1,x_2) 
& =p_3( t_2, t_3)\cdot\sigma_{23154}\\
&\quad + ( t_3^2+ t_3 t_5+ t_5^2)\cdot\sigma_{25134}
+( t_3^2+ t_3 t_4+ t_4^2)\cdot\sigma_{24153}\\
&\quad-( t_2+ t_3+ t_5)\cdot\sigma_{35124}-( t_2+ t_3+ t_4)\cdot\sigma_{34152}\\
&\quad
+( t_3+ t_4+ t_5)\cdot\sigma_{25143}-\sigma_{35142}
-\sigma_{45123}.
\end{align*}

\end{Eg}


\section{Grassmannian Cases}\label{GGG_I}

This section  concerns  the multiplication formulas for  CSM classes  and Schubert classes over Grassmannians. 
As will be seen, all formulas established before have parabolic analogues. We  illustrate this by two concrete  examples: parabolic versions of Theorem \ref{NoneqPierirulehook} and Corollary \ref{MNforequivariantSchubert}, both  of which can be described in terms of the combinatorics of partitions. 
The parabolic version of  Corollary \ref{MNforequivariantSchubert} will   be   used in 
Section \ref{SECT7} to investigate  the enumeration  formulas
of rim hook tableaux.

\subsection{Geometry of Grassmannians}
Denote by  $\Gr(k,n)$ the Grassmannian of $k$-planes in $\mathbb{C}^n$. 
There is a proper $T$-equivariant morphism $$\pi: \Fl(n)\longrightarrow \Gr(k,n)$$
by sending a flag $(V_\bullet)$ to its $k$-plane $V_k$. 
 This induces two morphisms: the pullback $\pi^*$ and the pushforward $\pi_*$:
\begin{gather*}
\pi^*\colon H_T^\bullet(\Gr(k,n))\to H_T^\bullet(\Fl(n))
\ \ \text{ and }\ \ 
H^\bullet(\Gr(k,n))\to H^\bullet(\Fl(n)),\\[5pt]
\pi_*\colon H_T^\bullet(\Fl(n))\to H_T^\bullet(\Gr(k,n))
\ \ \text{ and }\ \ 
H^\bullet(\Fl(n))\to H^\bullet(\Gr(k,n)). 
\end{gather*}
 Recall from Section \ref{Sec-2} that $\mathfrak{S}_{n}$ acts on $H_T^\bullet(\Fl(n))$ (resp., $H^\bullet(\Fl(n))$) by permuting the $x_i$'s. 
 It is known that
$H_T^\bullet(\Gr(k,n))$ (resp., $H^\bullet(\Gr(k,n))$) is isomorphic to the $(\mathfrak{S}_{k}\times \mathfrak{S}_{n-k})$-invariant algebra of $H_T^\bullet(\Fl(n))$ (resp., $H^\bullet(\Fl(n))$),  and $\pi^*$ coincides with the inclusion map \cite{BGG}. 

For any $w\in \S_n$, 
there is a unique decomposition 
\begin{equation}\label{wlambda}
w=w_{\lambda}\,v,
\end{equation}
where $v\in \mathfrak{S}_k\times \mathfrak{S}_{n-k}$, and $w_{\lambda}$ is a Grassmannian permutation 
associated to a partition $\lambda=(\lambda_1,\ldots, \lambda_k)$. Precisely, $w_{\lambda}$ is obtained from $w$ by rearranging the first $k$ elements and the last $n-k$ elements in increasing order, respectively. 
Formally, $\mathfrak{S}_k\times \mathfrak{S}_{n-k}$
is a maximal parabolic subgroup of $\S_n$, and 
$w_\lambda$ is the minimal coset representative of 
the left coset of  $\mathfrak{S}_k\times \mathfrak{S}_{n-k}$ with respect to $w$.
Notice that $\lambda$ is inside the $k\times(n-k)$ rectangle. We shall use $\mathrm{Gr}(w)$ to denote the partition $\lambda$. For example, for $k=4$ and 
$w=516342$, we have $w_\lambda=135624$ and $\mathrm{Gr}(w)=\lambda=(2,2,1,0)$.

The Schubert cell $Y(\lambda)^\circ$ refers to the set-theoretic image of $Y(w_{\lambda})^\circ$ under $\pi$, namely,
\[Y(\lambda)^\circ=\pi(Y(w_{\lambda})^\circ).\]
The Schubert variety $Y(\lambda)$ is the closure of $Y(\lambda)^\circ$. 
By \cite[Proposition 3.5]{AM}, if $\mathrm{Gr}(w)=\lambda$, 
\begin{equation}\label{eq:prj}
\pi_* \big(\,c_{\mathrm{SM}}^T(Y(w)^\circ) \,\big)= 
c_{\mathrm{SM}}^T(Y(\lambda)^\circ)
\quad \text{and} \quad
\pi_* \big(\,c_{\mathrm{SM}}(Y(w)^\circ) \,\big)= 
c_{\mathrm{SM}}(Y(\lambda)^\circ),
\end{equation}
see also \cite[Theorem 4.3]{MNS2}. 
This allows the multiplication formulas 
for CSM classes established in Sections \ref{Sect4} and \ref{Sect5} to be converted into
formulas involving the  associated partitions. 

For $u\in \S_n$ with $\mathrm{Gr}(u)=\lambda$ and a class $\alpha\in H_T^\bullet(\Gr(k,n))$, suppose that 
\begin{equation}\label{eq:CSMco}
c_{\mathrm{SM}}^T(Y(u)^\circ)\cdot \alpha = \sum_{w\in \S_n} c_{u,\alpha}^{w}( t)\cdot  c_{\mathrm{SM}}^T(Y(w)^\circ),
\end{equation}
where, via the pullback $\pi^*$,   $\alpha$ is also identified with a class in $H_T^\bullet(\Fl(n))$ symmetric under $\S_k\times \S_{n-k}$. 
Applying the pushforward map $\pi_*$ to the left-hand side of \eqref{eq:CSMco} and using the projection formula yields
\begin{align*}
\pi_*\big(\,c_{\mathrm{SM}}(Y(u)^\circ)\cdot\alpha\,\big)
&=\pi_*\big(\,c_{\mathrm{SM}}(Y(u)^\circ)\cdot\pi^*(\alpha)\,\big)\\
&=\pi_*\big(\,c_{\mathrm{SM}}(Y(u)^\circ)\,\big)\cdot\alpha\\
&=c_{\mathrm{SM}}(Y(\lambda)^\circ)\cdot\alpha.
\end{align*}
Applying   $\pi_*$ to the right-hand side of \eqref{eq:CSMco} and combining the above, we obtain that 
\begin{equation}\label{eq:eqPierirulehookParabolic}
c_{\mathrm{SM}}^T(Y(\lambda)^\circ)\cdot \alpha =\sum_{\mu} \left(\sum_{\begin{subarray}{c}w\in \S_n\\ \mathrm{Gr}(w)=\mu\end{subarray}}c_{u,\alpha}^{w}( t)\right) c_{\mathrm{SM}}^T(Y(\mu)^\circ).
\end{equation}
The same derivation  applies to   nonequivariant CSM classes.


We turn to  Schubert classes. It was shown in \cite{BGG} that
\begin{equation}\label{eq:preimageofSchubert}
\pi^* \big(\,[Y(\lambda)]_T\,\big)=[Y(w_{\lambda})]_T
\quad \text{and} \quad
\pi^* \big(\,[Y(\lambda)]\,\big)=[Y(w_{\lambda})].
\end{equation}
We refer to \cite[\textsection 10.8]{AF} for more information. 
As a result, multiplication formulas concerning Schubert classes, for example, Corollaries \ref{LRrulehookshapeSchubert} and \ref{MNforequivariantSchubert}, can be naturally restricted to Grassmannians in the following sense. 
For $u\in \S_n$ and a class $\alpha\in H_T^\bullet(\Gr(k,n))$, suppose that 
\begin{equation}\label{eq:Schubertco}
[Y(u)]_T\cdot \alpha = \sum_{w\in \S_n} \overline{c}_{u,\alpha}^{w}( t)\cdot  [Y(w)]_T,
\end{equation}
where  $\alpha$ is also regarded  as a class in $H_T^\bullet(\Fl(n))$. 
We  conclude that 
\begin{equation}\label{eq:eqSchubertcoParabolic}
[Y(\lambda)]_T\cdot \alpha = \sum_{\mu} \overline{c}_{w_\lambda,\alpha}^{w_\mu}( t)\cdot  [Y(\mu)]_T.
\end{equation}
Actually, by \eqref{eq:preimageofSchubert}, one can identify $[Y(\lambda)]_T$ with $[Y(w_{\lambda})]_T$ through $\pi^*$. Thus \eqref{eq:eqSchubertcoParabolic} is nothing but the case when $u=w_{\lambda}$. 


In the rest of this section, we shall pay attention 
to the parabolic versions of Theorem \ref{NoneqPierirulehook} and Corollary \ref{MNforequivariantSchubert}, both of which 
can be described explicitly  in terms of   operations on  partitions. In particular, the parabolic treatment of  Corollary \ref{MNforequivariantSchubert} will be applied  in the next section to  the enumeration of rim hook tableaux.

\subsection{Parabolic Version of Theorem \ref{NoneqPierirulehook}} 

For two distinct partitions $\lambda$ and $\mu$ inside the $k\times (n-k)$ rectangle, we write $\lambda\longrightarrow \mu$ if there exist $u, w\in \S_n$ with $\mathrm{Gr}(u)=\lambda$ and $\mathrm{Gr}(w)=\mu$ such that $u\longrightarrow w$ is a $k$-edge.   The $k$-edge $u\longrightarrow w$ is not necessarily  unique. 
In fact, it is easily checked  that the collection of 
$k$-edges generating $\lambda\longrightarrow \mu$
are of the form $uv\longrightarrow wv$, where $v\in \S_k\times \S_{n-k}$.
Such $k$-edges  enjoy the same label, say $uv\stackrel{\tau}\longrightarrow wv$, 
and    $\lambda\longrightarrow \mu$ inherits  this label, so we may write 
$$
\lambda\stackrel{\tau}\longrightarrow \mu.
$$

The partitions inside the $k\times (n-k)$ rectangle, along with the (labeled) edges $\lambda\stackrel{\tau}\longrightarrow \mu$, constitute 
a directed graph, which is called the (labeled) {\it $k$-Bruhat graph} on partitions inside the $k\times (n-k)$ rectangle. We also call the edges in this graph  $k$-edges since no confusion would arise from the context. We remark that if one ignores the edge labels, the $k$-Bruhat graph on partitions inside the $k\times (n-k)$ rectangle is exactly  the $\Lambda$-Bruaht graph on $\S_n/\mathfrak{S}_k\times \mathfrak{S}_{n-k}$ \cite[Definition 8.4]{MNS}, or the { singular Bruhat graph} on $\S_n/\mathfrak{S}_k\times \mathfrak{S}_{n-k}$ \cite[Definition 6.9]{GW}.

It was pointed out in \cite[\textsection 10.1]{MNS} that there exists a $k$-edge
$\lambda \stackrel{\tau}\longrightarrow \mu$
if and only if $\mu$ can be obtained from $\lambda$ by adding a {\it rim hook} (also called  border strip), that is, $\mu/\lambda$ is a  connected skew shape with no $2\times 2$ square. Figure \ref{parttk} displays the  $k$-Bruhat  graphs on   partitions respectively inside the $2\times 2$, $3\times1$, and $1\times 3$ rectangles, where dashed edges signify  $\lambda\stackrel{\tau}\longrightarrow \mu$ with $|\mu/\lambda|>1$.
\begin{figure}[htb]\center
$$
\def\labelstyle#1{\scriptstyle\rule[-0.25pc]{0pc}{0.9pc}\,\,#1\,\,}
\ytableausetup{mathmode, boxsize=0.5pc}
\xymatrix@!=0.5pc{
&&{\ydiagram{2,2}}
&&\\&&&
\ar@{}[dl]|{{\ydiagram{2,1}}\rule[-0.75pc]{0pc}{1.5pc}}="B"
    \kar"B";[ul]|>>>>>>{{2}}
    \\
{\ydiagram{2,0}\rule[-0.25pc]{0pc}{1.5pc}}
    \kar"B"|{{1}}
    \kkar[uurr]|{{1}}
     &&&&
{\ydiagram{1,1}\rule[-0.25pc]{0pc}{1.5pc}}
    \kar"B"|>>>>>>{{3}}
    \kkar[uull]|{{2}}\\&
\ar@{}[ur]|{\,\,\ydiagram{1}\,\,}="F"
    \kar"F";[ul]|>>>>>>>{{3}}
    \kar@<-1pt>"F";[urrr]|{{1}}
    \kkar"F";[uuur]|{{1}}\\&&
{\varnothing}
    \kkar[uull]|{{2}}
    \kkar[uurr]|{{1}}
    \kar"F"|{{2}}
    \kkar"B"|{{1}} 
}%
\qquad
%
\ytableausetup{mathmode, boxsize=0.5pc}
\xymatrix@!=1.2pc{
{\,\ydiagram{1,1,1}\,}\\
{\,\ydiagram{1,1}\,}
    \kar[u]^>>>>{{1\!\!}}\\
{\ydiagram{1}}
    \kar[u]|>>>>>{{2}}
    \kkar@/_1pc/[uu]|{{1}}\\
{\varnothing}
    \kar[u]|>>>>>>{{3}}
    \kkar@/^1pc/[uu]|{{2}}
    \kkar@/^2pc/[uuu]|{{1}}
}%
%
\qquad 
%
\ytableausetup{mathmode, boxsize=0.5pc}
\xymatrix@!=1.2pc{
{\ydiagram{3}}\\
{\ydiagram{2}}
    \kar[u]|>>>>>>{{3}}\\
{\ydiagram{1}}
    \kar[u]|>>>>>>{{2}}
    \kkar@/_1pc/[uu]|{{2}}\\
{\varnothing}
    \kar[u]|>>>>>>{{1}}
    \kkar@/^1pc/[uu]|{{1}}
    \kkar@/^2pc/[uuu]|{{1}}
}$$
\caption{The $k$-Bruhat graphs on   partitions inside   $2\times2$, $3\times1$, and $1\times3$.}\label{parttk}
\end{figure}
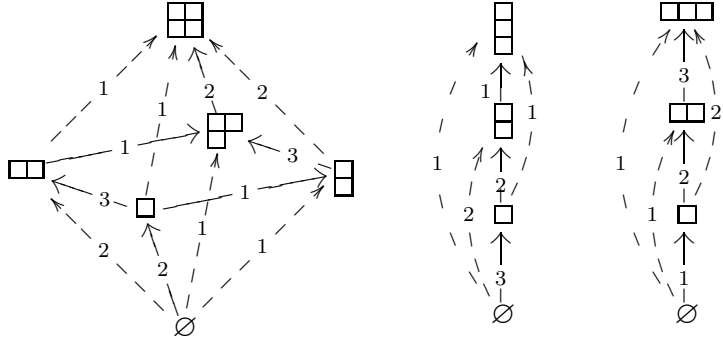

\begin{Lemma}\label{uniqueLiftingLemma}
Let $u\in \S_n$ be such that $\mathrm{Gr}(u)=\lambda$.  Then any path
$$\lambda=\lambda^{(0)}\stackrel{\tau_1}\longrightarrow 
\lambda^{(1)}\stackrel{\tau_2}\longrightarrow \cdots
\stackrel{\tau_m}\longrightarrow \lambda^{(m)}$$
in the $k$-Bruhat graph on partitions inside $k\times (n-k)$ admits a unique lifting 
$$u=\,w_0\,\stackrel{\tau_1}\longrightarrow 
\,w_1\,\stackrel{\tau_2}\longrightarrow \cdots
\stackrel{\tau_m}\longrightarrow \,w_m$$
in the $k$-Bruhat graph on $\S_n$ starting from $u$, where  $\mathrm{Gr}(w_i)=\lambda^{(i)}$  for $0\leq i\leq m$. 
\end{Lemma}

\begin{proof}
As aforementioned, the collection of 
$k$-edges generating $\lambda^{(0)}\stackrel{\tau_1}\longrightarrow 
\lambda^{(1)}$
are of the form $w_0v\stackrel{\tau_1}\longrightarrow w_1v$, where $v\in \S_k\times \S_{n-k}$.
Since  $w_0=u$ is fixed, the  permutation $w_1$ is also uniquely determined. For the same reason, the permutations $w_i$ for $i=2,\ldots, m$ are accordingly uniquely   determined. 
\end{proof}

There is  a simple combinatorial rule to determine the label $\tau$ of a $k$-edge
$\lambda \stackrel{\tau}\longrightarrow \mu$. 
 Draw the Young diagram of $\lambda$, and label the southeast boundary from bottom left to top right by $1,2,\ldots,n$ in increasing order. Note that the $k$ labels received by the $k$ rows are exactly the first $k$ elements of the corresponding Grassmannian permutation  $w_{\lambda}$. For example, for $n=9$,  $k=4$ and $\lambda=(4,2,2, 0)$, the labeling of $\lambda$ is illustrated in the left diagram of Figure \ref{eq:egofechoice2}, from which one can read off  $w_{\lambda}=145823679$.

\begin{figure}[h]\center
\setlength{\unitlength}{1.5pc}
\begin{picture}(5,5)
{\linethickness{1pt}\put(0,4){\line(1,0){5}}}
\put(0,3){\line(1,0){2}}
{\linethickness{1pt}\put(2,3){\line(1,0){2}}}
\put(0,2){\line(1,0){2}}
{\linethickness{1pt}\put(0,1){\line(1,0){2}}}
{\linethickness{1pt}\put(0,0){\line(0,1){4}}}
\put(1,1){\line(0,1){3}}
{\linethickness{1pt}\put(2,1){\line(0,1){2}}}
\put(2,3){\line(0,1){1}}
\put(3,3){\line(0,1){1}}
{\linethickness{1pt}\put(4,3){\line(0,1){1}}}
\put(-0.2,0.3){%
    \makebox[0pc][l]{\color[rgb]{1.00,1.00,1.00}%
    \rule{0.4\unitlength}{0.4\unitlength}}%
    \raisebox{0.1pc}{\makebox[0.4\unitlength]%
        {$\scriptscriptstyle\bf1$}}%
    }
\put(0.3,0.8){%
    \makebox[0pc][l]{\color[rgb]{1.00,1.00,1.00}%
    \rule{0.4\unitlength}{0.4\unitlength}}%
    \raisebox{0.1pc}{\makebox[0.4\unitlength]%
        {$\scriptscriptstyle\bf2$}}%
    }
\put(1.3,0.8){%
    \makebox[0pc][l]{\color[rgb]{1.00,1.00,1.00}%
    \rule{0.4\unitlength}{0.4\unitlength}}%
    \raisebox{0.1pc}{\makebox[0.4\unitlength]%
        {$\scriptscriptstyle\bf3$}}%
    }
\put(1.8,1.3){%
    \makebox[0pc][l]{\color[rgb]{1.00,1.00,1.00}%
    \rule{0.4\unitlength}{0.4\unitlength}}%
    \raisebox{0.1pc}{\makebox[0.4\unitlength]%
        {$\scriptscriptstyle\bf4$}}%
    }
\put(1.8,2.3){%
    \makebox[0pc][l]{\color[rgb]{1.00,1.00,1.00}%
    \rule{0.4\unitlength}{0.4\unitlength}}%
    \raisebox{0.1pc}{\makebox[0.4\unitlength]%
        {$\scriptscriptstyle\bf5$}}%
    }
\put(2.3,2.8){%
    \makebox[0pc][l]{\color[rgb]{1.00,1.00,1.00}%
    \rule{0.4\unitlength}{0.4\unitlength}}%
    \raisebox{0.1pc}{\makebox[0.4\unitlength]%
        {$\scriptscriptstyle\bf6$}}%
    }
\put(3.3,2.8){%
    \makebox[0pc][l]{\color[rgb]{1.00,1.00,1.00}%
    \rule{0.4\unitlength}{0.4\unitlength}}%
    \raisebox{0.1pc}{\makebox[0.4\unitlength]%
        {$\scriptscriptstyle\bf7$}}%
    }
\put(3.8,3.3){%
    \makebox[0pc][l]{\color[rgb]{1.00,1.00,1.00}%
    \rule{0.4\unitlength}{0.4\unitlength}}%
    \raisebox{0.1pc}{\makebox[0.4\unitlength]%
        {$\scriptscriptstyle\bf8$}}%
    }
\put(4.3,3.8){%
    \makebox[0pc][l]{\color[rgb]{1.00,1.00,1.00}%
    \rule{0.4\unitlength}{0.4\unitlength}}%
    \raisebox{0.1pc}{\makebox[0.4\unitlength]%
        {$\scriptscriptstyle\bf9$}}%
    }
\end{picture}
\qquad
\begin{picture}(5,5)
\put(2,1){\color{lightgray}\rule{\unitlength}{\unitlength}}
\put(2,2){\color{lightgray}\rule{2\unitlength}{\unitlength}}
\put(2,1.3){\makebox[\unitlength]{$*$}}
{\linethickness{1pt}\put(0,4){\line(1,0){5}}}
\put(0,3){\line(1,0){4}}
\put(0,2){\line(1,0){3}}
{\linethickness{1pt}\put(3,2){\line(1,0){1}}}
{\linethickness{1pt}\put(0,1){\line(1,0){3}}}
{\linethickness{1pt}\put(0,0){\line(0,1){4}}}
\put(1,1){\line(0,1){3}}
\put(2,1){\line(0,1){3}}
\put(3,2){\line(0,1){2}}
{\linethickness{1pt}\put(3,1){\line(0,1){1}}}
{\linethickness{1pt}\put(4,2){\line(0,1){2}}}
\put(-0.2,0.3){%
    \makebox[0pc][l]{\color[rgb]{1.00,1.00,1.00}%
    \rule{0.4\unitlength}{0.4\unitlength}}%
    \raisebox{0.1pc}{\makebox[0.4\unitlength]%
        {$\scriptscriptstyle\bf1$}}%
    }
\put(0.3,0.8){%
    \makebox[0pc][l]{\color[rgb]{1.00,1.00,1.00}%
    \rule{0.4\unitlength}{0.4\unitlength}}%
    \raisebox{0.1pc}{\makebox[0.4\unitlength]%
        {$\scriptscriptstyle\bf2$}}%
    }
\put(1.3,0.8){%
    \makebox[0pc][l]{\color[rgb]{1.00,1.00,1.00}%
    \rule{0.4\unitlength}{0.4\unitlength}}%
    \raisebox{0.1pc}{\makebox[0.4\unitlength]%
        {$\scriptscriptstyle\bf3$}}%
    }
\put(2.3,0.8){%
    \makebox[0pc][l]{\color[rgb]{1.00,1.00,1.00}%
    \rule{0.4\unitlength}{0.2\unitlength}}%
    \makebox[0pc][l]{\color{lightgray}%
    \rule[0.2\unitlength]{0.4\unitlength}{0.2\unitlength}}%
    \raisebox{0.1pc}{\makebox[0.4\unitlength]%
        {$\scriptscriptstyle\bf4$}}%
    }
\put(2.8,1.3){%
    \makebox[0pc][l]{{\color{lightgray}%
    \rule{0.2\unitlength}{0.4\unitlength}}%
                    {\color[rgb]{1.00,1.00,1.00}%
    \rule{0.2\unitlength}{0.4\unitlength}}}%
    \raisebox{0.1pc}{\makebox[0.4\unitlength]%
        {$\scriptscriptstyle\bf5$}}%
    }
\put(3.3,1.8){%
    \makebox[0pc][l]{\color[rgb]{1.00,1.00,1.00}%
    \rule{0.4\unitlength}{0.2\unitlength}}%
    \makebox[0pc][l]{\color{lightgray}%
    \rule[0.2\unitlength]{0.4\unitlength}{0.2\unitlength}}%
    \raisebox{0.1pc}{\makebox[0.4\unitlength]%
        {$\scriptscriptstyle\bf6$}}%
    }
\put(3.8,2.3){%
    \makebox[0pc][l]{{\color{lightgray}%
    \rule{0.2\unitlength}{0.4\unitlength}}%
                    {\color[rgb]{1.00,1.00,1.00}%
    \rule{0.2\unitlength}{0.4\unitlength}}}%
    \raisebox{0.1pc}{\makebox[0.4\unitlength]%
        {$\scriptscriptstyle\bf7$}}%
    }
\put(3.8,3.3){%
    \makebox[0pc][l]{\color[rgb]{1.00,1.00,1.00}%
    \rule{0.4\unitlength}{0.4\unitlength}}%
    \raisebox{0.1pc}{\makebox[0.4\unitlength]%
        {$\scriptscriptstyle\bf8$}}%
    }
\put(4.3,3.8){%
    \makebox[0pc][l]{\color[rgb]{1.00,1.00,1.00}%
    \rule{0.4\unitlength}{0.4\unitlength}}%
    \raisebox{0.1pc}{\makebox[0.4\unitlength]%
        {$\scriptscriptstyle\bf9$}}%
    }
\end{picture}
\caption{The construction of $w_{\lambda}$ and $L(\mu/\lambda)$}\label{eq:egofechoice2}
\end{figure}
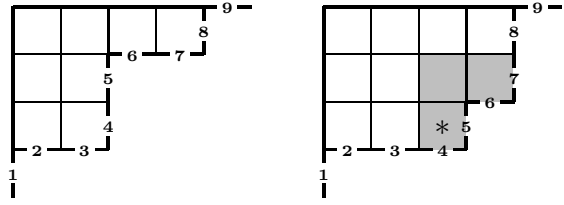

To locate $\tau$, label the southeast boundary of  $\mu$. Let  $L(\mu/\lambda)$   denote the set of labels   received by the southeast boundary of the  rim hook   $\mu/\lambda$.
For $\lambda=(4,2,2,0)$ and $\mu=(4,4,3,0)$, 
we see from the right diagram of Figure \ref{eq:egofechoice2} that $L(\mu/\lambda)=\{4,5,6,7\}$.
We have the following observation 
\[\tau=\min L(\mu/\lambda).\]
This can be understood as follows. Assume that 
 $u\stackrel{\tau}\longrightarrow w$ is a $k$-edge with $\mathrm{Gr}(u)=\lambda$ and $\mathrm{Gr}(w)=\mu$. 
 Then the set of labels received by the rows of $\lambda$ (resp., 
 $\mu$) is $\{u(1),\ldots, u(k)\}$ (resp., $\{w(1),\ldots, w(k)\}$). Notice that $\tau\in\{u(1),\ldots, u(k)\}\setminus\{w(1),\ldots, w(k)\}$, which, by the definition of $L(\mu/\lambda)$, is exactly the value $\min L(\mu/\lambda)$. With the above $\lambda$ and $\mu$, 
 we see that $\tau=4$. 
 
The  leftmost box in the bottom row of the rim hook $\mu/\lambda$ is referred to 
as the {\it tail} of $\mu/\lambda$. Clearly,  the bottom edge of the tail box is endowed with the label $\tau=\min L(\mu/\lambda)$. As depicted in  Figure \ref{eq:egofechoice2}, the  tail of $\mu/\lambda$ is marked with a star. 
With this notion, we have the following characterizations of increasing/decreasing paths in the $k$-Bruhat graph of partitions. 

\begin{Lemma}\label{uniqueLiftingLemma-II}
Assume that $\lambda$ and $\mu$ are two partitions inside 
the $k\times (n-k)$-rectangle. Then a path
$$\lambda=\lambda^{(0)}\stackrel{\tau_1}\longrightarrow 
\lambda^{(1)}\stackrel{\tau_2}\longrightarrow \cdots
\stackrel{\tau_m}\longrightarrow \lambda^{(m)}=\mu$$
from $\lambda$ to $\mu$ is increasing (resp., decreasing)
if and only if for $1\leq i\leq m-1$, 
the tail of $\lambda^{(i+1)}/\lambda^{(i)}$ is strictly 
to the right (resp., strictly below) of the tail of $\lambda^{(i)}/\lambda^{(i-1)}$. 
\end{Lemma}

Now, Theorem \ref{NoneqPierirulehook}, formula  \eqref{eq:eqPierirulehookParabolic}, and  Lemmas \ref{uniqueLiftingLemma} and \ref{uniqueLiftingLemma-II} together lead to 
the following multiplication formula over $H^\bullet(\operatorname{\mathcal{G}\it r}(k,n))$.

\begin{Th}
Let $\lambda$ be a partition inside the $k\times (n-k)$ rectangle. For a hook shape partition $\Gamma=(1+\alpha,1^{\beta})$, we have the following identity in $H^\bullet(\operatorname{\mathcal{G}\it r}(k,n))$:
\begin{equation*}
c_{\mathrm{SM}}(Y(\lambda)^\circ)\cdot s_{\Gamma}(x_{[k]}) = \sum_\mu c^{\mu}_{\lambda,\Gamma}\cdot c_{\mathrm{SM}}(Y(\mu)^\circ),
\end{equation*}
where the sum ranges over partitions   inside the $k\times (n-k)$ rectangle, and the coefficient $c^{\mu}_{\lambda,\Gamma}$ is equal to the number of
paths 
\[\lambda=\lambda^{(0)}\longrightarrow 
\lambda^{(1)}\longrightarrow \cdots
\longrightarrow \lambda^{(\alpha+\beta+1)}=\mu\]
 such that for $1\leq i\leq \beta$, the   tail
 of   $\lambda^{(i+1)}/\lambda^{(i)}$ is   strictly below the tail of $\lambda^{(i)}/\lambda^{(i-1)}$, and for $\beta+1\leq j\leq \alpha+\beta$, the   tail of $\lambda^{(j+1)}/\lambda^{(j)}$ is strictly to the   right of the tail of $\lambda^{(j)}/\lambda^{(j-1)}$. 
\end{Th}


\begin{Eg}%
\definecolor{a}{rgb}{0.8,0.8,0.8}%
\definecolor{b}{rgb}{0.4,0.4,0.4}%
\definecolor{c}{rgb}{0.0,0.0,0.0}%
 \definecolor{b}{rgb}{1,0.5,0.5}%
 \definecolor{a}{rgb}{1,1,0.5}%
 \definecolor{c}{rgb}{0.5,0.5,1}%
Write $\zeta_{\lambda}=c_{\mathrm{SM}}(Y(\lambda)^\circ)$. 
Let $k=3, n=7$, and $\lambda=(3,2,0)$. 
Consider  the cases  when $\alpha+\beta+1=3$. 
We use the colors ${\color{a}\blacksquare}$, ${\color{b}\blacksquare}$, ${\color{c}\blacksquare}$ 
to indicate  the first, the second   and the third added rim hooks, respectively. 
For $\alpha=0$ and $\beta=2$, we have $s_\Gamma(x_1, x_2, x_3)=e_3(x_1, x_2, x_3)$.  
The  decreasing paths inside the $3\times 4$ rectangle starting from $\lambda$ are
$$\ytableausetup{mathmode, boxsize=0.5pc}
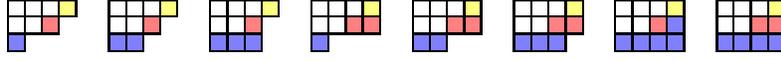
\begin{array}{c}
    \begin{ytableau}
     {} & {} & {} &*(a)\\
     {} & {} &*(b)\\
    *(c)
    \end{ytableau}
\quad
    \begin{ytableau}
     {} & {} & {} &*(a)\\
     {} & {} &*(b)\\
    *(c)&*(c)
    \end{ytableau}
\quad
    \begin{ytableau}
     {} & {} & {} &*(a)\\
     {} & {} &*(b)\\
    *(c)&*(c)&*(c)
    \end{ytableau}
\quad
    \begin{ytableau}
     {} & {} & {} &*(a)\\
     {} & {} &*(b)&*(b)\\
    *(c)
    \end{ytableau}
\quad
    \begin{ytableau}
     {} & {} & {} &*(a)\\
     {} & {} &*(b)&*(b)\\
    *(c)&*(c)
    \end{ytableau}
\quad
    \begin{ytableau}
     {} & {} & {} &*(a)\\
     {} & {} &*(b)&*(b)\\
    *(c)&*(c)&*(c)
    \end{ytableau}
\quad
    \begin{ytableau}
     {} & {} & {} &*(a)\\
     {} & {} &*(b)&*(c)\\
    *(c)&*(c)&*(c)&*(c)
    \end{ytableau}
\quad
    \begin{ytableau}
     {} & {} & {} &*(a)\\
     {} & {} &*(b)&*(b)\\
    *(c)&*(c)&*(c)&*(c)
    \end{ytableau}
\end{array}$$
So we have 
\begin{align*}\zeta_{(3,2,0)}\cdot e_3(x_1, x_2, x_3) =&
\zeta_{(4,3,1)}+\zeta_{(4,3,2)}+\zeta_{(4,3,3)}+\zeta_{(4,4,1)}\\
&\ +\zeta_{(4,4,2)}+\zeta_{(4,4,3)}+2\zeta_{(4,4,4)}.
\end{align*}
For $\alpha=2$ and $\beta=0$, we have $s_\Gamma(x_1, x_2, x_3)=h_3(x_1, x_2, x_3)$, and the increasing paths inside the $3\times 4$ rectangle starting from $\lambda$ are
$$\ytableausetup{mathmode, boxsize=0.5pc}
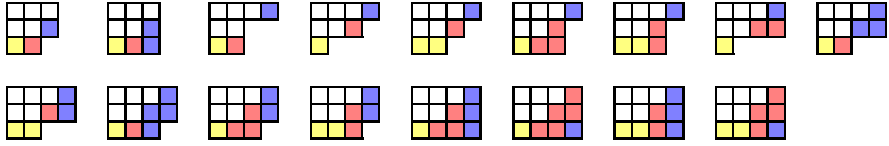
\begin{array}{l}
    \begin{ytableau}
     {} & {} & {} &\none\\
     {} & {} &*(c)\\
    *(a)&*(b)
    \end{ytableau}
\quad
    \begin{ytableau}
     {} & {} & {} &\none\\
     {} & {} &*(c)\\
    *(a)&*(b)&*(c)
    \end{ytableau}
\quad
    \begin{ytableau}
     {} & {} & {} &*(c)\\
     {} & {} \\
    *(a)&*(b)
    \end{ytableau}
\quad
    \begin{ytableau}
     {} & {} & {} &*(c)\\
     {} & {} &*(b)\\
    *(a)
    \end{ytableau}
\quad
    \begin{ytableau}
     {} & {} & {} &*(c)\\
     {} & {} &*(b)\\
    *(a)&*(a)
    \end{ytableau}
\quad
    \begin{ytableau}
     {} & {} & {} &*(c)\\
     {} & {} &*(b)\\
    *(a)&*(b)&*(b)
    \end{ytableau}
\quad
    \begin{ytableau}
     {} & {} & {} &*(c)\\
     {} & {} &*(b)\\
    *(a)&*(a)&*(b)
    \end{ytableau}
\quad
    \begin{ytableau}
     {} & {} & {} &*(c)\\
     {} & {} &*(b)&*(b)\\
    *(a)
    \end{ytableau}
\quad
    \begin{ytableau}
     {} & {} & {} &*(c)\\
     {} & {} &*(c)&*(c)\\
    *(a)&*(b)
    \end{ytableau}
\\[1.5pc]
    \begin{ytableau}
     {} & {} & {} &*(c)\\
     {} & {} &*(b)&*(c)\\
    *(a)&*(a)
    \end{ytableau}
\quad
    \begin{ytableau}
     {} & {} & {} &*(c)\\
     {} & {} &*(c)&*(c)\\
    *(a)&*(b)&*(c)
    \end{ytableau}
\quad
    \begin{ytableau}
     {} & {} & {} &*(c)\\
     {} & {} &*(b)&*(c)\\
    *(a)&*(b)&*(b)
    \end{ytableau}
\quad
    \begin{ytableau}
     {} & {} & {} &*(c)\\
     {} & {} &*(b)&*(c)\\
    *(a)&*(a)&*(b)
    \end{ytableau}
\quad
    \begin{ytableau}
     {} & {} & {} &*(c)\\
     {} & {} &*(b)&*(c)\\
    *(a)&*(b)&*(b)&*(c)
    \end{ytableau}
\quad
    \begin{ytableau}
     {} & {} & {} &*(b)\\
     {} & {} &*(b)&*(b)\\
    *(a)&*(b)&*(b)&*(c)
    \end{ytableau}
\quad
    \begin{ytableau}
     {} & {} & {} &*(c)\\
     {} & {} &*(b)&*(c)\\
    *(a)&*(a)&*(b)&*(c)
    \end{ytableau}
\quad
    \begin{ytableau}
     {} & {} & {} &*(b)\\
     {} & {} &*(b)&*(b)\\
    *(a)&*(a)&*(b)&*(c)
    \end{ytableau}
\end{array}$$
This gives 
\begin{align*}
\zeta_{(3,2,0)}\cdot h_3(x_1, x_2, x_3) =& \zeta_{(3,3,2)}+\zeta_{(3,3,3)}+\zeta_{(4,2,2)}+\zeta_{(4,3,1)}+\zeta_{(4,3,2)}+2\zeta_{(4,3,3)}\\
&\ +\zeta_{(4,4,1)}+2\zeta_{(4,4,2)}+3\zeta_{(4,4,3)}+4\zeta_{(4,4,4)}.
\end{align*}
For $\alpha=1$ and $\beta=1$, we have $s_\Gamma(x_1, x_2, x_3)=s_{(2,1)}(x_1, x_2, x_3)$. In this case,
\begin{align*}
\zeta_{(3,2,0)}\cdot s_{(2,1)}(x_1, x_2, x_3) =&
\zeta_{(3,3,2)}+2\zeta_{(3,3,3)}+\zeta_{(4,2,2)}+2\zeta_{(4,3,1)}+2\zeta_{(4,3,2)}+3\zeta_{(4,3,3)}\\
&\ +\zeta_{(4,4,0)}+2\zeta_{(4,4,1)}+3\zeta_{(4,4,2)}+6\zeta_{(4,4,3)}+8\zeta_{(4,4,4)},
\end{align*}
with the corresponding peakless paths  
$$\ytableausetup{mathmode, boxsize=0.5pc}
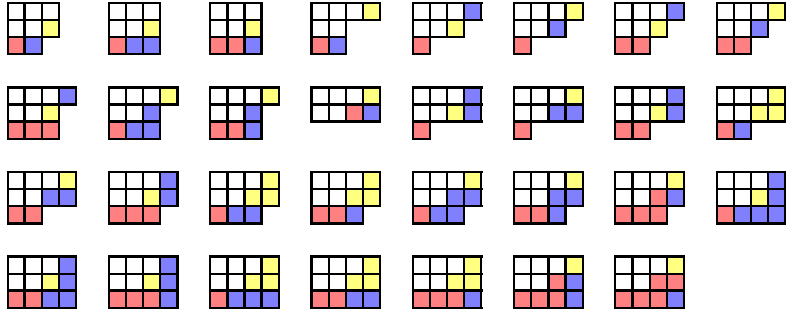
\begin{array}{l}
    \begin{ytableau}
     {} & {} & {} &\none\\
     {} & {} &*(a)\\
    *(b)&*(c)
    \end{ytableau}
\quad
    \begin{ytableau}
     {} & {} & {} &\none\\
     {} & {} &*(a)\\
    *(b)&*(c)&*(c)
    \end{ytableau}
\quad
    \begin{ytableau}
     {} & {} & {} &\none\\
     {} & {} &*(a)\\
    *(b)&*(b)&*(c)
    \end{ytableau}
\quad
    \begin{ytableau}
     {} & {} & {} &*(a)\\
     {} & {} \\
    *(b)&*(c)
    \end{ytableau}
\quad
    \begin{ytableau}
     {} & {} & {} &*(c)\\
     {} & {} &*(a)\\
    *(b)
    \end{ytableau}
\quad
    \begin{ytableau}
     {} & {} & {} &*(a)\\
     {} & {} &*(c)\\
    *(b)
    \end{ytableau}
\quad
    \begin{ytableau}
     {} & {} & {} &*(c)\\
     {} & {} &*(a)\\
    *(b)&*(b)
    \end{ytableau}
\quad
    \begin{ytableau}
     {} & {} & {} &*(a)\\
     {} & {} &*(c)\\
    *(b)&*(b)
    \end{ytableau}
\\[1.5pc]
    \begin{ytableau}
     {} & {} & {} &*(c)\\
     {} & {} &*(a)\\
    *(b)&*(b)&*(b)
    \end{ytableau}
\quad
    \begin{ytableau}
     {} & {} & {} &*(a)\\
     {} & {} &*(c)\\
    *(b)&*(c)&*(c)
    \end{ytableau}
\quad
    \begin{ytableau}
     {} & {} & {} &*(a)\\
     {} & {} &*(c)\\
    *(b)&*(b)&*(c)
    \end{ytableau}
\quad
    \begin{ytableau}
     {} & {} & {} &*(a)\\
     {} & {} &*(b)&*(c)\\
    \none
    \end{ytableau}
\quad
    \begin{ytableau}
     {} & {} & {} &*(c)\\
     {} & {} &*(a)&*(c)\\
    *(b)
    \end{ytableau}
\quad
    \begin{ytableau}
     {} & {} & {} &*(a)\\
     {} & {} &*(c)&*(c)\\
    *(b)
    \end{ytableau}
\quad
    \begin{ytableau}
     {} & {} & {} &*(c)\\
     {} & {} &*(a)&*(c)\\
    *(b)&*(b)
    \end{ytableau}
\quad
    \begin{ytableau}
     {} & {} & {} &*(a)\\
     {} & {} &*(a)&*(a)\\
    *(b)&*(c)
    \end{ytableau}
\\[1.5pc]
    \begin{ytableau}
     {} & {} & {} &*(a)\\
     {} & {} &*(c)&*(c)\\
    *(b)&*(b)
    \end{ytableau}
\quad
    \begin{ytableau}
     {} & {} & {} &*(c)\\
     {} & {} &*(a)&*(c)\\
    *(b)&*(b)&*(b)
    \end{ytableau}
\quad
    \begin{ytableau}
     {} & {} & {} &*(a)\\
     {} & {} &*(a)&*(a)\\
    *(b)&*(c)&*(c)
    \end{ytableau}
\quad
    \begin{ytableau}
     {} & {} & {} &*(a)\\
     {} & {} &*(a)&*(a)\\
    *(b)&*(b)&*(c)
    \end{ytableau}
\quad
    \begin{ytableau}
     {} & {} & {} &*(a)\\
     {} & {} &*(c)&*(c)\\
    *(b)&*(c)&*(c)
    \end{ytableau}
\quad
    \begin{ytableau}
     {} & {} & {} &*(a)\\
     {} & {} &*(c)&*(c)\\
    *(b)&*(b)&*(c)
    \end{ytableau}
\quad
    \begin{ytableau}
     {} & {} & {} &*(a)\\
     {} & {} &*(b)&*(c)\\
    *(b)&*(b)&*(b)
    \end{ytableau}
\quad
    \begin{ytableau}
     {} & {} & {} &*(c)\\
     {} & {} &*(a)&*(c)\\
    *(b)&*(c)&*(c)&*(c)
    \end{ytableau}
\\[1.5pc]
    \begin{ytableau}
     {} & {} & {} &*(c)\\
     {} & {} &*(a)&*(c)\\
    *(b)&*(b)&*(c)&*(c)
    \end{ytableau}
\quad
    \begin{ytableau}
     {} & {} & {} &*(c)\\
     {} & {} &*(a)&*(c)\\
    *(b)&*(b)&*(b)&*(c)
    \end{ytableau}
\quad
    \begin{ytableau}
     {} & {} & {} &*(a)\\
     {} & {} &*(a)&*(a)\\
    *(b)&*(c)&*(c)&*(c)
    \end{ytableau}
\quad
    \begin{ytableau}
     {} & {} & {} &*(a)\\
     {} & {} &*(a)&*(a)\\
    *(b)&*(b)&*(c)&*(c)
    \end{ytableau}
\quad
    \begin{ytableau}
     {} & {} & {} &*(a)\\
     {} & {} &*(a)&*(a)\\
    *(b)&*(b)&*(b)&*(c)
    \end{ytableau}
\quad
    \begin{ytableau}
     {} & {} & {} &*(a)\\
     {} & {} &*(b)&*(c)\\
    *(b)&*(b)&*(b)&*(c)
    \end{ytableau}
\quad
    \begin{ytableau}
     {} & {} & {} &*(a)\\
     {} & {} &*(b)&*(b)\\
    *(b)&*(b)&*(b)&*(c)
    \end{ytableau}
\end{array}$$
\end{Eg}
\bigbreak



\subsection{Parabolic Version of  Corollary \ref{MNforequivariantSchubert}}

Adopting the notion in \cite[\textsection 7.17]{stanley2},  define the {\it height}  $\operatorname{ht}(\mu/\lambda)$ of a rim hook $\mu/\lambda$ to be one less than its number of rows.

\begin{Lemma}\label{Ctildeforpartition} 
Let $\lambda$ and $\mu$ be two partitions  inside the $k\times (n-k)$ rectangle such that $\mu/\lambda$ is a rim hook. Then
\[\operatorname{ht}(\mu/\lambda)=
\operatorname{ht}_k(w_\lambda^{-1}w_\mu )
\quad \text{and}\quad L(\mu/ \lambda) =w_\lambda M(w_\lambda^{-1}w_\mu).\]
\end{Lemma}
   
\begin{proof}
 By \eqref{L-1}, $\operatorname{ht}_k(w_\lambda^{-1}w_\mu )=\texttt{\#}\{i\le k\colon w_\mu(i)\neq w_\lambda(i)\}-1$.
Recall that for $i\le k$, $w_\lambda(i)$ equals the label of  row $k+1-i$ in $\lambda$, and for $i>k$, $w_\lambda(i)$ equals the
label of   column $i-k$ in $\lambda$. It is easily seen that 
 $\{i\le k\colon w_\mu(i)\neq w_\lambda(i)\}$   consists exactly  of the  labels of rows in the  rim hook  $\mu/\lambda$, and so we have  $\operatorname{ht}(\mu/\lambda)=
\operatorname{ht}_k(w_\mu w_\lambda^{-1})$.
Moreover,  $w_\lambda M(w_\lambda^{-1}w_\mu)=\{w_\lambda(i)\colon w_\lambda(i)\neq w_\mu(i)\}$, which is exactly the set $L(\mu/\lambda)$. 
\end{proof}


Comparing the classical MN rule for Schur polynomials with the MN rule for Schubert polynomials  \cite{MS2},
$\mu/\lambda$  is a rim hook of size $r$ if and only if there exists an $(r+1)$-cycle $\eta$ such that $w_{\mu}=w_{\lambda}\eta$ and $w_{\lambda}\leq_k w_{\mu}$
in the ordinary $k$-Bruhat order.  
Hence, combining  Corollary \ref{MNforequivariantSchubert}, \eqref{eq:eqSchubertcoParabolic} and  Lemma \ref{Ctildeforpartition}, we arrive at the following equivariant MN rule over $H^\bullet(\operatorname{\mathcal{G}\it r}(k,n))$. 

\begin{Th}\label{MNforequivariantSchur}
Let $\lambda$ be a partition inside the $k\times (n-k)$ rectangle. For $r\geq 1$, we have the following identity in $H^\bullet_T(\operatorname{\mathcal{G}\it r}(k,n))$:
\begin{equation}\label{eq:MNforequivariantSchur}
[Y(\lambda)]_T\cdot p_{r}(x_{[k]})
=p_{r}( t_{w_{\lambda}[k]})\cdot [Y(\lambda)]_T
+ \sum_{\mu} (-1)^{\operatorname{ht}(\mu/ \lambda)}\cdot h_{r-r'}( t_{L(\mu/\lambda)})\cdot  [Y({\mu})]_T,
\end{equation}
where  $\mu$ ranges over partitions   inside the $k\times (n-k)$ rectangle such that $\mu/\lambda$ is a rim hook of size $r'$   with $1\leq r'\leq r$. 
\end{Th}

\begin{Eg}
Let $k=4$, $n=9$ and $\lambda=(4,2,2, 0)$.  Denote $\sigma_{\lambda}=[Y({\lambda})]_T$. 
For $r=3$, by Theorem \ref{MNforequivariantSchur}, we have 
\begin{align*}
\sigma_{(4,2,2, 0)}\cdot p_3(x_{[4]})
&=p_3( t_1, t_4, t_5, t_8) \cdot\sigma_{(4,2,2,0)} + ( t_8^2+ t_8 t_9+ t_9^2) \cdot\sigma_{(5,2,2,0)} \\
&\quad  
+ ( t_5^2+ t_5 t_6+ t_6^2) \cdot\sigma_{(4,3,2,0)}
+ ( t_1^2+ t_1 t_2+ t_2^2) \cdot\sigma_{(4,2,2,1)}\\
&\quad 
+( t_5+ t_6+ t_7) \cdot\sigma_{(4,4,2,0)}
-( t_4+ t_5+ t_6) \cdot\sigma_{(4,3,3,0)}
\\
&\quad 
+( t_1+ t_2+ t_3) \cdot\sigma_{(4,2,2,2)}-\sigma_{(4,4,3,0)}.
\end{align*}
The corresponding $\mu$'s appearing in \eqref{eq:MNforequivariantSchur}  are illustrated below:
$$\ytableausetup{mathmode, boxsize=0.5pc}
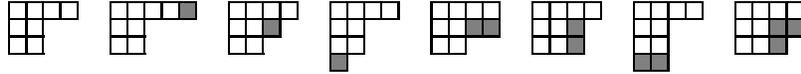
\begin{ytableau}\none\\
{}&{}&{}&{} \\
{}&{}\\
{}&{}\\
\end{ytableau}\quad
\begin{ytableau}\none\\
{}&{}&{}&{}&*(gray)\\
{}&{}\\
{}&{}
\end{ytableau}\quad
\begin{ytableau}\none\\
{}&{}&{}&{}\\
{}&{}&*(gray)\\
{}&{}
\end{ytableau}\quad
\begin{ytableau}\none\\
{}&{}&{}&{}\\
{}&{}\\
{}&{}\\
*(gray)
\end{ytableau}\quad
\begin{ytableau}\none\\
{}&{}&{}&{}\\
{}&{}&*(gray)&*(gray)\\
{}&{}
\end{ytableau}\quad
\begin{ytableau}\none\\
{}&{}&{}&{}\\
{}&{}&*(gray)\\
{}&{}&*(gray)
\end{ytableau}\quad
\begin{ytableau}\none\\
{}&{}&{}&{}\\
{}&{}\\
{}&{}\\
*(gray)&*(gray)
\end{ytableau}\quad
\begin{ytableau}\none\\
{}&{}&{}&{}\\
{}&{}&*(gray)&*(gray)\\
{}&{}&*(gray)
\end{ytableau}
$$
\end{Eg}

\section{Localization and Rim Hook Tableaux}\label{SECT7}

The purpose of this section is to apply  Theorem \ref{MNforequivariantSchur} to establish a relationship connecting  the localization of Schubert classes and the number of standard rim hook tableaux. We   discuss 
how to utilize   this connection to deduce formulas for the number of   standard rim hook tableaux, 
including the formulas obtained  by Alexandersson, Pfannerer,   Rubey  and   Uhlin \cite{APRU} and 
  Fomin and   Lulov \cite{FL}.

Throughout this section,  assume that $r$ and $d$ are fixed positive integers. Let $\Lambda/\lambda$ be a 
skew shape of size $rd$. A {\it standard $r$-rim hook tableau} of shape  $\Lambda/\lambda$ may be thought of as a sequence 
\[
\lambda=\lambda^{(0)}\subset \lambda^{(1)}\subset\cdots\subset \lambda^{(d)}=\Lambda 
\]
of partitions such that for $1\leq i\leq d$, $\lambda^{(i)}/\lambda^{(i-1)}$ is a rim hook of size $r$. A  rim hook of size $r$ is also called an $r$-rim hook. Usually, one assigns each box in the $r$-rim hook $\lambda^{(i)}/\lambda^{(i-1)}$ with the integer  $i$,
so that a standard $r$-rim hook tableau may be intuitively viewed as a filling of $\Lambda/\lambda$ 
with integers $1,2,\ldots, d$. For $r=2$, $\lambda=(1)$ and $\Lambda=(4,4,1)$, there are 
$4$ standard $2$-rim hook tableaux (also called domino tableaux) of shape  $\Lambda/\lambda$, as listed below:
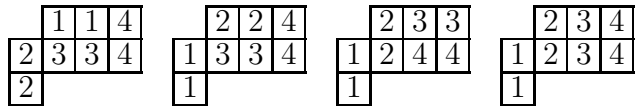
\begin{align*}\label{domino} 
\ytableausetup{mathmode, boxsize=1pc}
\begin{ytableau}
\none & 1 & 1 & 4\\
2 & 3 & 3 & 4\\
2
\end{ytableau}\quad
\begin{ytableau}
\none & 2 & 2 & 4 \\
1 & 3 & 3 & 4\\
1
\end{ytableau}\quad
\begin{ytableau}
\none & 2 & 3 & 3 \\
1 & 2 & 4 & 4\\
1
\end{ytableau}\quad
\begin{ytableau}
\none & 2 & 3 & 4 \\
1 & 2 & 3 & 4\\
1
\end{ytableau}
\end{align*}
In the case $r=1$, a standard $r$-rim hook tableau of shape $\Lambda/ \lambda$ becomes a standard Young tableau
of shape $\Lambda/ \lambda$.
Let $\operatorname{RHT}^r(\Lambda/ \lambda)$ stand for the set of standard $r$-rim hook tableaux of shape $\Lambda/ \lambda$.

As before, for a partition $\Lambda$ inside the $k\times (n-k)$-rectangle, denote by $w_\Lambda\in \S_n$ the Grassmannian permutation associated to $\Lambda$ with descent at $k$.
Let $\phi_{\Lambda}\in \Gr(k,n)$ represent the image of the fixed point $\phi_{w_{\Lambda}}$ under the natural projection $\Fl(n)\to \Gr(k,n)$. 
Consider  the localization morphism  
\begin{equation*}
-|_{\Lambda}: H_T^\bullet(\Gr(k,n)) \to  H_T^\bullet(\phi_{\Lambda})
\cong \mathbb{Q}[ t_1,\ldots, t_n].
\end{equation*}
It is easy to see that for a class $f(x, t)$ in $H_T^\bullet(\Gr(k,n))$,
\begin{equation*}
f(x, t)|_{\Lambda} = f(w_{\Lambda} t, t). 
\end{equation*}
We shall pay attention  to the localization  
 \begin{equation*}
[Y(\lambda)]_T\big|_{\Lambda}=[Y(w_\lambda)]_T\big|_{w_\Lambda}.
\end{equation*}
A   combinatorial formula for $[Y(\lambda)]_T\big|_{\Lambda}$ in terms of excited diagrams  was given by Ikeda and Naruse \cite{INaruse} and Kreiman \cite{VK}, see also Morales,   Pak and   Panova \cite{Pak-1}.
For   general permutations $u,w\in\S_n$, the localization  $[Y(u)]_T\big|_{w}$ can be computed by    Billey's formula \cite{Billey}.   

\subsection{Localization and Rim Hook Tableaux}

Denote by $o(1)$ the space of rational functions in $\mathbb{Q}(z)$ vanishing at one (thus all) of the primitive $r$-th roots of unity. 
Following the convention of analysis, instead of writing $f(z)\in g(z)+o(1)$, we will use the notation  
$f(z)=g(z)+o(1)$. 
For example, $z^r=1+ o(1)$. 

The localization $\left.[Y({\lambda})]_T\right|_{{\Lambda}}$ is a polynomial in $t_1,\ldots, t_n$. After the specialization $ t_i=z^{i}$, $\left.[Y({\lambda})]_T\right|_{{\Lambda}}$ becomes a 
polynomial in $z$, which is denoted 
\[Y_{\lambda, \Lambda}(z)=\left.(\left.[Y({\lambda})]_T\right|_{{\Lambda}})\right|_{t_i=z^i} .\]
For a  rim hook tableau $T\in \operatorname{RHT}^r(\Lambda/\lambda)$, let $\operatorname{ht}(T)$  denote  the total sum of   heights of $r$-rim hooks appearing in  $T$ (Recall that the height of a rim hook is one less than its number of rows). It is known that for distinct $T, T'\in \operatorname{RHT}^r(\Lambda/\lambda)$, $\operatorname{ht}(T)$ and $\operatorname{ht}(T')$ have the same parity, 
see for example \cite[Lemmas 28 and 29]{APRU}. So one may define
\[\text{sgn}(\Lambda/\lambda)=(-1)^{\operatorname{ht}(T)},\] where $T$ is any rim hook tableau in $  \operatorname{RHT}^r(\Lambda/\lambda)$.
In the case that $\Lambda/\lambda$ does not admit any $r$-rim hook tableau, $\text{sgn}(\Lambda/\lambda)$
is understood as zero.

Our main theorem   is the following Laurant expansion.


\begin{Th}\label{RimHooktableaux}
For a    skew shape $\Lambda/\lambda$    of size $rd$,
we have
\begin{equation}\label{OO-II}
\frac{Y_{\lambda, \Lambda}(z)}{Y_{\Lambda, \Lambda}(z)}
= \frac{1}{(z^r-1)^{d}}\left(\mathrm{sgn}(\Lambda/\lambda)\frac{ \texttt{\#}\operatorname{RHT}^r(\Lambda/ \lambda)}{r^d d!}+o(1)\right).
\end{equation} 
\end{Th}

The rest of this subsection is devoted to a proof of 
  Theorem \ref{RimHooktableaux}.

\begin{Lemma}\label{primitiverootlemma1}
Let   $1\leq \ell\leq r$, and let $\zeta$ be a primitive $r$-th root of unity. 
For $r-\ell<i<r$ and distinct integers $a_1,\ldots,a_\ell\in \{0,1,\ldots, r-1\}$, we have
$
h_{i}(\zeta^{a_1},\ldots,\zeta^{a_{\ell}}) = 0,
$
namely, $h_i(z^{a_1},\ldots,z^{a_\ell})= o(1)$.
\end{Lemma}
\begin{proof}

Let 
\[f(z)=\frac{1}{1-z\zeta^{a_1}}\cdots 
\frac{1}{1-z\zeta^{a_\ell}} =\sum_{j\geq 0}h_{j}(\zeta^{a_1},\ldots,\zeta^{a_{\ell}}) z^j \]
denote the generating function of $h_{j}(\zeta^{a_1},\ldots,\zeta^{a_{\ell}})$.
Since  $\zeta$ is a  primitive $r$-th root of unity,
we see that
\[f(z)-z^rf(z)=\frac{1-z^r}{(1-z\zeta^{a_1})\cdots 
(1-z\zeta^{a_\ell})}=\prod
_{\begin{subarray}{c}
0\le j\le r-1\\
j\neq a_1,\ldots,a_\ell
\end{subarray}}
(1-z\zeta^j).\]
Extract  the coefficient of $z^i$ for $r-\ell<i<r$ on both sides. Clearly, the left-hand side contributes  $h_{i}(\zeta^{a_1},\ldots,\zeta^{a_{\ell}})$.
Since   the right-hand side   is a polynomial of degree $r-\ell$, the coefficient of $z^i$ is zero.
This verifies the lemma. 
\end{proof}

\begin{Th}\label{III-I}
Let $\Lambda/\lambda$ be a nontrivial skew shape.  Then
\begin{align}
 Y_{\lambda, \Lambda}(z) 
=& \frac{1}{z^r-1}\left(\sum_{|\mu/ \lambda|=r} \frac{(-1)^{\operatorname{ht}(\mu/ \lambda)}}{|\Lambda/\lambda|}  Y_{\mu, \Lambda}(z) 
+
\sum_{1\leq |\mu'/ \lambda|\le r} o(1)\cdot  Y_{\mu', \Lambda}(z)\right),\label{eq:MNatzeta2}
\end{align}
where the first sum is over  $\mu\subseteq\Lambda$ such that $\mu/\lambda$ is a rim hook of size $r$, and the second sum is over  $\mu'\subseteq\Lambda$ such that $\mu'/\lambda$ is a rim hook of size $r'$ with $1\le r'\le r$. 
\end{Th}

\begin{proof}
Localizing both sides of \eqref{eq:MNforequivariantSchur} at $\Lambda$,  we get 
\begin{align*}
&\left.[Y({\lambda})]_T\right|_{{\Lambda}} \cdot (p_r( t_{w_{\Lambda}[k]})-p_r( t_{w_{\lambda}[k]}))\\[5pt]
&\quad=
\sum_{|\mu/ \lambda|=r} (-1)^{\operatorname{ht}(\mu/ \lambda)}\cdot\left. [Y({\mu})]_T\right|_{{\Lambda}} 
+\sum_{1\leq |\mu'/ \lambda|<r} (-1)^{\operatorname{ht}(\mu'/ \lambda)}\cdot  h_{r-r'}( t_{L(\mu'/\lambda)})\cdot \left.[Y({\mu'})]_T\right|_{{\Lambda}} .
\end{align*} 
To evaluate both sides by setting $t_i=z^i$, we 
need the following two claims. 

\newtheorem*{ClaimA}{Claim A}
\begin{ClaimA}
For   $\mu$ such that $\mu/\lambda$ is  a rim hook of size $r'$   with $1\leq r'<r$, we have 
\[\left.h_{r-r'}( t_{L(\mu/\lambda)})\right|_{t_i=z^i} = o(1).\]
\end{ClaimA}
 
Notice that $L(\mu/\lambda)$ consists of $r'+1$ consecutive integers, say, $m,\ldots,m+r'$,
which are distinct in $\mathbb{Z}/r\mathbb{Z}$ since $r'<r$. By Lemma \ref{primitiverootlemma1}, since $r-(r'+1)<r-r'<r$, we   conclude that 
\[
\left.h_{r-r'}( t_{L(\mu/\lambda)})\right|_{t_i=z^i} = h_{r-r'}(z^{m},\ldots,z^{m+r'})=o(1).
\]
  
\newtheorem*{ClaimB}{Claim B}
\begin{ClaimB}We have 
\[\left.(p_r( t_{w_{\Lambda}[k]})-p_r( t_{w_{\lambda}[k]}))\right|_{t_i=z^i} = (z^r-1)(\,|\Lambda/  
\lambda | + o(1)\,).\]
\end{ClaimB}

This claim can be proved  by induction on $|\Lambda/\lambda|$. Let us first check that case when  $|\Lambda/\lambda|=1$. Assume that the single box $\Lambda/\lambda$ is in row $k+1-i$ and column $j$.
Then we see that
\[\left.(p_r( t_{w_{\Lambda}[k]})-p_r( t_{w_{\lambda}[k]}))\right|_{t_i=z^i}
= (z^{n-i+j})^r-(z^{n-i+j-1})^r=(z^r-1)z^{r(n-i+j-1)}.\]
Clearly,
\[z^{r(n-i+j-1)}=1+o(1),\]
and so the claim follows. We next consider the case 
 $|\Lambda/\lambda|>1$. Take $\lambda'$ such that $\lambda\subsetneq \lambda'\subsetneq\Lambda$. By induction, we have
\[
\left.(p_r( t_{w_{\lambda'}[k]})-p_r( t_{w_{\lambda}[k]}))\right|_{t_i=z^i} = (z^r-1)(\,|\lambda'/  
\lambda | + o(1)\,)
\]
and
\[\left.(p_r( t_{w_{\Lambda}[k]})-p_r( t_{w_{\lambda'}[k]}))\right|_{t_i=z^i} = (z^r-1)(\,|\Lambda/  
\lambda'| + o(1)\,).\]
Adding them together completes the proof of Claim B.

By Claim A and Claim B, we obtain that 
\begin{align}\label{bcq}
&(z^r-1)\cdot (|\Lambda/\lambda|+o(1))\cdot Y_{\lambda, \Lambda}(z)\nonumber \\
&= \sum_{|\mu/ \lambda|=r} (-1)^{\operatorname{ht}(\mu/ \lambda)} \cdot Y_{\mu, \Lambda}(z) + 
\sum_{1\leq |\mu'/ \lambda|<r} o(1)\cdot Y_{\mu', \Lambda}(z).
\end{align}
Since $|\Lambda/\lambda|\neq 0$, we have 
$$\frac{1}{|\Lambda/\lambda|+o(1)}=\frac{1}{|\Lambda/\lambda|}+o(1). $$
So, dividing both sides of \eqref{bcq} by $(z^r-1)\cdot(|\Lambda/\lambda|+o(1))$ yields  that
\begin{align*}
Y_{\lambda, \Lambda}(z)
&= \sum_{|\mu/ \lambda|=r} \left(\frac{(-1)^{\operatorname{ht}(\mu/\lambda)}}{|\Lambda/\lambda|(z^r-1)}+\frac{o(1)}{z^r-1}\right)\cdot Y_{\mu, \Lambda}(z) \\[5pt]
&\ \ \ \ \ \  + 
\sum_{1\leq |\mu'/\lambda|<r}\frac{o(1)}{z^r-1}\cdot Y_{\mu', \Lambda}(z),
\end{align*}
which is the same as \eqref{eq:MNatzeta2}. 
\end{proof}

\begin{Coro}\label{gllambda}
For any skew shape $\Lambda/\lambda$, we have the following estimation
\[
\frac{Y_{\lambda,\Lambda}(z)}{Y_{\Lambda,\Lambda}(z)}= \frac{o(1)}{(z^r-1)^{m+1}},
\]
where  $m$ is the minimum nonnegative integer such that $(m+1)r> |\Lambda/\lambda|$. 
\end{Coro}

\begin{proof}
The proof is by induction on the size $|\Lambda/\lambda|$.
If $|\Lambda/\lambda|=0$, then $m=0$ and the estimation  trivially holds. Now assume $|\Lambda/\lambda|>0$. 
Dividing both sides of  \eqref{eq:MNatzeta2}  by $Y_{\Lambda,\Lambda}(z)$, we obtain that
\begin{align}\label{yylambda}
\frac{Y_{\lambda,\Lambda}(z)}{Y_{\Lambda,\Lambda}(z)}
=& \frac{1}{z^r-1}\left(\sum_{|\mu/ \lambda|=r} \frac{(-1)^{\operatorname{ht}(\mu/ \lambda)}}{|\Lambda/\lambda|} \frac{Y_{\mu,\Lambda}(z)}{Y_{\Lambda,\Lambda}(z)}
+
\sum_{1\leq |\mu'/ \lambda|\le r} o(1)\cdot \frac{Y_{\mu',\Lambda}(z)}{Y_{\Lambda,\Lambda}(z)}\right).
\end{align}
By induction,   
\[\frac{Y_{\mu,\Lambda}(z)}{Y_{\Lambda,\Lambda}(z)}=\frac{o(1)}{(z^r-1)^{m}}\]
and  
\[o(1)\cdot\frac{Y_{\mu',\Lambda}(z)}{Y_{\Lambda,\Lambda}(z)}=o(1)\cdot\frac{o(1)}{(z^r-1)^{m+1}}=\frac{o(1)}{(z^r-1)^{m}}.\] 
Putting the above into \eqref{yylambda}, we get 
\[
\frac{Y_{\lambda,\Lambda}(z)}{Y_{\Lambda,\Lambda}(z)}=\frac{1}{z^r-1}\frac{o(1)}{(z^r-1)^m}=\frac{o(1)}{(z^r-1)^{m+1}}.\qedhere
\] 
\end{proof}

Now we can give a proof of Theorem \ref{RimHooktableaux}.

\begin{proof}[Proof of Theorem \ref{RimHooktableaux}]
Since $r$ divides $|\Lambda/\lambda|$, for any $\lambda\subsetneq\mu\subseteq\Lambda$, we have $dr>|\Lambda/\mu|$. By Corollary \ref{gllambda},
the second summation  of \eqref{yylambda} in the bracket  belongs to  $o(1)\cdot\frac{o(1)}{(z^r-1)^{d}}=\frac{o(1)}{(z^r-1)^{d-1}}$.  By induction, the first summation  of \eqref{yylambda} in the bracket contributes
\begin{align*}
&\sum_{|\mu/\lambda|=r}\frac{(-1)^{\operatorname{ht}(\mu/\lambda)}}{rd}\frac{1}{(z^r-1)^{d-1}}\left(\sum_{T\in \operatorname{RHT}^r(\Lambda/\mu)}\frac{(-1)^{\operatorname{ht}(T)}}{r^{d-1}(d-1)!}+o(1)\right)\\[5pt]
&\quad=\frac{1}{(z^r-1)^{d-1}} \left(\sum_{T\in \operatorname{RHT}^r(\Lambda/\lambda)} \frac{(-1)^{\operatorname{ht}(T)}}{r^dd!}+o(1)\right)\\[5pt]
&\quad=\frac{1}{(z^r-1)^{d-1}} \left(\mathrm{sgn}(\Lambda/\lambda)\frac{ \texttt{\#}\operatorname{RHT}^r(\Lambda/ \lambda)}{r^d d!}+o(1)\right).
\end{align*}
Plugging the above into \eqref{yylambda}
leads to  \eqref{OO-II}, as required. 
\end{proof}

\subsection{Connections with Previous Work}

Still, let $\zeta$ be a  primitive $r$-th root of unity. 
As a direct consequence of  Theorem \ref{RimHooktableaux},
we obtain that 

\begin{Coro} 
For a skew shape $\Lambda/\lambda$ 
of size $rd$,
\begin{equation}\label{eq:hookformulae1}
\texttt{\#}\operatorname{RHT}^r(\Lambda/ \lambda) =\mathrm{sgn}(\Lambda/\lambda)\, r^d d! \cdot  \lim_{z\to \zeta} \frac{Y_{\lambda,\Lambda}(z)}{Y_{\Lambda,\Lambda}(z)}(z^r-1)^{d}.
\end{equation}
\end{Coro}

We reformulate the right-hand side of
\eqref{eq:hookformulae1} based on the following lemma. 

\begin{Lemma}\label{factorialidentity}
We have
\[\lim_{z\to \zeta} \frac{1}{(z^r-1)^d}\prod_{i=1}^{dr}(z^i-1)=(-1)^{rd-d}r^dd!. 
\]
\end{Lemma}

\begin{proof}
We evaluate the left-hand side by grouping the factors. 
Observe  that for   $0\leq a\leq d-1$, 
\[\lim_{z\to \zeta} \frac{1}{z^r-1}\prod_{i=1}^r (z^{ar+i}-1)
= 
(\zeta-1)\cdots (\zeta^{r-1}-1)\cdot\lim_{z\to \zeta} \frac{z^{(a+1)r}-1}{z^r-1}. 
\]
Notice that $(\zeta-1)\cdots (\zeta^{r-1}-1)=(-1)^{r-1}r$, which can be deduced  by substituting $z=1$ in the following identity 
\[
\prod_{i=1}^{r-1}(z-\zeta^i) = \frac{z^r-1}{z-1} = 1+\cdots+z^{r-1}.\]
Moreover,  applying L'Hospital's rule gives
\[
\lim_{z\to \zeta} \frac{z^{(a+1)r}-1}{z^r-1}
= \left.\frac{(a+1)r\cdot z^{(a+1)r-1}}{r\cdot z^{r-1}}\right|_{z=\zeta}=a+1. 
\]
Hence, the limit on the left-hand side in the lemma can
be expressed as 
\[\prod_{a=0}^{d-1} \lim_{z\to \zeta} \frac{1}{z^r-1}\prod_{i=1}^r (z^{ar+i}-1)  = (-1)^{rd-d}(r\cdot 1)\cdots (r\cdot d)=(-1)^{rd-d}r^dd!.\qedhere\]
\end{proof}

By Lemma \ref{factorialidentity},  we may reformulate \eqref{eq:hookformulae1}   as
\begin{equation}\label{rhty}
\texttt{\#}\operatorname{RHT}^r(\Lambda/ \lambda) = \mathrm{sgn}(\Lambda/\lambda)(-1)^{rd-d} \lim_{z\to \zeta} \frac{Y_{\lambda,\Lambda}(z)}{Y_{\Lambda,\Lambda}(z)}(z-1)\cdots (z^{|\Lambda/\lambda|}-1).
\end{equation}
It was deduced  in \cite[Equation (4.10)]{Pak-1} 
that for a skew shape $\Lambda/\lambda$,
\begin{equation}\label{AAA_B}
\frac{Y_{\lambda,\Lambda}(z)}{Y_{\Lambda,\Lambda}(z)}={(-1)^{|\Lambda/\lambda|}}z^{-g(\Lambda)+g(\lambda)}s_{\Lambda/ \lambda}(1,z,z^2,\ldots),
\end{equation}
where, for a partition $\mu=(\mu_1,\ldots, \mu_k)$ inside the $k\times (n-k)$
rectangle,
\[
g(\mu)=\sum_{i=1}^k \binom{\mu_i+d+1-i}{2}.
\]
Moreover, by Stanley \cite[Proposition 7.19.11]{stanley2},
{ \begin{align*}
s_{\Lambda/ \lambda}(1,z,z^2,\ldots)=\frac{1}{(1-z)\cdots(1-z^{|\Lambda/\lambda|})}\sum_{T\in \operatorname{SYT}(\Lambda/\lambda)}z^{{\rm maj}(T)},
\end{align*}}%
where $\operatorname{SYT}(\Lambda/\lambda)$ 
denotes the set of standard Young tableaux  of shape $\Lambda/\lambda$,  and the major index $\text{maj}(T)$ of $T$ is defined to be the   sum of $i$ such that $i+1$ appears in a lower row of $T$ than $i$. 

Combining the above and noting that $|\Lambda/\lambda|=rd$,   \eqref{rhty} can be rewritten as
\begin{align}
\texttt{\#}\operatorname{RHT}^r(\Lambda/ \lambda) &= \mathrm{sgn}(\Lambda/\lambda)(-1)^{rd-d}  \lim_{z\to \zeta}z^{-g(\Lambda)+g(\lambda)} \sum_{T\in \operatorname{SYT}(\Lambda/\lambda)}z^{{\rm maj}(T)}
\nonumber\\[5pt]
&=\mathrm{sgn}(\Lambda/\lambda)(-1)^{rd-d} \frac{1}{\zeta^{g(\Lambda)-g(\lambda)}}
\sum_{T\in \operatorname{SYT}(\Lambda/\lambda)} \zeta^{{\rm maj}(T)}.\label{rhty-I}
\end{align}
By \cite[Proposition 4.7]{Pak-1}, 
\[g(\Lambda)-g(\lambda)=k\cdot|\Lambda/\lambda|+\sum_{\square\in \Lambda/\lambda}c(\square),\]
where $c(\square)=j-i$ for a box $\square$ in row $i$ and column $j$. Since $r$ divides $ |\Lambda/\lambda|$, we get $\zeta^{k\cdot|\Lambda/\lambda|}=1$. 
We next evaluate 
\begin{equation}\label{UUUP-I}
\zeta^{\sum_{\square\in \Lambda/\lambda}c(\square)}.
\end{equation}
Notice that
the integers  $c(\square)$, where $\square$ runs over boxes of any $r$-rim hook in $\Lambda/\lambda$, are distinct in $\mathbb{Z}/r\mathbb{Z}$. So each $r$-rim hook  in $\Lambda/\lambda$ contributes to \eqref{UUUP-I} a value
$\zeta^{1+2+\ldots+r}=\zeta^{\frac{r(r+1)}{2}}$, which is easily checked to be 
$(-1)^{r-1}$. Hence the value in \eqref{UUUP-I}
equals $(-1)^{|\Lambda/\lambda|-d}=(-1)^{rd-d}$, and so
$\zeta^{g(\Lambda)-g(\lambda)}=(-1)^{{rd-d}}$. Putting this into \eqref{rhty-I}, 
we recover  the following formula  deduced by Alexandersson, Pfannerer,   Rubey  and   Uhlin \cite[Corollary 30]{APRU} based on the character theory.

\begin{Th}[{\cite[Corollary 30]{APRU}}]\label{skewrrimhookformula}
Let $\Lambda/\lambda$ be a skew shape of size $rd$, and $\zeta$ be a primitive   $r$-th root of unity. Then   
\begin{align*}
\texttt{\#} \operatorname{RHT}^r(\Lambda/ \lambda) 
= \mathrm{sgn}(\Lambda/\lambda)
\sum_{T\in \operatorname{SYT}(\Lambda/\lambda)} \zeta^{\operatorname{maj}(T)}.
\end{align*}
\end{Th}
 




We now restrict   to   straight shapes, that is, $\lambda=\varnothing$ and $\Lambda$ has size $rd$. 
In this case, $Y_{\lambda,\Lambda}(z)=1$.
By \cite[\textsection 4.2]{Pak-1}, 
\[
Y_{\Lambda,\Lambda}(z)=\prod_{\square\in \Lambda}(z^{h(\square)}-1)\ \text{ up to a power of $z$},
\]
which also follows from  \eqref{AAA_B} 
and \cite[Corollary 7.21.3]{stanley2}.
In this formula, $h(\square)$ is the hook length   of a box $\square\in \Lambda$, which is the  number of boxes directly to its right or directly below it, including  the box itself.
So, by \eqref{eq:hookformulae1},  we see that
\begin{equation}\label{eq:hookstrShp2}
\texttt{\#} \operatorname{RHT}^r(\Lambda)
= r^dd! \cdot \lim_{z\to \zeta}\left|(z^r-1)^d\prod_{\square\in \Lambda} \frac{1}{z^{h(\square)}-1}\right|,
\end{equation}
where $|\cdot|$ denotes the modulus of complex numbers. 
Notice  that if $r\mid h(\square)$, by L'Hospital's rule,
\[\lim_{z\to \zeta}\left|\frac{z^{r}-1}{z^{h(\square)}-1}\right|
=\frac{r}{h(\square)},\]
and if $r\nmid h(\square)$,
$$\lim_{z\to\zeta}\left|z^{h(\square)}-1\right|\neq 0.$$
This observation immediately reveals    the well-known criterion that $\texttt{\#} \operatorname{RHT}^r(\Lambda)\neq 0$ if and only if
\begin{equation}\label{eq:numberofevenhooks}
\texttt{\#}\left\{\square \in \Lambda: r\mid h(\square)\right\}=d,
\end{equation}
see for example Fomin and   Lulov \cite[Corollary 2.7]{FL}.
Now, when $  \operatorname{RHT}^r(\Lambda) $ is nonempty, \eqref{eq:hookstrShp2} can be expressed as
\begin{equation}\label{eq:tocomputenorm}
\texttt{\#} \operatorname{RHT}^r(\Lambda) = \frac{r^d d!}{\prod\limits_{\begin{subarray}{c}\square \in \Lambda\\ r\mid h(\square)\end{subarray}} h(\square)}\cdot \frac{1}{|N|},
\end{equation}
where 
\begin{equation}
N=\prod_{\square\in \Lambda}\begin{cases}
1-\zeta^{h(\square)}, & \text{if $r\nmid h(\square)$},\\[5pt]
1/r, &\text{if $r\mid h(\square)$}.
\end{cases}
\end{equation}

\begin{Lemma}\label{123qwe}
We have  $|N|=1$. 
\end{Lemma}

If $r=1$, then it is trivially true that $|N|=1$. 
If $r=2$, then $\zeta=-1$ and so $1-\zeta^{h(\square)}=2$ for $h(\square)$ odd, which implies  that $|N|=1$ since exactly half of the hook lengths $h(\square)$ are odd.  
However, we do not have an elementary proof of Lemma \ref{123qwe} for general $r$.

\begin{proof}[Proof of Lemma \ref{123qwe}.]
By \eqref{eq:tocomputenorm},  $|N|$  does not depend on the choice of $\zeta$, thus the Galois group $G=\operatorname{Gal}(\mathbb{Q}[\zeta]/\mathbb{Q})$ acts on $|N|$ trivially. That is, 
$$|N|^{|G|}=\prod_{\sigma\in G} |\sigma N|=\big|\operatorname{Nm}(N)\big|,$$
where $\operatorname{Nm}\colon \mathbb{Q}[\zeta]\to \mathbb{Q}$ is the norm of the field extension $\mathbb{Q}[\zeta]/\mathbb{Q}$, see \cite[\textsection VI.5]{SLAlg}. 
To conclude $|N|=1$, it suffices to show that the norm of $N$ is $1$. 

For any integer $h$, the power $\zeta^h$ is a primitive $\frac{r}{\operatorname{gcd}(h,r)}$-th root of unity, and so  the norm of $1-\zeta^h$ for the field extension $\mathbb{Q}[\zeta]/\mathbb{Q}$ depends only on $\operatorname{gcd}(h,r)$, see \cite[Theorem VI.5.1 and Ex VI.19]{SLAlg}. 
This, together with the following claim, 
will lead to a proof that the norm of $N$ is 1.

\begin{Claim} For $m\mid r$, 
\begin{equation}\label{eq:Mobiusargument1}
\texttt{\#}\left\{\square\in \Lambda\colon \operatorname{gcd}(h(\square),r)=m\right\}= \texttt{\#}\left\{i\in [rd]\colon \operatorname{gcd}(i,r)=m\right\}.
\end{equation}
\end{Claim}

 Notice that for any integer $i$, $r=\operatorname{gcd}(i,r)$ if and only if $r\mid i$, and $r>\operatorname{gcd}(i,r)$ if and only if $r\nmid i$. 
By  the above claim, we see that
\begin{align*}
\operatorname{Nm}(N)
& =\prod_{\square\in \Lambda}
\begin{cases}
\operatorname{Nm}(1-\zeta^{h(\square)}), & \text{if $r\nmid h(\square)$},\\[5pt]
\operatorname{Nm}(1/r), &\text{if $r\mid h(\square)$},
\end{cases}\\
& =\prod_{i\in [rd]}
\begin{cases}
\operatorname{Nm}(1-\zeta^{i}), & \text{if $r\nmid i$},\\[5pt]
\operatorname{Nm}(1/r), &\text{if $r\mid i$},
\end{cases}\\
& =\operatorname{Nm}\left(\prod_{i\in [rd]}
\begin{cases}
1-\zeta^{i}, & \text{if $r\nmid i$},\\[5pt]
1/r, &\text{if $r\mid i$},
\end{cases}\right)\\[5pt]
&=\operatorname{Nm}\left(\frac{1}{r^d}\left((1-\zeta)\cdots(1-\zeta^{r-1})\right)^d\right).
\end{align*}
Setting $d=1$ in Lemma \ref{factorialidentity}, we get $(\zeta-1)\cdots(\zeta^{r-1}-1)=(-1)^{r-1}r$, and so $\operatorname{Nm}(N)=1$.

It remains to  prove the claim in \eqref{eq:Mobiusargument1}.
For $m\mid r$, by \cite[Lemma 7.3]{Stanley84}, if $\operatorname{RHT}^{r}(\Lambda)$ is nonempty, then $\operatorname{RHT}^{m}(\Lambda)$ is nonempty. 
So it follows from  \eqref{eq:numberofevenhooks} that
\begin{equation}\label{eq:Mobiusargument2}
\texttt{\#}\left\{\square\in \Lambda\colon m\mid h(\square)\right\} = \tfrac{r}{m}d
= \texttt{\#}\left\{i\in [rd]\colon m\mid i\right\}.
\end{equation}
We explain that the claim in \eqref{eq:Mobiusargument1} follows from \eqref{eq:Mobiusargument2} by M\"obius inversion.
For any finite sequence of   integers $a_1,\ldots,a_{\ell}$ and any $m\mid r$, denote
\begin{equation*}
f_m=\texttt{\#}\left\{i\in [\ell]\colon m\mid a_i\right\},\qquad
g_m=\texttt{\#}\left\{i\in [\ell]\colon \operatorname{gcd}(r,a_i)=m\right\}.
\end{equation*}
Since 
$m\mid a_i$ if and only if   
$m\mid \operatorname{gcd}(r,a_i),$
it is not hard to check that   $f_{r/m}=\sum_{\delta\mid m}g_{r/\delta}.$ 
Therefore,
\[g_{r/m}=\sum_{\delta\mid m} f_{r/\delta}\cdot \mu\left(\tfrac{m}{\delta}\right),\]
where $\mu$ is  the classical number-theoretic M\"obius function. 
So, $f_m$ and $g_m$ are determined by each other, and in particular, \eqref{eq:Mobiusargument2} implies \eqref{eq:Mobiusargument1}. 
\end{proof}

The above lemma along with \eqref{eq:tocomputenorm}
allows us to reach the following hook  formula
 due to  Fomin and   Lulov \cite[Corollary 2.2]{FL}. 

\begin{Th}[{\cite[Corollary 2.2]{FL}}]\label{straightrrimhookformula}
For a partition  $\Lambda$ of size $rd$ with $\operatorname{RHT}^r(\Lambda)$ nonempty, we have 
\[
\texttt{\#} \operatorname{RHT}^r(\Lambda) = \frac{r^dd!}{\prod\limits_{\begin{subarray}{c}\square \in \Lambda\\ r\mid h(\square)\end{subarray}} h(\square)}. 
\]
\end{Th}

Notice  that in the case  $r=1$, Theorem \ref{straightrrimhookformula}
specifies to the classical hook   formula for standard Young tableaux.



\section{Conjectures}\label{CCC_I}

In this section, we discuss some positivity conjectures about CSM classes. 
We mainly adhere to the notation   from  \cite{AMSS17}. 
Let $G/B$ be the flag variety for a reductive group $G$ over $\mathbb{C}$ with $B$ a fixed Borel subgroup, and $W$ be the associated Weyl group. 
For an element $w$  in $W$, let
$X(w)^\circ=BwB/B$ denote the Schubert cell, and   $X(w) = \overline{X(w)^\circ}$ denote the Schubert variety. We also let 
  $Y(w)^\circ= B^-wB/B$ be the opposite  Schubert cell,  and    $Y(w)= \overline{Y(w)^\circ}$ be the opposite Schubert variety,
where $B^-=w_0Bw_0$ is the opposite Borel subgroup ($w_0$ is the longest element of $W$). 
Here, we remark that in the above sections of this paper, we are only concerned with  opposite Schubert cells/varieties of type $A$, and so we
abbreviate ``opposite'' and  simply call them
Schubert cells/varieties. 

The following is our main conjecture. As will be explained, its geometric version is recently proposed independently by Kumar  {\cite[Conjecture B]{Kumar}}.

\begin{Conj}\label{non-equivariantpositivity}
For $u, v\in W$,
\begin{equation}\label{eq:mainconj}
c_{\mathrm{SM}}(Y(u)^\circ)\cdot [Y(v)]  \in \sum_{w\in W} \mathbb{Z}_{\geq0}\cdot c_{\mathrm{SM}}(Y(w)^\circ),
\end{equation}
where $\mathbb{Z}_{\geq0}$ is the set of nonenegative integers.
\end{Conj}

Theorem \ref{NoneqPierirulehook} confirms  Conjecture \ref{non-equivariantpositivity} in type $A$ for $v$ being a Grassmannian permutation  of hook shape. 
This conjecture has been checked  for groups of types $A_{\leq 7}$, $B_{\leq 4}$, $C_{\leq 4}$, $D_{\leq 4}$, $G_{2}$.

 Taking the lowest degree component of \eqref{eq:mainconj} leads to the famous  positivity for the structure constants of Schubert classes:
\begin{equation*}
[Y(u)]\cdot [Y(v)]  \in \sum_{w\in W} \mathbb{Z}_{\geq0}\cdot [Y(w)].
\end{equation*}
Another implication of \eqref{eq:mainconj} is
\begin{equation}\label{cschubt}
c_{\mathrm{SM}}(Y(u)^\circ) \in \sum_{w\in W} \mathbb{Z}_{\geq0}\cdot [Y(w)],
\end{equation}
which was conjectured in \cite{AM} and proved in \cite{AMSS17} using characteristic cycles of $\mathcal{D}$-modules.
Actually,  \eqref{cschubt} is a direct consequence of 
Kumar's conjecture {\cite[Conjecture B]{Kumar}},
which turns out to be equivalent to Conjecture \ref{non-equivariantpositivity}.

\begin{Conj}[{\cite[Conj. B]{Kumar}} $\Leftrightarrow$ Conj. \ref{non-equivariantpositivity}]\label{EEE_I}
For $u,v\in W$, the CSM class $c_{\mathrm{SM}}(X(u)^\circ\cap Y(v)^\circ)$  of the Richardson cell $X(u)^\circ\cap Y(v)^\circ$ is effective, namely,
\[c_{\mathrm{SM}}(X(u)^\circ\cap Y(v)^\circ) \in \sum_{w\in W} \mathbb{Z}_{\geq0}\cdot [Y(w)].
\]
\end{Conj}

In particular, for   $u\in W$, 
$$c_{\mathrm{SM}}(Y(u)^\circ)=\sum_{w\in W} c_{\mathrm{SM}}(X(w)^\circ \cap Y(u)^\circ)$$
is effective assuming Conjecture \ref{non-equivariantpositivity}. 

We explain the equivalence between Conjecture \ref{non-equivariantpositivity} and Conjecture \ref{EEE_I}. 
For a constructible subset $Z$ of a smooth variety $X$,   its  {\it Segre--Schwartz--MacPherson} (SSM) class 
is defined to be 
$$s_{\mathrm{SM}}(Z)=c_{\mathrm{SM}}(Z)/c(\mathscr{T}_X),$$
where $c(\mathscr{T}_X)$ is the total Chern class of the tangent bundle of $X$.  
It was shown in \cite[Theorem 7.1]{AMSS17} that under Poincar\'e pairing, CSM classes and SSM classes of Schubert cells are dual. Specifically, for   $u, w\in W$, we have
$$\int_{G/B} s_{\mathrm{SM}}(X(u)^\circ)\cdot c_{\mathrm{SM}}(Y(w)^\circ)=
\begin{cases}
1, & u=w,\\
0, & u\neq w.
\end{cases}$$
So \eqref{eq:mainconj} reads
$$\int_{G/B} s_{\text{SM}}(X(u)^\circ)\cdot c_{\text{SM}}(Y(v)^\circ)\cdot [Y(w)]\geq 0.$$
In other words,
$s_{\mathrm{SM}}(X(u)^\circ)
\cdot c_{\mathrm{SM}}(Y(v)^\circ)$ is effective. 
Since $X(u)$ and $Y(v)$ are stratified transverse in the sense of \cite{Sch17}, one can conclude that (see also \cite[Theorem 3.6]{AMSS17}) 
\[
s_{\mathrm{SM}}(X(u)^\circ)
\cdot c_{\mathrm{SM}}(Y(v)^\circ)
=c_{\mathrm{SM}}(X(u)^\circ\cap Y(v)^\circ),\]
which is exactly what we require.

 Conjecture \ref{non-equivariantpositivity} also implies a conjecture by Mihalcea \cite[Page 6]{MihalceaSlides} that the structure constants of SSM classes are alternately nonnegative.  

\begin{Conj}[\cite{MihalceaSlides}]\label{CSMpositiveConj}
For   $u,v\in W$, 
$$
s_{\mathrm{SM}}(Y(u)^\circ)\cdot 
s_{\mathrm{SM}}(Y(v)^\circ)\in 
\sum_{w\in W} (-1)^{\ell(u)+\ell(v)-\ell(w)}\,\mathbb{Z}_{\geq0}\cdot s_{\mathrm{SM}}(Y(w)^\circ).$$
\end{Conj}

Since the Schubert expansion in \eqref{cschubt} of CSM classes is  nonnegative,  Conjecture \ref{non-equivariantpositivity} implies
\begin{equation}\label{DDS_I}
c_{\mathrm{SM}}(Y(u)^\circ)\cdot 
c_{\mathrm{SM}}(Y(v)^\circ)\in 
\sum_{w\in W} \mathbb{Z}_{\geq0}\cdot c_{\mathrm{SM}}(Y(w)^\circ).
\end{equation}
By \cite[Theorem 7.5]{AMSS17} (see also \cite[Equation (2)]{Su3}), after specialization to the nonequivariant case, 
\begin{equation*}
c_{\mathrm{SM}}(Y(u)^\circ)=(-1)^{\ell(u)}\overline{s_{\mathrm{SM}}(Y(u)^\circ)},
\end{equation*}
where $\overline{c} = \sum (-1)^i c_i$ for $c=\sum c_i$ with $c_i\in H^{2i}(G/B)$. 
Applying the involution  $\overline{c }$  to both sides of \eqref{DDS_I}, we get the assertion in 
Conjecture \ref{CSMpositiveConj}.

 {Notice that a formula for the  coefficients   in Conjecture \ref{CSMpositiveConj} was derived   in \cite{Su3}, but is not manifestly nonnegative.}

Just like the proof of \cite[Theorem 3]{KJ2}, using \cite[Theorem 1.2]{Sch17}, 
Conjecture \ref{CSMpositiveConj}  is equivalent to the following   geometric  interpretation about   intersections of Schubert varieties  \cite[Conjecture D]{Kumar}. 

\begin{Conj}[{\cite[Conj. D]{Kumar}} $\Leftrightarrow$ Conj. \ref{CSMpositiveConj}]
For generic $g,g',g''\in G$, the Euler characteristic 
\[\chi_c\big(\, gY(u)^\circ\cap g'Y(v)^\circ\cap g''Y(w)^\circ\,\big) \in (-1)^{\rho}\cdot \mathbb{Z}_{\geq0}
\]
where $\rho= \dim(gY(u)\cap g'Y(v)\cap g''Y(w))$. 
\end{Conj}

The Grassmannian (more generally, partial flag varieties of step $\leq 3$) analogue  of the above conjecture was confirmed  by Knutson and Zinn-Justin \cite{KJ2}.

{
In  the case of type A (i.e., the case of $\Fl(n)$), we are able to  use  the Pieri formula in Theorem \ref{CSMPieriRuleforeh} to prove a weaker form of Kumar's Conjecture \ref{EEE_I}. 
Write $\delta = (n-1,n-2,\ldots,1,0)$.
Since $\Fl(n)$ can be constructed as an iterated projective bundle, we have 
$$H^\bullet(\Fl(n);\mathbb{Q})=\bigoplus_{a\leq \delta} \mathbb{Q}\cdot x^{a},$$
where $a=(a_1,\ldots,a_n)\leq \rho$ means $0\leq a_i\leq n-i$. 
This   also follows from the following property  of Schubert polynomials \cite{Ma}:
$$[Y(w)]=\mathfrak{S}_w(x) \in \sum_{a\leq \delta}\mathbb{Z}_{\geq0}\cdot x^a.$$

We say that a class $\alpha\in H^\bullet(\Fl(n))$ is monomial-positive if the coefficient of   $x^a$ in $\alpha$ is nonnegative for any $a\leq \delta$. 
Clearly,  an effective class is monomial-positive.  

\begin{Th}\label{MonomialPositive}
For $u,v\in \S_n$, the CSM class $c_{\mathrm{SM}}(X(u)^\circ\cap Y(v)^\circ)$ of the Richardson cell
$X(u)^\circ\cap Y(v)^\circ$ is monomial-positive.
\end{Th}

\begin{proof}
Write 
\begin{equation*}
\mathfrak{S}_w(x) = \sum_{a\leq \delta} K_{w,a} \, x^a.
\end{equation*}
As  noticed by Postnikov and Stanley  \cite[\textsection 17]{PS},
\begin{equation*}
e_{w_0(\delta-a)}(x) = \sum_{w\in \S_n} K_{w,a} \mathfrak{S}_{ww_0}(x),
\end{equation*}
where 
$$e_{a}(x)=e_{a_2}(x_{[1]})e_{a_3}(x_{[2]})\cdots e_{a_n}(x_{[n-1]}).$$
Let $k_{u,v}^a$ be the coefficient  of $x^a$ in $c_{\mathrm{SM}}(X(u)^\circ\cap Y(v)^\circ)$.
Then we have   
\begin{align}\label{kuva}
k_{u,v}^a& =\sum_w K_{w,a}\int_{\Fl(n)}c_{\mathrm{SM}}(X(u)^\circ\cap Y(v)^\circ)\cdot[Y(w_0w)]\nonumber \\
& 
= \sum_w K_{w,a}\int_{\Fl(n)} c_{\mathrm{SM}}(Y(v)^\circ)\cdot [Y(w_0w)]\cdot
s_{\mathrm{SM}}(X(u)^\circ).
\end{align}
We need   an involution $\Theta$ over $\Fl(n)$, which sends a flag 
$$0=V_0\subseteq V_1\subseteq \cdots \subseteq V_n=\mathbb{C}^n
$$
to its annihilator 
$$0=V_n^{\perp}\subseteq V_{n-1}^\perp\subseteq \cdots \subseteq V_0^\perp=\mathbb{C}^n
$$
under any identification $(\mathbb{C}^n)^* = \mathbb{C}^n$. 
By definition, we have
$$\Theta(Y(w)^\circ)=Y(w_0ww_0)^\circ,\qquad
\Theta(X(w)^\circ)=X(w_0ww_0)^\circ.$$
Thus  $\Theta$ sends the Schubert class (resp., CSM class) of $w$ to the Schubert class (resp., CSM class) of $w_0ww_0$. 
Applying the involution $\Theta$ to \eqref{kuva},  we obtain
that
\begin{align*}
k_{u,v}^a& 
= \sum_w K_{w,a}\int_{\Fl(n)} c_{\mathrm{SM}}(Y(w_0vw_0)^\circ)\cdot [Y(ww_0)]\cdot
s_{\mathrm{SM}}(X(w_0uw_0)^\circ) \\
& 
= \int_{\Fl(n)} c_{\mathrm{SM}}(Y(w_0vw_0)^\circ)\cdot \sum_w K_{w,a}\mathfrak{S}_{ww_0}(x)\cdot
s_{\mathrm{SM}}(X(w_0uw_0)^\circ) \\
& 
= \int_{\Fl(n)} c_{\mathrm{SM}}(Y(w_0vw_0)^\circ)\cdot e_{w_0(\delta-a)}(x)\cdot
s_{\mathrm{SM}}(X(w_0uw_0)^\circ). 
\end{align*}
This implies  that $k_{u,v}^a$ is the coefficient of $c_{\mathrm{SM}}(X(w_0uw_0)^\circ)$ in 
$$c_{\mathrm{SM}}(Y(w_0vw_0)^\circ)\cdot e_{w_0(\delta-a)}(x).$$
By repeatedly  using Theorem \ref{CSMPieriRuleforeh}, we see that this coefficient is nonnegative. 
\end{proof}
}

{ It is worth mentioning that in  the case of type A, }
Conjecture \ref{non-equivariantpositivity} can be implied by a conjecture of Fomin and Kirillov. 
The {\it Fomin--Kirillov algebra} $\mathcal{E}_n$ is generated by $\mathbf{x}_{ab}$ for $1\leq a< b\leq n$, subject to the following relations \begin{gather*}
\mathbf{x}_{ab}^2=0,\\
\mathbf{x}_{ab}\mathbf{x}_{bc}=
\mathbf{x}_{ac}\mathbf{x}_{ab}+
\mathbf{x}_{bc}\mathbf{x}_{ac},\\
\mathbf{x}_{bc}\mathbf{x}_{ab}=
\mathbf{x}_{ab}\mathbf{x}_{ac}+
\mathbf{x}_{ac}\mathbf{x}_{bc},\\\label{eq:commutativeRelation}
\mathbf{x}_{ab}\mathbf{x}_{cd}=\mathbf{x}_{cd}\mathbf{x}_{ab} \quad \text{for distinct $a,b,c,d$}.
\end{gather*}
By \cite[Lemma 5.1]{FK}, the {\it Dunkl elements}, which are defined as  
\begin{equation*}
\theta_i = -\sum_{a<i} \mathbf{x}_{ai}+\sum_{i<b} \mathbf{x}_{ib},
\end{equation*}
are pairwisely commutative. 
As  pointed out in \cite[Remark 3.2]{Lee},   $\mathcal{E}_n$ acts on $H^\bullet(\Fl(n))$ (on the right) by 
\begin{equation}\label{eq:extendedBruhataction}
{}c_{\mathrm{SM}}(Y(w)^\circ)* \mathbf{x}_{ab} = \begin{cases}
{}c_{\mathrm{SM}}(Y(wt_{ab})^\circ), & \ell(wt_{ab})\geq \ell(w)+1, \\[5pt]
0, & \text{otherwise}. 
\end{cases}
\end{equation}
By the CSM Chevalley formula \cite{AMSS17}, for     $w\in \S_n$ and $i\in [n]$,
\[
c_{\mathrm{SM}}(Y(w)^\circ)\cdot x_i = c_{\mathrm{SM}}(Y(w)^\circ)*\theta_i.
\]
Hence, 
\begin{align*}
c_{\mathrm{SM}}(Y(w)^\circ)\cdot [Y(u)] 
& =c_{\mathrm{SM}}(Y(w)^\circ)\cdot \mathfrak{S}_{u}(x_1,\ldots,x_n)\\[5pt]
&= c_{\mathrm{SM}}(Y(w)^\circ)* \mathfrak{S}_{u}(\theta_1,\ldots,\theta_n),
\end{align*} 
where $\S_u$ is the Schubert polynomial of $u$.

{

\begin{Conj}[{\cite[Conj. 8.1]{FK}}]
For $w\in \S_n$,
\[\mathfrak{S}_w(\theta_1,\ldots,\theta_n)\in \mathcal{E}_n^+,\]
where $\mathcal{E}_n^+$ is the cone of all nonnegative integer linear combinations of 
  (noncommutative) monomials in the generators $\mathbf{x}_{ab}$ for $a<b$. 
\end{Conj}

}

In view of \eqref{eq:extendedBruhataction},
the above  conjecture implies that the CSM expansion of $c_{\mathrm{SM}}(Y(w)^\circ)\cdot [Y(u)]$ is nonnegative, 
as stated in Conjecture \ref{non-equivariantpositivity}.

Lastly, we investigate  the equivariant setting of Conjecture \ref{non-equivariantpositivity}. 
Let $T$ be the maximal torus of the Borel subgroup $B$. 
Denote by $\{\alpha_i\}_{i\in I}$ the set of simple roots.
Let $\alpha\in H^\bullet_T(\mathsf{pt})$ be the first Chern class of $\mathbb{C}_{\alpha}$. 
For example, in type $A$, $\alpha_i=-t_i+t_{i+1}$ due to our convention in Subsection \ref{AAAC}. 
See also the remarks before Theorem 6.4 in Chapter 10 of  \cite{AF} for the convention about signs.


{  Let $\mathbb{Z}_{\geq0}[\alpha_i]_{i\in I}$ be the set of polynomials in $\alpha_i$'s with nonnegative integer coefficients.}
The equivariant analogue  of Conjecture \ref{non-equivariantpositivity} can be stated as follows. 

\begin{Conj}\label{equivariantpositivity}
For   $u,v\in W$,
\begin{equation*}
c_{\mathrm{SM}}^T(Y(u)^\circ) \cdot [Y(v)]_T  \in \sum_{w\in W} \mathbb{Z}_{\geq0}[\alpha_i]_{i\in I}\cdot c_{\mathrm{SM}}^T(Y(w)^\circ).
\end{equation*}
\end{Conj}

Theorem \ref{LRruleforehCSM} confirms this conjecture in the case of type A when $v$ is a Grassmannian permutation corresponding to a one row/column partition. 
Actually,  thanks to   Billey's formula \cite{Billey}, the coefficients are localizations of Schubert classes which   automatically lie in $\mathbb{Z}_{\geq0}[\alpha_i]_{i \in I}$.

Notice also that Conjecture  \ref{equivariantpositivity} generalizes  Graham’s positivity theorem \cite[Corollary 4.1]{Graham}: 
\begin{equation*}
[Y(u)]_T \cdot [Y(v)]_T  \in \sum_{w\in W} \mathbb{Z}_{\geq0}[\alpha_i]_{i\in I}\cdot [Y(w)]_T.
\end{equation*}

Along the same line as the nonequivariant case, Conjecture \ref{equivariantpositivity} has an equivalent geometric description  for the CSM classes of the Richardson cells.

\begin{Conj}[$\Leftrightarrow$ Conj. \ref{equivariantpositivity}]\label{EEDD}
For   $u,v\in W$,  
$$c^T_{\mathrm{SM}}(X(u)^\circ\cap Y(v)^\circ)\in \sum_{w\in W} \mathbb{Z}_{\geq 0}[\alpha_i]_{i\in I}\cdot [X(w)]_T.$$
\end{Conj}

 Conjecture \ref{EEDD} implies the following conjecture of Aluffi and Mihalcea \cite[Conj. 2]{AM}:   
$$
c_{\mathrm{SM}}^T(X(u)^\circ)
\in \sum_w \mathbb{Z}_{\geq0}[\alpha_i]_{i\in I}\cdot  [X(v)]_T.$$



\begin{thebibliography}{99}

\bibitem{AF}
D. Anderson and W. Fulton,
Equivariant Cohomology in Algebraic Geometry (draft), 
Available at \url{https://people.math.osu.edu/anderson.2804/ecag/index.html}. 


\bibitem{APRU}
P. Alexandersson, S. Pfannerer, M. Rubey and J. Uhlin, Skew characters and cyclic sieving, Forum Math. Sigma 9 (2021), e41.

\bibitem{AM0}
P. Aluffi and L. Mihalcea, Chern classes of Schubert cells and varieties, J. Algebraic  Geom. 18 (2009),  63--100.


\bibitem{AM}
P. Aluffi and  L. Mihalcea,  Chern--Schwartz--MacPherson classes for Schubert cells in flag manifolds,
Compos. Math. 152 (2016), 2603--2625.



\bibitem{AMSS17}
P. Aluffi, L. Mihalcea, J. Sch\"{u}rmann and C. Su, Shadows of characteristic
cycles, Verma modules, and positivity of Chern--Schwartz--MacPherson classes of Schubert
cells,  arXiv:1709.08697v3.


\bibitem{BS}
N. Bergeron and F. Sottile, Schubert polynomials, the Bruhat order, and the geometry of flag
manifolds, Duke Math. J. 95 (1998),  373--423.

\bibitem{Sottile}
N. Bergeron and F. Sottile, A Pieri-type formula for isotropic flag manifolds, Trans. Amer. Math. Soc.  354 (2002), 2659--2705.

\bibitem{BGG}
I. N. Bern\u{s}te\u{i}n, I. M. Gel'fand and S. I. Gel'fand, Schubert cells, and the cohomology of the
spaces $G/P$, Uspekhi Mat. Nauk 28 (1973), 3--26. 


\bibitem{Borel}
A. Borel, Sur la cohomologie des espaces fibr\'{e}s principaux et des espaces homog\'{e}nes de
groupes de Lie compacts, Ann.  Math.   57 (1953), 115--207 

\bibitem{Billey}
 S. Billey, Kostant polynomials and the cohomology ring for $G/B$, Duke Math. J.  96 (1999), 205--224.

 

\bibitem{BB}
A. Bj\"orner and F. Brenti, Combinatorics of Coxeter Groups, Grad. Texts in Math., Vol. 231, Springer, New York, 2005.


\bibitem{BMO}
A. Braverman, D. Maulik and A. Okounkov,  Quantum cohomology of the Springer resolution, Adv. Math.  227 (2011), 421--458.





\bibitem{Dyer}
M.J. Dyer, On the ``Bruhat graph'' of a Coxeter system, Compos. Math. 
78 (1991), 185--191.


\bibitem{GW}
B. Elias,  S. Makisumi,  U. Thiel  and G. Williamson, Introduction  to  Soergel Bimodules, Vol. 5, RSME Springer Series, Springer International Publishing, 2020.


\bibitem{FK}
S. Fomin and A. N. Kirillov, Quadratic algebras, Dunkl elements, and Schubert calculus,
Adv. Geom.  172 (1999), 147-182.

\bibitem{FL}
S. Fomin and N. Lulov,
On the number of rim hook tableaux
J. Math. Sci. (New York)  87 (1997),   4118--4123.

\bibitem{FandS}
S. Fomin and R.P. Stanley, Schubert polynomials and the
nilCoxeter algebra, Adv. Math. 103 (1994), 196--207.


\bibitem{Fulton}
W. Fulton, Intersection Theory,  Springer-Verlag, Berlin-Heidelberg-New York, 1984.


\bibitem{Graham}
W. Graham, Positivity in equivariant Schubert calculus,
Duke Math. J.  109 (2001),  599--614.

\bibitem{HpGr}
J. Huh, Positivity of Chern classes of Schubert cells and varieties,  J. Algebraic Geom.  25 (2016), 177--199.

\bibitem{INaruse}
T. Ikeda and H. Naruse, Excited Young diagrams and equivariant Schubert calculus, Trans. Amer. Math. Soc. 361 (2009), 5193--5221.

\bibitem{Jones}
B.F. Jones, Singular Chern classes of Schubert varieties via small resolution,  Int. Math. Res. Not. IMRN   (2010), 1371--1416.


\bibitem{KL}
G. Kempf and D. Laksov,  
The determinantal formula of Schubert calculus,
Acta Math. 132 (1974), 153--162.


\bibitem{KnMi}
A. Knutson and E. Miller, Gr\"obner geometry of Schubert polynomials, Ann. Math.
161 (2005), 1245--1318.

\bibitem{KJ2}
A. Knutson and P. Zinn-Justin, 
Schubert puzzles and integrability II: multiplying motivic Segre classes, 
arXiv:2102.00563, 2021.


\bibitem{VK}
V. Kreiman, Schubert classes in the equivariant K-theory and equivariant cohomology of the Grassmannian, arXiv:0512204.




\bibitem{Kumar}
S. Kumar,
Conjectural positivity of Chern--Schwartz--MacPherson classes for Richardson cells, arXiv:2208.03527.


\bibitem{LLS}
T. Lam, S. Lee and M. Shimozono, Back stable Schubert calculus, Compos. Math.  157 (2021), 883--962.

\bibitem{LS} 
A. Lascoux and M. Sch\"utzenberger, Polyn\^omes de Schubert, C. R. Math. Acad. Sci. Paris, S\'er. I Math.  294(1982), 447--450.


\bibitem{Lee}
S.J. Lee,
Chern class of Schubert cells in the ﬂag manifoldand related algebras, J. Algebraic Combin.  47 (2018), 213--231.

\bibitem{LRS}
C. Lenart, S. Robinson and  F. Sottile, Grothendieck polynomials via permutation patterns and chains in the Bruhat order,  Amer. J. Math.  128 (2006), 805--848.

\bibitem{LSY}
C. Li, V. Ravikumar, F. Sottile and M. Yang, A geometric proof of an equivariant Pieri rule for flag manifolds, Forum Math.  31 (2019), 779--783.

\bibitem{Liu}
R.I. Liu, 
Twisted Schubert polynomials,
 Selecta  Math. (2022), https://doi.org/10.1007/s00029-022-00802-1.

\bibitem{Lusztig}
G. Lusztig,  Affine Hecke algebras and their graded version, J. Amer. Math. Soc.  2 (1989), 599--599.

\bibitem{Ma}
I.G. Macdonald, Notes on Schubert Polynomials, Laboratoire de combinatoire et d'informatique math\'ematique (LACIM), Universit\'e du Qu\'ebec \'a Montr\'eal, Montreal, 1991.



\bibitem{MacPherson}
R. MacPherson,  Chern classes for singular algebraic varieties, Ann. Math.  100 (1974), 423--432.




\bibitem{Manivel}
L. Manivel, Symmetric Functions, Schubert Polynomials and Degeneracy Loci, Vol. 6, SMF/AMS Texts and Monographs, AMS, Providence, RI, 2001.
 
\bibitem{QGQC}
D. Maulik and A. Okounkov, Quantum groups and quantum cohomology, Ast\'erisque No. 408 (2019),  209 pp. 
 


\bibitem{MihalceaSlides}
L. Mihalcea, 
Positivity of Chern and Segre--MacPherson classes,
Special Session on Recent Advances in Schubert Calculus and Related Topics, AMS Brown, 2021.

\bibitem{MNS}
L. Mihalcea, H. Naruse  and C. Su, Hook formulae from Segre--MacPherson classes, arXiv:2203.16461v1.

\bibitem{MNS2}
L. Mihalcea, H. Naruse and C. Su, Left Demazure--Lusztig operators on equivariant (quantum) cohomology and K-theory, 
 Int. Math. Res. Not. IMRN (2022),  12096--12147.

 


\bibitem{MiSt}
E. Miller and B. Sturmfels, Combinatorial Commutative Algebra, Graduate Texts in Mathematics Vol. 227, Springer--Verlag, New York, 2004.





\bibitem{Molev}
A.I. Molev, Factorial supersymmetric Schur functions and super Capelli identities,  Proc. of the AMS-Kirillov's seminar on representation theory, pages 109--137, Providence, RI, 1998.


\bibitem{Pak-1}
A.H. Morales, I. Pak and G. Panova, Hook formulas for skew shapes I. q-analogues and bijections,
J. Combin. Theory, Ser. A 154 (2018), 350--405.

\bibitem{Pak-4}
A.H. Morales, I. Pak and G. Panova, Hook formulas for skew shapes IV. increasing tableaux and factorial Grothendieck polynomials, 
 J. Math. Sci. (N.Y.)  261  (2022), 630--657. 


\bibitem{MS2}
A. Morrison and F. Sottile, Two Murnaghan--Nakayama rules
in Schubert calculus, Ann. Combin.  22 (2018), 363--375.


\bibitem{Narusee}
H. Naruse, Schubert calculus and hook formula, talk slides at 73rd Sem. Lothar. Combin., Strobl, Austria, 2014,  available at https://www.emis.de/journals/SLC/wpapers/s73vortrag/naruse.pdf.

\bibitem{Ohmoto}
T. Ohmoto, Equivariant Chern classes of singular algebraic varieties with group actions, Math.
Proc. Camb. Philos. Soc.  140 (2006), 115--134.


\bibitem{PS}
A. Postnikov and  R.P. Stanley, Chains in the Bruhat order, J. Algebraic Combin. 29 (2009), 133--174. 

\bibitem{LectureK}
A. Okounkov, Lectures on K-theoretic computations in enumerative geometry, In Geometry of moduli spaces and representation theory, Vol. 24, IAS/Park City Math. Ser., pages 251--380, Amer. Math. Soc., Providence, RI, 2017.




\bibitem{RV18} R. Rim\'{a}nyi and A. Varchenko, Equivariant Chern--Schwartz--MacPherson classes in partial flag varieties: interpolation and formulae, in Schubert Varieties, Equivariant Cohomology and Characteristic Classes, IMPANGA2015 (eds. J. Buczynski, M. Michalek, E. Postingel), EMS 2018, pp. 225--235.

\bibitem{Robinson}
S. Robinson, A Pieri-type formula for $H^*_T(SL_n(\mathbb{C})/B)$,  J. Algebra  249 (2002), 38--58.

\bibitem{Sagan}
B. Sagan, The Symmetric Group: Representations, Combinatorial Algorithms, and Symmetric
Functions, 2nd edition, Springer-Verlag, New York, 2001.

\bibitem{SLAlg}
L. Serge, Algebra, Vol. 211, Springer Science \& Business Media, 2012.


 
\bibitem{Sottile1}
F. Sottile, Pieri’s formula for flag manifolds and Schubert polynomials, Ann. de l’Institut Fourier  46 (1996), 89--110.



\bibitem{Sch17}
J. Sch\"urmann, Chern classes and transversality for singular spaces, Singularities in Geometry,
Topology, Foliations and Dynamics (Cham) (Jos\'e Luis Cisneros-Molina, D\~ung Tr\'ang L\^e, Mutsuo Oka, and Jawad Snoussi, eds.), Springer International Publishing,  pages 207--231, 2017.


\bibitem{stanley2}
R.P. Stanley, Enumerative Combinatorics, Vol. 2, Cambridge Studies in Advanced Mathematics,
vol. 62, Cambridge University Press, Cambridge, 1999.

\bibitem{Stanley84}
R.P. Stanley, The stable behavior of some characters of $\mathrm{SL}(n, \mathbb{C})$, Lin. Multilin. Algebra 16 (1984), 29--34.

\bibitem{Su}
C. Su, Equivariant quantum cohomology of cotangent bundle of $G/P$,  Adv. Math.  289 (2016), 362--383.



\bibitem{Su2}
C. Su, Restriction formula for stable basis of the Springer resolution, Selecta  Math.  23 (2017), 497--518.


\bibitem{Su3}
C. Su, Structure constants for Chern classes of Schubert cells, 
Math. Z.  298 (2021),  193--213. 


\bibitem{Sullivan}
D. Sullivan, Combinatorial Invariants of Analytic Spaces, Proceedings of Liverpool
 Singularities Symposium I, Lecture Notes in Mathematics No. 192, Springer Verlag,
 1970, 165--168.




\end{thebibliography}
\end{document}